\documentclass[11pt]{article}

\usepackage{array}
\usepackage{calrsfs}
\usepackage{amsfonts}
\usepackage{amscd}
\usepackage{amssymb}
\usepackage{amsthm}
\usepackage{amsmath}
\usepackage{stmaryrd}
\usepackage{graphicx}
\usepackage{color}
\usepackage{verbatim}
\usepackage{mathrsfs}
\usepackage{tikz}
\usetikzlibrary{calc}
\usepackage{caption}
\usepackage[neveradjust]{paralist}

 \theoremstyle{plain}
 \newtheorem{thm1}{Theorem}
 
\newtheorem{thm}{Theorem}[section]
\newtheorem{lemma}[thm]{Lemma}
\newtheorem{prop}[thm]{Proposition}
\newtheorem{cor}[thm]{Corollary}

\theoremstyle{definition}
\newtheorem{defn}[thm]{Definition}
 \newtheorem*{defn1}{Definition}
\newtheorem{remark}[thm]{Remark}
\newtheorem{example}[thm]{Example}

\numberwithin{equation}{section}

\def\sA{\mathsf{A}}
\def\sB{\mathsf{B}}
\def\sC{\mathsf{C}}
\def\sD{\mathsf{D}}
\def\sE{\mathsf{E}}
\def\sF{\mathsf{F}}
\def\sG{\mathsf{G}}

\def\sX{\mathsf{X}}
\def\PGL{\mathsf{PGL}}

\def\cC{\mathcal{C}}

\def\cL{\mathcal{L}}

\def \cP{\mathcal{P}}
\def\cQ{\mathcal{Q}}

\def\cS{\mathcal{S}}

\newcommand{\PG}{\mathsf{PG}}
\def\K{\mathbb{K}}

\def\CC{\mathbb{C}}

\def\FF{\mathbb{F}}

\def\HH{\mathbb{H}}

\def\KK{\mathbb{K}}
\def\LL{\mathbb{L}}

\def\QQ{\mathbb{Q}}
\def\RR{\mathbb{R}}

\def\K{\mathbb{K}}
\def\PG{\mathsf{PG}}
\def\kar{\mathsf{char}}
\def\codim{\mathsf{codim}}
\def\Ga{\Gamma}

\DeclareMathOperator\Type{\mathsf{Typ}}
\DeclareMathOperator\Res{\mathsf{Res}}
\DeclareMathOperator\proj{\mathsf{proj}}

\DeclareMathOperator\Opp{\mathsf{Opp}}
\DeclareMathOperator\Diag{\mathsf{Diag}}

\DeclareMathOperator\tr{\mathsf{tr}}

\def\<{\langle}
\def\>{\rangle}

\makeatletter
\renewcommand{\@makefnmark}{\mbox{\textsuperscript{}}}
\makeatother

\title{Automorphisms and opposition in spherical buildings of classical type}
\author{James Parkinson 
\and
Hendrik Van Maldeghem}
\date{\today}

\begin{document}

\maketitle

\begin{abstract}
An automorphism of a spherical building is called domestic if it maps no chamber to an opposite chamber. In this paper we classify domestic automorphisms of spherical buildings of classical type.
\end{abstract}

%\tableofcontents

\section*{Introduction}

The study of the geometry of fixed elements of automorphisms of spherical buildings is a well-established and beautiful topic (see~\cite{PMW:15}). Recently a complementary theory concerning the ``opposite geometry'' $\Opp(\theta)$, consisting of those elements mapped to opposite elements by an automorphism, has been developed (see  \cite{PVM:19a,PVM:19b,PVM:21,PVM:23}). The generic situation is that this opposite geometry is rather large, and typically contains many chambers of the building. The more delicate and interesting situation is when the opposite geometry $\Opp(\theta)$ contains no chamber, in which case the automorphism $\theta$ is called \textit{domestic}. This paper is part of series of papers directed towards classifying, as precisely as possible, the class of domestic automorphisms of spherical buildings, with the focus of this paper being on the buildings of classical types.

Central to the study of the domesticity and the opposite geometry is the notion of the \textit{opposition diagram} $\Diag(\theta)$ of an automorphism~$\theta$, introduced in~\cite{PVM:19a,PVM:19b}. This diagram encodes the types of the simplices of the building that are mapped onto opposite simplices by~$\theta$, and one surprising aspect of the theory is that there are relatively few permitted opposition diagrams for a given spherical building. This opens up the possibility of classifying the class of domestic automorphisms with each diagram. In~\cite{PVM:21} and \cite{PVM:23} we carried out this program for split spherical buildings of exceptional type (with the exception of the three opposition diagrams $\sE_{7;3}$, $\sE_{7;4}$, and $\sE_{8;4}$ that will be dealt with in future work), and for all Moufang hexagons. In the present paper we carry out the program for spherical buildings of classical type (both split and non-split) of rank at least~$3$.

Before stating our main results, we briefly expand on some of the above concepts (see Section~\ref{sec:background} for remaining concepts). Let $\Delta$ be a thick spherical building of irreducible type $(W,S)$ with Coxeter diagram $\Gamma$, and let $\theta$ be an automorphism of $\Delta$. Let $\pi\in\mathrm{Aut}(\Gamma)$ denote the automorphism of $\Gamma$ induced by $\theta$ (thus $\pi=\mathrm{id}$ if and only if $\theta$ is type preserving). Let $\pi_0\in\mathrm{Aut}(\Gamma)$ be the automorphism of $\Gamma$ induced by the longest element~$w_0$. The \textit{distinguished orbits} are the nontrivial subsets of $S$ that are minimal subject to being preserved by $\pi_0\circ\pi$. In particular, the distinguished orbits are all singletons if and only if $\pi_0\circ\pi=\mathrm{id}$. The \textit{opposition diagram} $\Diag(\theta)$ of $\theta$ is drawn by encircling those distinguished orbits $J\subseteq S$ of $\Gamma$ such that there exists a type $J$ simplex mapped onto an opposite simplex. The opposition diagram $\Diag(\theta)$ is drawn ``straight'' if $\pi_0\circ\pi=\mathrm{id}$, and ``bent'' in the usual way otherwise. 

Thus, for example, the opposition diagrams
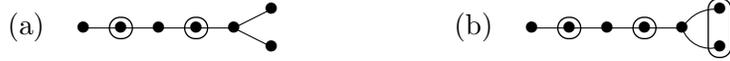
\begin{center}
(a)\quad \begin{tikzpicture}[scale=0.5,baseline=-0.5ex]
\node at (1,0.8) {};
%\node [inner sep=0.8pt,outer sep=0.8pt] at (-5,0) (-5) {$\bullet$};
%\node [inner sep=0.8pt,outer sep=0.8pt] at (-4,0) (-4) {$\bullet$};
%\node [inner sep=0.8pt,outer sep=0.8pt] at (-3,0) (-3) {$\bullet$};
%\node [inner sep=0.8pt,outer sep=0.8pt] at (-2,0) (-2) {$\bullet$};
%\node [inner sep=0.8pt,outer sep=0.8pt] at (-1,0) (-1) {$\bullet$};
%\node at (0,0) (0) {$\bullet$};
\node [inner sep=0.8pt,outer sep=0.8pt] at (0,0) (0) {$\bullet$};
\node [inner sep=0.8pt,outer sep=0.8pt] at (1,0) (1) {$\bullet$};
\node [inner sep=0.8pt,outer sep=0.8pt] at (2,0) (2) {$\bullet$};
\node [inner sep=0.8pt,outer sep=0.8pt] at (3,0) (3) {$\bullet$};
\node [inner sep=0.8pt,outer sep=0.8pt] at (4,0) (4) {$\bullet$};
\node [inner sep=0.8pt,outer sep=0.8pt] at (5,0.5) (5a) {$\bullet$};
\node [inner sep=0.8pt,outer sep=0.8pt] at (5,-0.5) (5b) {$\bullet$};
\draw (0,0)--(4,0);
\draw (4,0) to   (5,0.5);
\draw (4,0) to   (5,-0.5);
%\draw [line width=0.5pt,line cap=round,rounded corners] (-4.north west)  rectangle (-4.south east);
%\draw [line width=0.5pt,line cap=round,rounded corners] (-2.north west)  rectangle (-2.south east);
\draw [line width=0.5pt,line cap=round,rounded corners] (1.north west)  rectangle (1.south east);
\draw [line width=0.5pt,line cap=round,rounded corners] (3.north west)  rectangle (3.south east);
%\node [below] at (2,-0.25) {$2i$};
\end{tikzpicture}\qquad\qquad\qquad (b)\quad \begin{tikzpicture}[scale=0.5,baseline=-0.5ex]
\node at (0,0.8) {};
\node at (0,-0.8) {};
%\node [inner sep=0.8pt,outer sep=0.8pt] at (-5,0) (-5) {$\bullet$};
%\node [inner sep=0.8pt,outer sep=0.8pt] at (-4,0) (-4) {$\bullet$};
%\node [inner sep=0.8pt,outer sep=0.8pt] at (-3,0) (-3) {$\bullet$};
%\node [inner sep=0.8pt,outer sep=0.8pt] at (-2,0) (-2) {$\bullet$};
%\node [inner sep=0.8pt,outer sep=0.8pt] at (-1,0) (-1) {$\bullet$};
%\node at (0,0) (0) {$\bullet$};
\node [inner sep=0.8pt,outer sep=0.8pt] at (0,0) (0) {$\bullet$};
\node [inner sep=0.8pt,outer sep=0.8pt] at (1,0) (1) {$\bullet$};
\node [inner sep=0.8pt,outer sep=0.8pt] at (2,0) (2) {$\bullet$};
\node [inner sep=0.8pt,outer sep=0.8pt] at (3,0) (3) {$\bullet$};
\node [inner sep=0.8pt,outer sep=0.8pt] at (4,0) (4) {$\bullet$};
\node [inner sep=0.8pt,outer sep=0.8pt] at (5,0.5) (5a) {$\bullet$};
\node [inner sep=0.8pt,outer sep=0.8pt] at (5,-0.5) (5b) {$\bullet$};
%\draw (-5,0)--(-0.5,0);
\draw (0,0)--(4,0);
\draw (4,0) to [bend left] (5,0.5);
\draw (4,0) to [bend right=45] (5,-0.5);
%\draw [style=dashed] (-1,0)--(1,0);
%\draw [line width=0.5pt,line cap=round,rounded corners] (-4.north west)  rectangle (-4.south east);
%\draw [line width=0.5pt,line cap=round,rounded corners] (-2.north west)  rectangle (-2.south east);
\draw [line width=0.5pt,line cap=round,rounded corners] (1.north west)  rectangle (1.south east);
\draw [line width=0.5pt,line cap=round,rounded corners] (3.north west)  rectangle (3.south east);
\draw [line width=0.5pt,line cap=round,rounded corners] (5a.north west)  rectangle (5b.south east);
%\node [below] at (2,-0.25) {$2i$};
\end{tikzpicture}
\end{center}
represent (a) a non-type preserving automorphism of a $\sD_7$ building mapping type $2$ and type $4$ vertices to opposite vertices, and (b) a type preserving automorphism of a $\sD_7$ building mapping type $2$ and $4$ vertices onto opposite, and type $\{6,7\}$ simplices onto opposite simplices (recall that the opposition relation $\pi_0$ on $\sD_n$ is type preserving if $n$ is even, and interchanges types $n-1$ and $n$ if $n$ is odd).

In \cite{PVM:19a,PVM:19b} we showed that if $\theta$ is an automorphism of a thick irreducible spherical building then $\Diag(\theta)$ belongs to a very restricted list of possible diagrams. These diagrams are called the \textit{admissible diagrams}, and are defined by simple combinatorial rules. We will not require the formal definition of admissible diagrams, instead the explicit list of admissible diagrams of classical type given in Tables~\ref{table:1} and~\ref{table:2} is sufficient for the purposes of this paper (see \cite{PVM:19a} for the formal details).

\begin{center}
\noindent\begin{tabular}{|l|l|l|}
\hline
Symbol&Diagram&Conditions \\
\hline\hline
\begin{tabular}{l}
$^{2}\sA_{n;i}^1$
\end{tabular}
&
\begin{tabular}{l}
\begin{tikzpicture}[scale=0.5,baseline=-0.5ex]
%\node at (0,0.8) {};
%\node at (0,-0.8) {};
%\node at (0,0) (0) {};
\node [inner sep=0.8pt,outer sep=0.8pt] at (1,0.5) (1a) {$\bullet$};
\node [inner sep=0.8pt,outer sep=0.8pt] at (1,-0.5) (1b) {$\bullet$};
\node [inner sep=0.8pt,outer sep=0.8pt] at (2,0.5) (2a) {$\bullet$};
\node [inner sep=0.8pt,outer sep=0.8pt] at (2,-0.5) (2b) {$\bullet$};
\node [inner sep=0.8pt,outer sep=0.8pt]  at (3,0.5) (3a) {$\bullet$};
\node [inner sep=0.8pt,outer sep=0.8pt] at (3,-0.5) (3b) {$\bullet$};
\node [inner sep=0.8pt,outer sep=0.8pt] at (5,0.5) (4a) {$\bullet$};
\node [inner sep=0.8pt,outer sep=0.8pt] at (5,-0.5) (4b) {$\bullet$};
\node [inner sep=0.8pt,outer sep=0.8pt] at (6,0.5) (5a) {$\bullet$};
\node [inner sep=0.8pt,outer sep=0.8pt] at (6,-0.5) (5b) {$\bullet$};
\node [inner sep=0.8pt,outer sep=0.8pt] at (7,0.5) (6a) {$\bullet$};
\node [inner sep=0.8pt,outer sep=0.8pt] at (7,-0.5) (6b) {$\bullet$};
%\draw [dashed] (7,0.5) to [bend left=45] (8,0);
%\draw [dashed] (7,-0.5) to [bend right=45] (8,0);
\draw (1,0.5)--(3,0.5);
\draw (1,-0.5)--(3,-0.5);
\draw [dashed] (3,0.5)--(5,0.5);
\draw [dashed] (3,-0.5)--(5,-0.5);
\draw (5,0.5)--(7,0.5);
\draw (5,-0.5)--(7,-0.5);
\draw [line width=0.5pt,line cap=round,rounded corners] (1a.north west)  rectangle (1b.south east);
\draw [line width=0.5pt,line cap=round,rounded corners] (2a.north west)  rectangle (2b.south east);
\draw [line width=0.5pt,line cap=round,rounded corners] (3a.north west)  rectangle (3b.south east);
\draw [line width=0.5pt,line cap=round,rounded corners] (4a.north west)  rectangle (4b.south east);
\node at (5,0) {$i$};
\draw [domain=-90:90, dashed] plot ({7+(0.7)*cos(\x)}, {(0.5)*sin(\x)});
\end{tikzpicture}
\end{tabular}&$0\leq i\leq n/2$\\
\hline
\begin{tabular}{l}
$^1\mathsf{A}_{n;(n-1)/2}^2$
\end{tabular}&
\begin{tabular}{l}
\begin{tikzpicture}[scale=0.5,baseline=-0.5ex]
\node at (0,0.3) {};
\node at (0,-0.3) {};
\node [inner sep=0.8pt,outer sep=0.8pt] at (-5,0) (-5) {$\bullet$};
\node [inner sep=0.8pt,outer sep=0.8pt] at (-4,0) (-4) {$\bullet$};
\node [inner sep=0.8pt,outer sep=0.8pt] at (-3,0) (-3) {$\bullet$};
\node [inner sep=0.8pt,outer sep=0.8pt] at (-2,0) (-2) {$\bullet$};
\node [inner sep=0.8pt,outer sep=0.8pt] at (-1,0) (-1) {$\bullet$};
\node [inner sep=0.8pt,outer sep=0.8pt] at (1,0) (1) {$\bullet$};
\node [inner sep=0.8pt,outer sep=0.8pt] at (2,0) (2) {$\bullet$};
\node [inner sep=0.8pt,outer sep=0.8pt] at (3,0) (3) {$\bullet$};
\node [inner sep=0.8pt,outer sep=0.8pt] at (4,0) (4) {$\bullet$};
\node [inner sep=0.8pt,outer sep=0.8pt] at (5,0) (5) {$\bullet$};
\draw (-5,0)--(-0.5,0);
\draw (0.5,0)--(5,0);
\draw [style=dashed] (-1,0)--(1,0);
\draw [line width=0.5pt,line cap=round,rounded corners] (-4.north west)  rectangle (-4.south east);
\draw [line width=0.5pt,line cap=round,rounded corners] (-2.north west)  rectangle (-2.south east);
\draw [line width=0.5pt,line cap=round,rounded corners] (2.north west)  rectangle (2.south east);
\draw [line width=0.5pt,line cap=round,rounded corners] (4.north west)  rectangle (4.south east);
\phantom{\draw [line width=0.5pt,line cap=round,rounded corners] (-5.north west)  rectangle (-5.south east);}
\end{tikzpicture}
\end{tabular}&$n$ odd\\
\hline
\begin{tabular}{l}
$^1\mathsf{A}_{n;n}^1$
\end{tabular}&
\begin{tabular}{l}
\begin{tikzpicture}[scale=0.5,baseline=-0.5ex]
\node at (0,0.3) {};
\node at (0,-0.3) {};
\node [inner sep=0.8pt,outer sep=0.8pt] at (-5,0) (-5) {$\bullet$};
\node [inner sep=0.8pt,outer sep=0.8pt] at (-4,0) (-4) {$\bullet$};
\node [inner sep=0.8pt,outer sep=0.8pt] at (-3,0) (-3) {$\bullet$};
\node [inner sep=0.8pt,outer sep=0.8pt] at (-2,0) (-2) {$\bullet$};
\node [inner sep=0.8pt,outer sep=0.8pt] at (-1,0) (-1) {$\bullet$};
\node [inner sep=0.8pt,outer sep=0.8pt] at (1,0) (1) {$\bullet$};
\node [inner sep=0.8pt,outer sep=0.8pt] at (2,0) (2) {$\bullet$};
\node [inner sep=0.8pt,outer sep=0.8pt] at (3,0) (3) {$\bullet$};
\node [inner sep=0.8pt,outer sep=0.8pt] at (4,0) (4) {$\bullet$};
\node [inner sep=0.8pt,outer sep=0.8pt] at (5,0) (5) {$\bullet$};
\draw (-5,0)--(-0.5,0);
\draw (0.5,0)--(5,0);
\draw [style=dashed] (-1,0)--(1,0);
\draw [line width=0.5pt,line cap=round,rounded corners] (-4.north west)  rectangle (-4.south east);
\draw [line width=0.5pt,line cap=round,rounded corners] (-2.north west)  rectangle (-2.south east);
\draw [line width=0.5pt,line cap=round,rounded corners] (2.north west)  rectangle (2.south east);
\draw [line width=0.5pt,line cap=round,rounded corners] (4.north west)  rectangle (4.south east);
\draw [line width=0.5pt,line cap=round,rounded corners] (-3.north west)  rectangle (-3.south east);
\draw [line width=0.5pt,line cap=round,rounded corners] (-1.north west)  rectangle (-1.south east);
\draw [line width=0.5pt,line cap=round,rounded corners] (1.north west)  rectangle (1.south east);
\draw [line width=0.5pt,line cap=round,rounded corners] (3.north west)  rectangle (3.south east);
\draw [line width=0.5pt,line cap=round,rounded corners] (-5.north west)  rectangle (-5.south east);
\draw [line width=0.5pt,line cap=round,rounded corners] (5.north west)  rectangle (5.south east);
\end{tikzpicture}
\end{tabular}& \\
\hline
\begin{tabular}{l}
${^1}\sB_{n;i}^1$ or ${^1}\sC_{n;i}^1$
\end{tabular}&
\begin{tabular}{l}
\begin{tikzpicture}[scale=0.5,baseline=-0.5ex]
\node at (0,0.5) {};
\node [inner sep=0.8pt,outer sep=0.8pt] at (-5,0) (-5) {$\bullet$};
\node [inner sep=0.8pt,outer sep=0.8pt] at (-4,0) (-4) {$\bullet$};
\node [inner sep=0.8pt,outer sep=0.8pt] at (-3,0) (-3) {$\bullet$};
\node [inner sep=0.8pt,outer sep=0.8pt] at (-2,0) (-2) {$\bullet$};
\node [inner sep=0.8pt,outer sep=0.8pt] at (-1,0) (-1) {$\bullet$};
\node [inner sep=0.8pt,outer sep=0.8pt] at (1,0) (1) {$\bullet$};
\node [inner sep=0.8pt,outer sep=0.8pt] at (2,0) (2) {$\bullet$};
\node [inner sep=0.8pt,outer sep=0.8pt] at (3,0) (3) {$\bullet$};
\node [inner sep=0.8pt,outer sep=0.8pt] at (4,0) (4) {$\bullet$};
\node [inner sep=0.8pt,outer sep=0.8pt] at (5,0) (5) {$\bullet$};
\node at (1,-0.7) {$i$};
\draw (-5,0)--(-0.5,0);
\draw (0.5,0)--(4,0);
\draw (4,0.07)--(5,0.07);
\draw (4,-0.07)--(5,-0.07);
\draw [style=dashed] (-1,0)--(1,0);
%\draw [line width=0.5pt,line cap=round,rounded corners] (-5.north west)  rectangle (1.south east);
\draw [line width=0.5pt,line cap=round,rounded corners] (-5.north west)  rectangle (-5.south east);
\draw [line width=0.5pt,line cap=round,rounded corners] (-4.north west)  rectangle (-4.south east);
\draw [line width=0.5pt,line cap=round,rounded corners] (-3.north west)  rectangle (-3.south east);
\draw [line width=0.5pt,line cap=round,rounded corners] (-2.north west)  rectangle (-2.south east);
\draw [line width=0.5pt,line cap=round,rounded corners] (-1.north west)  rectangle (-1.south east);
\draw [line width=0.5pt,line cap=round,rounded corners] (1.north west)  rectangle (1.south east);
\end{tikzpicture}
\end{tabular}
&$0\leq i\leq n$\\
\hline
\begin{tabular}{l}
${^1}\sB_{n;i}^2$ or ${^1}\sC_{n;i}^2$
\end{tabular}&
\begin{tabular}{l}
\begin{tikzpicture}[scale=0.5,baseline=-0.5ex]
\node at (0,0.3) {};
\node [inner sep=0.8pt,outer sep=0.8pt] at (-5,0) (-5) {$\bullet$};
\node [inner sep=0.8pt,outer sep=0.8pt] at (-4,0) (-4) {$\bullet$};
\node [inner sep=0.8pt,outer sep=0.8pt] at (-3,0) (-3) {$\bullet$};
\node [inner sep=0.8pt,outer sep=0.8pt] at (-2,0) (-2) {$\bullet$};
\node [inner sep=0.8pt,outer sep=0.8pt] at (-1,0) (-1) {$\bullet$};
\node [inner sep=0.8pt,outer sep=0.8pt] at (1,0) (1) {$\bullet$};
\node [inner sep=0.8pt,outer sep=0.8pt] at (2,0) (2) {$\bullet$};
\node [inner sep=0.8pt,outer sep=0.8pt] at (3,0) (3) {$\bullet$};
\node [inner sep=0.8pt,outer sep=0.8pt] at (4,0) (4) {$\bullet$};
\node [inner sep=0.8pt,outer sep=0.8pt] at (5,0) (5) {$\bullet$};
\draw (-5,0)--(-0.5,0);
\draw (0.5,0)--(4,0);
\draw (4,0.07)--(5,0.07);
\draw (4,-0.07)--(5,-0.07);
\node at (2,-0.7) {$2i$};
\draw [style=dashed] (-1,0)--(1,0);
\draw [line width=0.5pt,line cap=round,rounded corners] (-4.north west)  rectangle (-4.south east);
\draw [line width=0.5pt,line cap=round,rounded corners] (-2.north west)  rectangle (-2.south east);
\draw [line width=0.5pt,line cap=round,rounded corners] (2.north west)  rectangle (2.south east);
\phantom{\draw [line width=0.5pt,line cap=round,rounded corners] (-5.north west)  rectangle (-5.south east);}
\end{tikzpicture}
\end{tabular}&$0\leq i\leq n/2$\\
\hline
\begin{tabular}{l}
${^1}\sD_{n;i}^1$
\end{tabular}&
\begin{tabular}{l}
\begin{tikzpicture}[scale=0.5,baseline=-0.5ex]
\node at (0,0.8) {};
\node [inner sep=0.8pt,outer sep=0.8pt] at (-5,0) (-5) {$\bullet$};
\node [inner sep=0.8pt,outer sep=0.8pt] at (-4,0) (-4) {$\bullet$};
\node [inner sep=0.8pt,outer sep=0.8pt] at (-3,0) (-3) {$\bullet$};
\node [inner sep=0.8pt,outer sep=0.8pt] at (-2,0) (-2) {$\bullet$};
\node [inner sep=0.8pt,outer sep=0.8pt] at (-1,0) (-1) {$\bullet$};
%\node at (0,0) (0) {$\bullet$};
\node [inner sep=0.8pt,outer sep=0.8pt] at (1,0) (1) {$\bullet$};
\node [inner sep=0.8pt,outer sep=0.8pt] at (2,0) (2) {$\bullet$};
\node [inner sep=0.8pt,outer sep=0.8pt] at (3,0) (3) {$\bullet$};
\node [inner sep=0.8pt,outer sep=0.8pt] at (4,0) (4) {$\bullet$};
\node [inner sep=0.8pt,outer sep=0.8pt] at (5,0.5) (5a) {$\bullet$};
\node [inner sep=0.8pt,outer sep=0.8pt] at (5,-0.5) (5b) {$\bullet$};
\draw (-5,0)--(-0.5,0);
\draw (0.5,0)--(4,0);
\draw (4,0) to (5,0.5);
\draw (4,0) to (5,-0.5);
\draw [style=dashed] (-1,0)--(1,0);
%\draw [line width=0.5pt,line cap=round,rounded corners] (-5.north west)  rectangle (2.south east);
\draw [line width=0.5pt,line cap=round,rounded corners] (-5.north west)  rectangle (-5.south east);
\draw [line width=0.5pt,line cap=round,rounded corners] (-4.north west)  rectangle (-4.south east);
\draw [line width=0.5pt,line cap=round,rounded corners] (-3.north west)  rectangle (-3.south east);
\draw [line width=0.5pt,line cap=round,rounded corners] (-2.north west)  rectangle (-2.south east);
\draw [line width=0.5pt,line cap=round,rounded corners] (-1.north west)  rectangle (-1.south east);
\draw [line width=0.5pt,line cap=round,rounded corners] (1.north west)  rectangle (1.south east);
\draw [line width=0.5pt,line cap=round,rounded corners] (2.north west)  rectangle (2.south east);
\node [below] at (2,-0.25) {$i$};
\end{tikzpicture}
\end{tabular} &
$0\leq i<n$ and $i= n\mod 2$\\
\begin{tabular}{l}
${^1}\sD_{n;n}^1$
\end{tabular}&
\begin{tabular}{l}
\begin{tikzpicture}[scale=0.5,baseline=-0.5ex]
\node at (0,0.8) {};
\node at (0,-0.8) {};
\node [inner sep=0.8pt,outer sep=0.8pt] at (-5,0) (-5) {$\bullet$};
\node [inner sep=0.8pt,outer sep=0.8pt] at (-4,0) (-4) {$\bullet$};
\node [inner sep=0.8pt,outer sep=0.8pt] at (-3,0) (-3) {$\bullet$};
\node [inner sep=0.8pt,outer sep=0.8pt] at (-2,0) (-2) {$\bullet$};
\node [inner sep=0.8pt,outer sep=0.8pt] at (-1,0) (-1) {$\bullet$};
%\node at (0,0) (0) {$\bullet$};
\node [inner sep=0.8pt,outer sep=0.8pt] at (1,0) (1) {$\bullet$};
\node [inner sep=0.8pt,outer sep=0.8pt] at (2,0) (2) {$\bullet$};
\node [inner sep=0.8pt,outer sep=0.8pt] at (3,0) (3) {$\bullet$};
\node [inner sep=0.8pt,outer sep=0.8pt] at (4,0) (4) {$\bullet$};
\node [inner sep=0.8pt,outer sep=0.8pt] at (5,0.5) (5a) {$\bullet$};
\node [inner sep=0.8pt,outer sep=0.8pt] at (5,-0.5) (5b) {$\bullet$};
\draw (-5,0)--(-0.5,0);
\draw (0.5,0)--(4,0);
\draw (4,0) to  (5,0.5);
\draw (4,0) to  (5,-0.5);
\draw [style=dashed] (-1,0)--(1,0);
\draw [line width=0.5pt,line cap=round,rounded corners] (-4.north west)  rectangle (-4.south east);
\draw [line width=0.5pt,line cap=round,rounded corners] (-2.north west)  rectangle (-2.south east);
\draw [line width=0.5pt,line cap=round,rounded corners] (2.north west)  rectangle (2.south east);
\draw [line width=0.5pt,line cap=round,rounded corners] (4.north west)  rectangle (4.south east);
\draw [line width=0.5pt,line cap=round,rounded corners] (5a.north west)  rectangle (5a.south east);
\draw [line width=0.5pt,line cap=round,rounded corners] (-5.north west)  rectangle (-5.south east);
\draw [line width=0.5pt,line cap=round,rounded corners] (-3.north west)  rectangle (-3.south east);
\draw [line width=0.5pt,line cap=round,rounded corners] (-1.north west)  rectangle (-1.south east);
\draw [line width=0.5pt,line cap=round,rounded corners] (1.north west)  rectangle (1.south east);
\draw [line width=0.5pt,line cap=round,rounded corners] (3.north west)  rectangle (3.south east);
\draw [line width=0.5pt,line cap=round,rounded corners] (5b.north west)  rectangle (5b.south east);
%\node [below] at (2,-0.25) {$2i$};
\end{tikzpicture}
\end{tabular} &
\\
\hline
\begin{tabular}{l}
${^2}\sD_{n;i}^1$
\end{tabular}&
\begin{tabular}{l}
\begin{tikzpicture}[scale=0.5,baseline=-0.5ex]
\node at (0,0.8) {};
\node [inner sep=0.8pt,outer sep=0.8pt] at (-5,0) (-5) {$\bullet$};
\node [inner sep=0.8pt,outer sep=0.8pt] at (-4,0) (-4) {$\bullet$};
\node [inner sep=0.8pt,outer sep=0.8pt] at (-3,0) (-3) {$\bullet$};
\node [inner sep=0.8pt,outer sep=0.8pt] at (-2,0) (-2) {$\bullet$};
\node [inner sep=0.8pt,outer sep=0.8pt] at (-1,0) (-1) {$\bullet$};
%\node at (0,0) (0) {$\bullet$};
\node [inner sep=0.8pt,outer sep=0.8pt] at (1,0) (1) {$\bullet$};
\node [inner sep=0.8pt,outer sep=0.8pt] at (2,0) (2) {$\bullet$};
\node [inner sep=0.8pt,outer sep=0.8pt] at (3,0) (3) {$\bullet$};
\node [inner sep=0.8pt,outer sep=0.8pt] at (4,0) (4) {$\bullet$};
\node [inner sep=0.8pt,outer sep=0.8pt] at (5,0.5) (5a) {$\bullet$};
\node [inner sep=0.8pt,outer sep=0.8pt] at (5,-0.5) (5b) {$\bullet$};
\draw (-5,0)--(-0.5,0);
\draw (0.5,0)--(4,0);
\draw (4,0) to [bend left] (5,0.5);
\draw (4,0) to [bend right=45] (5,-0.5);
\draw [style=dashed] (-1,0)--(1,0);
%\draw [line width=0.5pt,line cap=round,rounded corners] (-5.north west)  rectangle (2.south east);
\draw [line width=0.5pt,line cap=round,rounded corners] (-5.north west)  rectangle (-5.south east);
\draw [line width=0.5pt,line cap=round,rounded corners] (-4.north west)  rectangle (-4.south east);
\draw [line width=0.5pt,line cap=round,rounded corners] (-3.north west)  rectangle (-3.south east);
\draw [line width=0.5pt,line cap=round,rounded corners] (-2.north west)  rectangle (-2.south east);
\draw [line width=0.5pt,line cap=round,rounded corners] (-1.north west)  rectangle (-1.south east);
\draw [line width=0.5pt,line cap=round,rounded corners] (1.north west)  rectangle (1.south east);
\draw [line width=0.5pt,line cap=round,rounded corners] (2.north west)  rectangle (2.south east);
\node [below] at (2,-0.25) {$i$};
\end{tikzpicture}
\end{tabular} &
$0\leq i< n-1$ and $i=n+1\mod 2$\\
\begin{tabular}{l}
${^2}\sD_{n;n-1}^1$
\end{tabular}&
\begin{tabular}{l}
\begin{tikzpicture}[scale=0.5,baseline=-0.5ex]
\node at (0,0.8) {};
\node at (0,-0.8) {};
\node [inner sep=0.8pt,outer sep=0.8pt] at (-5,0) (-5) {$\bullet$};
\node [inner sep=0.8pt,outer sep=0.8pt] at (-4,0) (-4) {$\bullet$};
\node [inner sep=0.8pt,outer sep=0.8pt] at (-3,0) (-3) {$\bullet$};
\node [inner sep=0.8pt,outer sep=0.8pt] at (-2,0) (-2) {$\bullet$};
\node [inner sep=0.8pt,outer sep=0.8pt] at (-1,0) (-1) {$\bullet$};
%\node at (0,0) (0) {$\bullet$};
\node [inner sep=0.8pt,outer sep=0.8pt] at (1,0) (1) {$\bullet$};
\node [inner sep=0.8pt,outer sep=0.8pt] at (2,0) (2) {$\bullet$};
\node [inner sep=0.8pt,outer sep=0.8pt] at (3,0) (3) {$\bullet$};
\node [inner sep=0.8pt,outer sep=0.8pt] at (4,0) (4) {$\bullet$};
\node [inner sep=0.8pt,outer sep=0.8pt] at (5,0.5) (5a) {$\bullet$};
\node [inner sep=0.8pt,outer sep=0.8pt] at (5,-0.5) (5b) {$\bullet$};
\draw (-5,0)--(-0.5,0);
\draw (0.5,0)--(4,0);
\draw (4,0) to [bend left] (5,0.5);
\draw (4,0) to [bend right=45] (5,-0.5);
\draw [style=dashed] (-1,0)--(1,0);
\draw [line width=0.5pt,line cap=round,rounded corners] (-4.north west)  rectangle (-4.south east);
\draw [line width=0.5pt,line cap=round,rounded corners] (-2.north west)  rectangle (-2.south east);
\draw [line width=0.5pt,line cap=round,rounded corners] (2.north west)  rectangle (2.south east);
\draw [line width=0.5pt,line cap=round,rounded corners] (4.north west)  rectangle (4.south east);
\draw [line width=0.5pt,line cap=round,rounded corners] (-5.north west)  rectangle (-5.south east);
\draw [line width=0.5pt,line cap=round,rounded corners] (-3.north west)  rectangle (-3.south east);
\draw [line width=0.5pt,line cap=round,rounded corners] (-1.north west)  rectangle (-1.south east);
\draw [line width=0.5pt,line cap=round,rounded corners] (1.north west)  rectangle (1.south east);
\draw [line width=0.5pt,line cap=round,rounded corners] (3.north west)  rectangle (3.south east);
\draw [line width=0.5pt,line cap=round,rounded corners] (5a.north west)  rectangle (5b.south east);
%\node [below] at (2,-0.25) {$2i$};
\end{tikzpicture}
\end{tabular} &
\\
\hline
\begin{tabular}{l}
${^1}\sD_{n;i}^2$
\end{tabular}&
\begin{tabular}{l}
\begin{tikzpicture}[scale=0.5,baseline=-0.5ex]
\node at (0,0.8) {};
\node [inner sep=0.8pt,outer sep=0.8pt] at (-5,0) (-5) {$\bullet$};
\node [inner sep=0.8pt,outer sep=0.8pt] at (-4,0) (-4) {$\bullet$};
\node [inner sep=0.8pt,outer sep=0.8pt] at (-3,0) (-3) {$\bullet$};
\node [inner sep=0.8pt,outer sep=0.8pt] at (-2,0) (-2) {$\bullet$};
\node [inner sep=0.8pt,outer sep=0.8pt] at (-1,0) (-1) {$\bullet$};
%\node at (0,0) (0) {$\bullet$};
\node [inner sep=0.8pt,outer sep=0.8pt] at (1,0) (1) {$\bullet$};
\node [inner sep=0.8pt,outer sep=0.8pt] at (2,0) (2) {$\bullet$};
\node [inner sep=0.8pt,outer sep=0.8pt] at (3,0) (3) {$\bullet$};
\node [inner sep=0.8pt,outer sep=0.8pt] at (4,0) (4) {$\bullet$};
\node [inner sep=0.8pt,outer sep=0.8pt] at (5,0.5) (5a) {$\bullet$};
\node [inner sep=0.8pt,outer sep=0.8pt] at (5,-0.5) (5b) {$\bullet$};
\draw (-5,0)--(-0.5,0);
\draw (0.5,0)--(4,0);
\draw (4,0) to   (5,0.5);
\draw (4,0) to   (5,-0.5);
\draw [style=dashed] (-1,0)--(1,0);
\draw [line width=0.5pt,line cap=round,rounded corners] (-4.north west)  rectangle (-4.south east);
\draw [line width=0.5pt,line cap=round,rounded corners] (-2.north west)  rectangle (-2.south east);
\draw [line width=0.5pt,line cap=round,rounded corners] (2.north west)  rectangle (2.south east);
\phantom{\draw [line width=0.5pt,line cap=round,rounded corners] (-5.north west)  rectangle (-5.south east);}
\node [below] at (2,-0.25) {$2i$};
\end{tikzpicture}
\end{tabular} &
$n$ even and $0\leq i< n/2$\\
\begin{tabular}{l}
${^1}\sD_{n;n/2}^2$
\end{tabular}&
\begin{tabular}{l}
\begin{tikzpicture}[scale=0.5,baseline=-0.5ex]
\node at (0,0.8) {};
\node at (0,-0.8) {};
\node [inner sep=0.8pt,outer sep=0.8pt] at (-5,0) (-5) {$\bullet$};
\node [inner sep=0.8pt,outer sep=0.8pt] at (-4,0) (-4) {$\bullet$};
\node [inner sep=0.8pt,outer sep=0.8pt] at (-3,0) (-3) {$\bullet$};
\node [inner sep=0.8pt,outer sep=0.8pt] at (-2,0) (-2) {$\bullet$};
\node [inner sep=0.8pt,outer sep=0.8pt] at (-1,0) (-1) {$\bullet$};
%\node at (0,0) (0) {$\bullet$};
\node [inner sep=0.8pt,outer sep=0.8pt] at (1,0) (1) {$\bullet$};
\node [inner sep=0.8pt,outer sep=0.8pt] at (2,0) (2) {$\bullet$};
\node [inner sep=0.8pt,outer sep=0.8pt] at (3,0) (3) {$\bullet$};
\node [inner sep=0.8pt,outer sep=0.8pt] at (4,0) (4) {$\bullet$};
\node [inner sep=0.8pt,outer sep=0.8pt] at (5,0.5) (5a) {$\bullet$};
\node [inner sep=0.8pt,outer sep=0.8pt] at (5,-0.5) (5b) {$\bullet$};
\draw (-5,0)--(-0.5,0);
\draw (0.5,0)--(4,0);
\draw (4,0) to   (5,0.5);
\draw (4,0) to   (5,-0.5);
\draw [style=dashed] (-1,0)--(1,0);
\draw [line width=0.5pt,line cap=round,rounded corners] (-4.north west)  rectangle (-4.south east);
\draw [line width=0.5pt,line cap=round,rounded corners] (-2.north west)  rectangle (-2.south east);
\draw [line width=0.5pt,line cap=round,rounded corners] (2.north west)  rectangle (2.south east);
\draw [line width=0.5pt,line cap=round,rounded corners] (4.north west)  rectangle (4.south east);
\draw [line width=0.5pt,line cap=round,rounded corners] (5b.north west)  rectangle (5b.south east);
\phantom{\draw [line width=0.5pt,line cap=round,rounded corners] (-5.north west)  rectangle (-5.south east);}
%\node [below] at (2,-0.25) {$2i$};
\end{tikzpicture}
\end{tabular} &
$n$ even\\
\hline
\begin{tabular}{l}
${^2}\sD_{n;i}^2$
\end{tabular}&
\begin{tabular}{l}
\begin{tikzpicture}[scale=0.5,baseline=-0.5ex]
\node at (0,0.8) {};
\node at (0,-0.8) {};
\node [inner sep=0.8pt,outer sep=0.8pt] at (-5,0) (-5) {$\bullet$};
\node [inner sep=0.8pt,outer sep=0.8pt] at (-4,0) (-4) {$\bullet$};
\node [inner sep=0.8pt,outer sep=0.8pt] at (-3,0) (-3) {$\bullet$};
\node [inner sep=0.8pt,outer sep=0.8pt] at (-2,0) (-2) {$\bullet$};
\node [inner sep=0.8pt,outer sep=0.8pt] at (-1,0) (-1) {$\bullet$};
%\node at (0,0) (0) {$\bullet$};
\node [inner sep=0.8pt,outer sep=0.8pt] at (1,0) (1) {$\bullet$};
\node [inner sep=0.8pt,outer sep=0.8pt] at (2,0) (2) {$\bullet$};
\node [inner sep=0.8pt,outer sep=0.8pt] at (3,0) (3) {$\bullet$};
\node [inner sep=0.8pt,outer sep=0.8pt] at (4,0) (4) {$\bullet$};
\node [inner sep=0.8pt,outer sep=0.8pt] at (5,0.5) (5a) {$\bullet$};
\node [inner sep=0.8pt,outer sep=0.8pt] at (5,-0.5) (5b) {$\bullet$};
\draw (-5,0)--(-0.5,0);
\draw (0.5,0)--(4,0);
\draw (4,0) to [bend left] (5,0.5);
\draw (4,0) to [bend right=45] (5,-0.5);
\draw [style=dashed] (-1,0)--(1,0);
\draw [line width=0.5pt,line cap=round,rounded corners] (-4.north west)  rectangle (-4.south east);
\draw [line width=0.5pt,line cap=round,rounded corners] (-2.north west)  rectangle (-2.south east);
\draw [line width=0.5pt,line cap=round,rounded corners] (2.north west)  rectangle (2.south east);
\node [below] at (2,-0.25) {$2i$};
\end{tikzpicture}
\end{tabular} &
$n$ odd and $0\leq i<(n-1)/2$\\
\begin{tabular}{l}
${^2}\sD_{n;(n-1)/2}^2$
\end{tabular}&
\begin{tabular}{l}
\begin{tikzpicture}[scale=0.5,baseline=-0.5ex]
\node at (0,0.8) {};
\node at (0,-0.8) {};
\node [inner sep=0.8pt,outer sep=0.8pt] at (-5,0) (-5) {$\bullet$};
\node [inner sep=0.8pt,outer sep=0.8pt] at (-4,0) (-4) {$\bullet$};
\node [inner sep=0.8pt,outer sep=0.8pt] at (-3,0) (-3) {$\bullet$};
\node [inner sep=0.8pt,outer sep=0.8pt] at (-2,0) (-2) {$\bullet$};
\node [inner sep=0.8pt,outer sep=0.8pt] at (-1,0) (-1) {$\bullet$};
%\node at (0,0) (0) {$\bullet$};
\node [inner sep=0.8pt,outer sep=0.8pt] at (1,0) (1) {$\bullet$};
\node [inner sep=0.8pt,outer sep=0.8pt] at (2,0) (2) {$\bullet$};
\node [inner sep=0.8pt,outer sep=0.8pt] at (3,0) (3) {$\bullet$};
\node [inner sep=0.8pt,outer sep=0.8pt] at (4,0) (4) {$\bullet$};
\node [inner sep=0.8pt,outer sep=0.8pt] at (5,0.5) (5a) {$\bullet$};
\node [inner sep=0.8pt,outer sep=0.8pt] at (5,-0.5) (5b) {$\bullet$};
\draw (-5,0)--(-0.5,0);
\draw (0.5,0)--(4,0);
\draw (4,0) to [bend left] (5,0.5);
\draw (4,0) to [bend right=45] (5,-0.5);
\draw [style=dashed] (-1,0)--(1,0);
\draw [line width=0.5pt,line cap=round,rounded corners] (-4.north west)  rectangle (-4.south east);
\draw [line width=0.5pt,line cap=round,rounded corners] (-2.north west)  rectangle (-2.south east);
\draw [line width=0.5pt,line cap=round,rounded corners] (1.north west)  rectangle (1.south east);
\draw [line width=0.5pt,line cap=round,rounded corners] (3.north west)  rectangle (3.south east);
\draw [line width=0.5pt,line cap=round,rounded corners] (5a.north west)  rectangle (5b.south east);
%\node [below] at (2,-0.25) {$2i$};
\end{tikzpicture}
\end{tabular} &
$n$ odd\\
\hline
\end{tabular}
\captionof{table}{Admissible diagrams of classical type, general rank}\label{table:1}
\end{center}

\begin{center}
\noindent\begin{tabular}{|l|l||l|l||l|l|l|l}
\hline
Symbol&Diagram&Symbol&Diagram&Symbol&Diagram\\
\hline\hline
\begin{tabular}{l}
${^2}\sB_{2;1}^1$ or ${^2}\sC_{2;1}^1$
\end{tabular}&
\hspace{0.25cm}\begin{tabular}{l}
\begin{tikzpicture}[scale=0.5,baseline=-0.5ex]
\node at (7,0.8) {};
\node at (7,-0.8) {};
\node at (7,0) (0) {};
%\node [inner sep=0.8pt,outer sep=0.8pt] at (6,0.5) (5a) {$\bullet$};
%\node [inner sep=0.8pt,outer sep=0.8pt] at (6,-0.5) (5b) {$\bullet$};
\node [inner sep=0.8pt,outer sep=0.8pt] at (7,0.5) (6a) {$\bullet$};
\node [inner sep=0.8pt,outer sep=0.8pt] at (7,-0.5) (6b) {$\bullet$};
%\draw (7,0.55) to [bend left=45] (7.7,0);
%\draw (7,-0.55) to [bend right=45] (7.7,0);
%\draw (7,0.45) to [bend left=45] (7.6,0);
%\draw (7,-0.45) to [bend right=45] (7.6,0);
%\draw (6,0.5)--(7,0.5);
%\draw (6,-0.5)--(7,-0.5);
\draw [line width=0.5pt,line cap=round,rounded corners] (6a.north west)  rectangle (6b.south east);
%\draw [line width=0.5pt,line cap=round,rounded corners] (5a.north west)  rectangle (5b.south east);
\draw [domain=-90:90] plot ({7+(0.6)*cos(\x)}, {(0.45)*sin(\x)});
\draw [domain=-90:90] plot ({7+(0.7)*cos(\x)}, {(0.55)*sin(\x)});
%\node at (5,0) {$i$};
\end{tikzpicture} 
\end{tabular}&
${^3}\sD_{4;1}^2$&\hspace{0.25cm}\begin{tikzpicture}[scale=0.5,baseline=-0.5ex]
\node at (0,0.8) {};
\node at (0,-0.8) {};
\node [inner sep=0.8pt,outer sep=0.8pt] at (1,0) (-1) {$\bullet$};
\node [inner sep=0.8pt,outer sep=0.8pt] at (0,0) (0) {$\bullet$};
\node [inner sep=0.8pt,outer sep=0.8pt] at (1,0.5) (1a) {$\bullet$};
\node [inner sep=0.8pt,outer sep=0.8pt] at (1,-0.5) (1b) {$\bullet$};
\draw (0,0)--(1,0);
\draw (0,0) to [bend left=45] (1,0.5);
\draw (0,0) to [bend right=45] (1,-0.5);
\draw [line width=0.5pt,line cap=round,rounded corners] (1a.north west)  rectangle (1b.south east);
\end{tikzpicture} & ${^3}\sD_{4;2}^1$ &\hspace{0.25cm} \begin{tikzpicture}[scale=0.5,baseline=-0.5ex]
\node at (0,0.8) {};
\node at (0,-0.8) {};
\node [inner sep=0.8pt,outer sep=0.8pt] at (1,0) (-1) {$\bullet$};
\node [inner sep=0.8pt,outer sep=0.8pt] at (0,0) (0) {$\bullet$};
\node [inner sep=0.8pt,outer sep=0.8pt] at (1,0.5) (1a) {$\bullet$};
\node [inner sep=0.8pt,outer sep=0.8pt] at (1,-0.5) (1b) {$\bullet$};
\draw (0,0)--(1,0);
\draw (0,0) to [bend left=45] (1,0.5);
\draw (0,0) to [bend right=45] (1,-0.5);
\draw [line width=0.5pt,line cap=round,rounded corners] (1a.north west)  rectangle (1b.south east);
\draw [line width=0.5pt,line cap=round,rounded corners] (0.north west)  rectangle (0.south east);
\end{tikzpicture} \\
\hline
\end{tabular}
\captionof{table}{Admissible diagrams of classical type, special rank}\label{table:2}
\end{center}

As in the tables, each admissible diagram of classical type is denoted by a symbol
$
{^t}\sX_{n;i}^j
$
where 
\begin{compactenum}[$(1)$]
\item $\sX\in\{\sA,\sB,\sC,\sD\}$ is the Coxeter type of $\Gamma$, and $n$ is the rank;
\item $t\in\{1,2,3\}$ is the order of the automorphism $\pi_0\circ\pi$ of $\Gamma$ (the ``twisting'');
\item $i$ is the number of distinguished orbits encircled in the diagram; and
\item $j\in\{1,2\}$ is the (graph) distance between successive encircled distinguished orbits in $J$. 
\end{compactenum}

Domestic automorphisms of generalised polygons (rank~$2$ buildings) are studied in~\cite{TTM:12b,PTM:15}, and domestic trialities of building of type $\sD_4$ are studied in~\cite{HVM:13}, and the diagrams in Table~\ref{table:2} are not discussed further in this paper. For the diagrams ${^t}\sX_{n;i}^j$ in Table~\ref{table:1} it turns out that the twisting index $t\in\{1,2\}$ is determined by the indices $n,i,j$ (and~$\sX$). This is clear if $\sX\in\{\sA,\sB,\sC\}$, and in type $\sD$ note that $t\in\{1,2\}$ is determined by the condition $t=n+ij+1\mod 2$. With this in mind, it is usually convenient to omit the twisting index $t$ from the notation ${^t}\sX_{n;i}^j$ without risk of confusion. 

An automorphism $\theta$ is called \textit{capped} if whenever there exist both type $J_1$ and $J_2$ simplices in $\Opp(\theta)$, then there exists a type $J_1\cup J_2$ simplex in $\Opp(\theta)$. In \cite[Theorem~1]{PVM:19a} we proved that every automorphism of a large spherical building is capped (a thick spherical building of rank at least $3$ is called \textit{large} if it contains no Fano plane residues, and \textit{small} otherwise). Since automorphisms of small buildings are studied in~\cite{PVM:19b} we shall, when required, restrict attention to large spherical buildings in this paper, and by doing so we may thus assume that all automorphisms are capped. In this case an automorphism is domestic if and only if its opposition diagram has at least one distinguished orbit not encircled. 

Domestic automorphisms of large buildings of type $\sA_n$ have been completely classified in previous work (see Theorem~\ref{thm1:An} of this paper for a summary). Thus the main focus of this paper is on buildings of types $\sB_n,\sC_n$ and $\sD_n$, where the situation is considerably more complicated. Such buildings may naturally be regarded in a standard and uniform way as \textit{polar spaces} $\Pi=(\cP,\Omega)$ of rank~$n$, where $\cP$ is the set of points of~$\Pi$ (the type~$1$ vertices of the building in standard Bourbaki labelling), and $\Omega$ is the set of singular subspaces of~$\Pi$ (see Section~\ref{sec:polarspace} for more details). The singular subspaces of (projective) dimension~$n-1$ are called the \textit{maximal singular subspaces}.  

Since we consider thick buildings, the associated polar spaces have thick lines (meaning that each line contains at least $3$ points). We adopt the convention that buildings of type $\sC_n$ are those corresponding to symplectic polar spaces (that is, split buildings of type $\sC_n$), buildings of type $\sD_n$ are those corresponding to the non-thick polar spaces (meaning that each singular $(n-2)$-space is contained in precisely $2$ maximal singular subspaces), and buildings associated to all other polar spaces are considered to be of type~$\sB_n$ (however note that this convention is somewhat non-standard as buildings of absolute type ${^2}\sA_n$ are often considered as ``type $\sC_n$'' elsewhere). Thus for buildings of type $\sB_n$ and $\sC_n$ the type~$i$ vertices ($1\leq i\leq n$) correspond to the singular subspaces of projective dimension $i-1$. This is also true for buildings of type $\sD_n$ for $1\leq i\leq n-2$, while the type $n-1$ and $n$ vertices of the building correspond to the two types of maximal singular subspaces. Points $x,y$ of a polar space are opposite one another if and only if they are distinct and not collinear. Singular subspaces $X,Y\in\Omega$ of a polar space are opposite one another if and only if they have the same dimension and for each point $x\in X$ there is a point $y\in Y$ with $x$ and $y$ opposite.

The automorphisms of spherical buildings with opposition diagrams in Table~\ref{table:1} are all collineations of the associated polar space $\Pi$ (however, note that in the $\sD_n$ case, the automorphism $\theta$ may interchange the type $n-1$ and $n$ vertices of the building). A collineation $\theta$ of a polar space is called \textit{point-domestic} if it maps no point (equivalently, no type $1$ vertex of the building) to an opposite. More generally, $\theta$ is called $i$-domestic ($0\leq i\leq n-1$) if it maps no singular subspace of projective dimension~$i$ to an opposite (hence point-domestic is equivalent to $0$-domestic). Domestic automorphisms of rank~$3$ polar spaces are classified in~\cite[Theorem~7.2]{TTM:12}, and so, when required, we may restrict our attention to rank at least~$4$, and moreover, from the comments above, we may restrict to large polar spaces. 

It is convenient to divide the class of domestic collineations of a polar space into three mutually disjoint classes, as follows.

\begin{defn1}
A domestic collineation $\theta$ of a polar space is said to be:
\begin{compactenum}[$(1)$]
\item of \textit{class} I if it is not point-domestic. 
\item of \textit{class} II if it is point-domestic, and has a fixed point.
\item of \textit{class} III if it is point-domestic, and has no fixed points. 
\end{compactenum}
\end{defn1}

Note that if $\theta$ is of class I, then by the classification of admissible diagrams, we have $\Diag(\theta)=\sX_{n;i}^1$ for some $1\leq i<n$ (with $i<n-1$ in the $\sX=\sD$ case). Similarly, if $\theta$ is of class II or class III then $\Diag(\theta)=\sX_{n;i}^2$ for some $0\leq i\leq n/2$ (where $\sX\in\{\sB,\sC,\sD\}$).

The main results of this paper are summarised as follows (for undefined notions, see Section~\ref{sec:background}). To begin with, consider class I collineations of a polar space $\Pi=(\cP,\Omega)$. We first show that such automorphisms are characterised by big fixed point sets (see Definition~\ref{defn:polarspacethings}).

\begin{thm1}\label{thm1:nonpointdom_1}
If $\theta$ is a class \emph{I} automorphism of a large polar space $\Pi$ with diagram $\Diag(\theta)=\sX_{n;i}^1$ (with $\sX\in\{\sB,\sC,\sD\}$) then $\theta$ fixes pointwise a subspace of corank~$i$. In particular, the fixed point set of $\theta$ is a possibly degenerate and possibly non-thick polar space (possibly with lines of size $2$). Conversely, if $\theta$ is a collineation of $\Pi$ mapping at least one point to an opposite and fixing pointwise a subspace of corank~$i$ and fixing no subspace of corank $j<i$, then $\Diag(\theta)=\sX_{n;i}^1$ (with $\sX\in\{\sB,\sC,\sD\}$).
\end{thm1}

On an intuitive level, Theorem~\ref{thm1:nonpointdom_1} says that domestic automorphisms that are not point-domestic are domestic for the following reason: by fixing a subspace of suitably high dimension it is forced that each singular subspace of suitably high dimension contains a fixed point, and thus is not mapped onto an opposite singular subspace.  

In certain cases we are able to be more precise. For example, we prove the following (see Definition~\ref{defn:polarspacethings} for the definition of elations).

\begin{thm1}\label{thm1:symplectic1}
Let $\Pi=\sC_{n,1}(\KK)$ be the (large) symplectic polar space over a field $\KK$ with $n\geq 3$ and let $\theta$ be a collineation of $\Pi$. Then:
\begin{compactenum}[$(1)$]
\item $\Diag(\theta)=\sC_{n;1}^1$ if and only if $\theta$ is a central elation (that is, a long root elation).
\item $\Diag(\theta)=\sC_{n;2}^1$ if and only if $\theta$ is either
\begin{compactenum}[$(a)$]
\item a product of two nontrivial perpendicular long root elations, or
\item a nontrivial member of the group generated by two opposite long root groups and not conjugate to a long root elation. 
\end{compactenum}
\end{compactenum}
\end{thm1}

In a similar way we also classify those collineations of thick non-embeddable polar spaces with diagram $\mathsf{C_{3;1}^1}$ (see Theorem~\ref{polarnonemb}) and those of split polar spaces with diagram $\sB_{n;1}^1$ and $\sD_{n;1}^1$ (see Proposition~\ref{polarquadric}, Theorem~\ref{polarDn} and Remark~\ref{remnonperfect} for the statements).

The situation for class II and class III automorphisms of polar spaces (that is, point-domestic collineations) is more complicated, but richer. Recall that the diagrams of such (nontrivial) automorphisms are of the form $\sX_{n;i}^2$ with $\sX\in\{\sB,\sC,\sD\}$ and $1\leq i\leq \lfloor n/2\rfloor$. The automorphisms whose diagrams have the fewest encircled nodes (that is $i=1$) admit a complete classification.

\begin{thm1}\label{thm1:base}
Let $\Pi=(\cP,\Omega)$ be a polar space of rank~$n\geq 3$ and let $\theta$ be a collineation of~$\Pi$.
\begin{compactenum}[$(1)$]
\item If $\Pi$ is of symplectic type then $\theta$ has opposition diagram $\sC_{n;1}^2$ if and only if $\theta$ is an axial collineation (in the case $\kar(\KK)= 2$) or a specific involutive homology (if $\kar(\KK)\neq 2$).
\item If $\Pi$ is not of symplectic type then $\theta$ has diagram $\sX_{n;1}^2$, $\sX\in\{\sB,\sD\}$, if and only if $\theta$ is an axial elation (and so $\Pi$ is an orthogonal polar space), or the rank is $3$ and $\theta$ is an ideal Baer collineation.
\end{compactenum}
\end{thm1}

In the split case (that is, those arising from Chevalley groups over commutative fields), it is possible to completely classify the point-domestic automorphisms with diagram not equal to one of the ``extreme diagrams'' $\sB_{n;n/2}^2$, $\sC_{n;n/2}^2$, $\sD_{n;n/2}^2$ (with $n$ even) or $\sD_{n;(n-1)/2}^2$ (with $n$ odd).

\begin{thm1}~\label{thm1:middle}
Let $\theta$ be a point-domestic collineation of a parabolic, symplectic (with characteristic not~$2$), or hyperbolic polar space $\Delta$ with opposition diagram $\mathsf{B}_{n;i}^2$, $1\leq i\leq (n-1)/2$, $\mathsf{C}_{n,i}^2$, $1\leq i\leq (n-1)/2$, or $\mathsf{D}_{n,i}^2$, $1\leq i\leq (n-2)/2$, respectively. Then $\theta$ is the product of $i$ pairwise orthogonal long root elations, an $(I_{2n-2i},-I_{2i})$-homology, or the product of $i$ pairwise orthogonal long root elations, respectively. The conclusion also holds for $\sC_{n,n/2}^2$ if it is assumed that $\theta$ fixes at least one point.
\end{thm1}

Remarkably, when the polar space of rank $n\geq 3$ is Hermitian, there is only one type of automorphisms of class II, and they have opposition diagram $\mathsf{B}_{n;m}^2$, with $2m\in\{n,n+1\}$ (for the precise statement, we refer to Proposition~\ref{hermpointfix}).

The most beautiful situation is a domestic automorphism that fixes no points (class III). These are very restricted, as illustrated by the following theorem. In particular, note that every class III automorphism must have an ``extreme'' opposition diagram $\sX_{n;n/2}^2$ with $n$ even and $\sX\in\{\sB,\sC,\sD\}$.

\begin{thm1}\label{thm1:BCD1}
Let $\Delta$ be a spherical building of type $\sB_n$, $\sC_n$, or $\sD_n$ with $n\geq 4$ and let $\Pi=(\cP,\Omega)$ be the associated polar space defined over the field $\KK$. Suppose there exists a point domestic collineation $\theta$ of $\Pi$ with no fixed points. Let $\cQ$ denote the set of fixed lines of $\theta$, and let $\Omega'$ denote the set of fixed singular subspaces. Then 
\begin{compactenum}[$(1)$]
\item $n$ is even,
\item $\Diag(\theta)=\sX_{n;n/2}^2$ with $\sX\in\{\sB,\sC,\sD\}$, and
\item the pair $\Pi'=(\cQ,\Omega')$ is a polar space defined over either a quadratic extension of $\KK$, or over a quaternion division algebra $\HH$ which is $2$-dimensional over $\KK$ (and $\KK$ is $2$-dimensional over the centre of $\HH$).
\end{compactenum}
\end{thm1}

In the split case we can be even more explicit.

\begin{thm1}\label{thm1:pointdomnofixedpoint}
Let $\Delta$ be a split polar space of rank $n\geq 4$ over a field $\KK$. Let $\theta$ be a point-domestic collineation of $\Delta$ with no fixed points. Then $n$ is even, and $\Delta$ is either a symplectic polar space (type $\sC_n$), or a hyperbolic polar space (type $\sD_n$), and $\theta$ has opposition diagram $\sC_{n,n/2}^2$ or $\sD_{n,n/2}^2$, respectively. Moreover, the fixed element structure of $\theta$ is either
\begin{compactenum}[$(1)$]
\item a symplectic polar space over a quadratic extension of $\KK$, or
\item a minimal Hermitian polar space, or a mixed polar space,
\end{compactenum}
respectively. In each case, the points of the fixed element structure are the lines of $\Delta$ that are fixed by~$\theta$.
\end{thm1}

We can be more precise in each case, as follows.

\begin{thm1}\label{thm1:moreprecise1}
Let $\Delta$ be a symplectic polar space of rank $n\geq 4$ defined over a field $\KK$. Let $\theta$ be a point-domestic collineation of $\Delta$ with no fixed points. Let $\LL$ be the quadratic extension of $\KK$ over which the fixed element structure of $\theta$ is defined (according to Theorem~\ref{thm1:pointdomnofixedpoint}). Then $\theta$ is an involution, and if $\kar(\KK)=2$ then $\LL$ is an inseparable extension of $\KK$. Conversely, if $\KK$ admits a quadratic extension $\LL/\KK$, inseparable in the case $\kar(\KK)=2$, then $\Delta$ admits a point-domestic collineation $\theta$ without fixed points as above. 
\end{thm1}

\begin{thm1}\label{thm1:moreprecise2}
Let $\Delta$ be a hyperbolic polar space of rank $n\geq 4$ defined over a field $\KK$. Let $\theta$ be a point-domestic collineation of $\Delta$ with no fixed points. Let $\PG(2n-1,\KK)$ be the ambient projective space. Then $\theta$ naturally extends to $\PG(2n-1,\KK)$ pointwise fixing a spread $\mathcal{S}$ which defines a projective space $\PG(n-1,\LL)$ over a quadratic extension $\LL$ of $\KK$, and every collineation of $\PG(2n-1,\KK)$ which pointwise fixes $\mathcal{S}$ is a point-domestic collineation of $\Delta$ without fixed points. Moreover:
\begin{compactenum}[$(1)$]
\item If $\LL/\KK$ is separable then the fixed element polar space (c.f. Theorem~\emph{\ref{thm1:pointdomnofixedpoint}}) is minimal Hermitian.
\item If $\LL/\KK$ is inseparable then the fixed element polar space is the subspace of the symplectic polar space over $\LL$ obtained by restricting the long root elations to~$\KK$. 
\end{compactenum}
Conversely, if $\KK$ admits a quadratic extension $\LL/\KK$ then there exists a point-domestic collineation of $\Delta$ without fixed points as above.
\end{thm1}

For non-split polar spaces the situation depends highly on the underlying field and the form defining the polar space. The neatest case is as follows.

\begin{thm1}\label{thm1:nonsplit}
Let $\Delta$ be a minimal Hermitian polar space of rank $n\geq 4$ defined over a field $\KK$, with corresponding field involution $\sigma$ and field extension $\KK/\FF$. Let $\theta$ be a point-domestic collineation of $\Delta$ with no fixed points. Let $\PG(2n-1,\KK)$ be the ambient projective space. Then $\theta$ is an involution and naturally extends to $\PG(2n-1,\KK)$ pointwise fixing a spread $\mathcal{S}$ which defines a projective space $\PG(n-1,\HH)$ over a quaternion division algebra $\HH$ over its centre $\FF$, with $\KK$ a $2$-dimensional subalgebra of $\HH$. The only nontrivial collineation of $\Delta$ fixing every member of $\mathcal{S}$ in $\Delta$ is $\theta$. The fixed element polar space is a minimal quaternion polar space related to a non-standard involution of $\HH$ pointwise fixing a $3$-dimensional subalgebra over $\FF$. 

Conversely, if $\KK$ admits a quaternion extension $\HH$ such that the centre $\FF$ of $\HH$ is a subfield of $\KK$ and $\KK/\FF$ is a separable quadratic field extension, then there exists a point-domestic collineation of $\Delta$ without fixed points as above. 
\end{thm1}

The remaining cases are orthogonal and Hermitian polar spaces whose standard form has nontrivial anisotropic forms. Since these contain the hyperbolic and minimal Hermitian polar spaces fixed under the given point-domestic collineation, the previous two theorems give necessary conditions for the existence of collineations of class III. We provide examples, and some instances where such collineations do not exist (see Section~\ref{sec:specialfields}).

%
%We provide the following summary results (incorporating results from this paper, and previous work). First, on the existence (and non-existence) of some admissible diagrams. 
%
%
%\begin{cor1}
%In the split case, all diagrams can be obtained.
%\end{cor1}
%
%\begin{proof}
%Follows from\james{cite}
%\end{proof}
%
%\begin{thm1}
%In the non-split case, sometimes not (examples). 
%\end{thm1}
%
%We conclude by noting what remains. 
%\begin{compactenum}[$(1)$]
%\item For split buildings of exceptional type, the classification of domestic automorphisms of $\sE_6(\KK)$, $\sE_7(\KK)$ and $\sE_8(\KK)$ fixing no chamber remains (this is an analogue of class III automorphisms of polar spaces, and we expect tight classifications are possible). 
%\item The classification of domestic automorphisms of non-split buildings of type $\sF_4$
%\item While in the rank~$2$ case considerable information is known, the following remain feasible: full classification for Moufang hexagons (currently only known for the split Cayley hexagon), and Moufang octagons.  
%\end{compactenum}

In this paper we shall also classify those admissible diagrams that can be obtained as the opposition diagram of a \textit{unipotent element} of a split spherical building (that is, an element conjugate to an element of $U^+$), extending the corresponding result \cite[Theorem~5]{PVM:21} for buildings of exceptional type. In \cite{PVM:21} we introduced a combinatorial notion of an admissible diagram being \textit{polar closed}. See Section~\ref{sec:polarclosed} for the definition, however for now we note that in the classical case all admissible diagrams other than $\sA_{n;(n-1)/2}^2$ ($n$ odd), $\sA_{n;n}^1$, $\sB_{n;i}^1$ or $\sD_{n;i}^1$ (with $i$ odd and $1\leq i<n$), $\sC_{n;i}^2$ (with $1\leq i\leq n/2$), and those in Table~\ref{table:2} turn out to be polar closed. The classification of admissible diagrams arising as the opposition diagram of a unipotent element is then as follows.

\begin{thm1}\label{thm1:polarclosed}
Let $\Delta$ be a split irreducible spherical building with Dynkin diagram $\Ga$. The admissible Dynkin diagrams of type $\Ga$ that can be obtained as opposition diagrams of a unipotent element are precisely the polar closed diagrams. 
\end{thm1}

In particular, Theorem~\ref{thm1:polarclosed} shows that all polar closed admissible diagrams (of classical type) occur as the opposition diagram of some automorphism of the respective split building (over any field). We will prove the following theorem, completing the proof of \cite[Corollary~10]{PVM:21}, showing that the list of admissible diagrams contains no redundancies. 

\begin{thm1}\label{thm1:attained}
Let $\Delta$ be a split spherical building of irreducible type $\sX_n$ (with the underlying field assumed to be perfect in the case of $\sB_n$, $\sC_n$ or $\sF_4$ in characteristic $2$ and $\sG_2$ in characteristic~$3$). Every admissible diagram with underlying Dynkin diagram~$\sX_n$ arises as the opposition diagram of some automorphism of $\Delta$. 
\end{thm1}

We conclude this introduction with a summary of the structure of this paper. Section~\ref{sec:background} provides background on spherical buildings, domestic automorphisms, and polar spaces. The main theorems are proved in Sections~\ref{sec2}--\ref{sec:polarclosed1}. More precisely, Theorem~\ref{thm1:nonpointdom_1} is proved at the beginning of Section~\ref{sec2}, Theorem~\ref{thm1:symplectic1} follows from Theorems~\ref{polarsymplectic} and~\ref{pocopo}, and Theorem~\ref{thm1:base} follows from Theorem~\ref{copolar} and Proposition~\ref{axialbase}. Theorem~\ref{thm1:middle} follows from Theorems~\ref{BCD} and~\ref{symplectichomology}, Theorem~\ref{thm1:BCD1} follows from Corollaries~\ref{oddrankchamber}, \ref{oppdiafull}, and~\ref{corclassIII}. Theorems~\ref{thm1:pointdomnofixedpoint}, \ref{thm1:moreprecise1}, \ref{thm1:moreprecise2} and~\ref{thm1:nonsplit} are proved in Theorem~\ref{splitcasewithoutfixedpoints}, Proposition~\ref{morepprecise1a}, Proposition~\ref{hyperbolicpointdomestic}, and Proposition~\ref{hermitianpointdomestic}, respectively. Theorems~\ref{thm1:polarclosed} and~\ref{thm1:attained} are proved in Section~\ref{sec:polarclosed1}.

\section{Background and definitions}\label{sec:background}

This section contains background on Coxeter groups, spherical buildings, domestic automorphisms, opposition diagrams, Chevalley groups, projective spaces, and polar spaces. 

\subsection{Coxeter groups and spherical buildings}

Let $(W,S)$ be a spherical Coxeter system with Coxeter diagram $\Ga=\Ga(W,S)$. For $J\subseteq S$ let $W_J$ be the parabolic subgroup generated by~$J$, and let $\Ga_J=\Ga(W_J,J)$ be the associated subgraph of~$\Ga$. For each $J\subseteq S$ let $w_J$ be the longest element of $W_J$, and write $w_0=w_S$. For each $J\subseteq S$ the element $w_J$ induces an automorphism $\pi_J$ of $\Ga_J$ by $\pi_J(s)=w_Jsw_J^{-1}$ for $s\in J$. Write $\pi_0=\pi_S$.

Let $\Delta$ be a thick spherical building of type $(W,S)$, regarded as a simplicial complex, with chamber set $\cC=\cC(\Delta)$ and $W$-distance function~$\delta:\cC\times\cC\to W$. Let $\tau:\Delta\to 2^S$ be a fixed type map on the simplicial complex~$\Delta$ (we adopt Bourbaki~\cite{Bou:02} conventions for the indexing of the generators of spherical Coxeter systems). Our main references for the theory of buildings are \cite{AB:08,Tit:74}, and we assume that the reader is already acquainted with the theory. 

Chambers $A,B\in\cC$ are \textit{opposite} if they are at maximum distance in the chamber graph, or equivalently if $\delta(A,B)=w_0$. Simplices $\alpha,\beta$ of $\Delta$ are \textit{opposite} if $\tau(\beta)=\pi_0(\tau(\alpha))$ and there exists a chamber $A$ containing $\alpha$ and a chamber $B$ containing $\beta$ such that $A$ and $B$ are opposite.

The \textit{residue} of a simplex $\alpha$ of $\Delta$ is the set of all simplices of $\Delta$ which contain $\alpha$, together with the order relation induced by that on~$\Delta$. Then $\Res(\alpha)$ is a building whose Coxeter diagram is obtained from the Coxeter diagram of $\Delta$ by removing all nodes which belong to $\tau(\alpha)$. 

Let $\alpha$ be a simplex of $\Delta$. The \textit{projection onto $\alpha$} is the map $\proj_{\alpha}:\Delta\to\Res(\alpha)$ defined as follows (see~\cite[Section~3]{Tit:74}). For $\beta$ a simplex of $\Delta$, we set $\proj_{\alpha}(\beta)$ to be the unique simplex $\gamma$ of $\Res(\alpha)$ which is maximal subject to the property that every minimal length gallery from a chamber of $\Res(\beta)$ to $\Res(\alpha)$ ends in a chamber containing~$\gamma$.

We call the thick irreducible spherical buildings of rank at least $3$ with no Fano plane residues \textit{large buildings}, and those containing at least one Fano plane residue are called \textit{small buildings}.

\subsection{Automorphisms and opposition diagrams}\label{sec:automorphisms}

An \textit{automorphism} of $\Delta$ is a simplicial complex automorphism $\theta:\Delta\to\Delta$. Note that $\theta$ does not necessarily preserve types. Indeed, each automorphism $\theta:\Delta\to\Delta$ induces an automorphism $\pi_{\theta}$ of $\Ga$ by $\delta(A,B)=s$ if and only if $\delta(A^{\theta},B^{\theta})=\pi_{\theta}(s)$.

Let $\theta$ be an automorphism of $\Delta$. The \textit{opposite geometry} of $\theta$ is 
$$
\Opp(\theta)=\{\alpha\in\Delta\mid \alpha\text{ is opposite }\alpha^{\theta}\},
$$
and the \textit{type} $\Type(\theta)$ of $\theta$ is the union of all subsets $J\subseteq S$ such that there is a type $J$ simplex in $\Opp(\theta)$. 

The following fundamental theorem, due to Leeb~\cite{Lee:00} and Abramenko and Brown~\cite{AB:09}, shows that $\Opp(\theta)$ is empty if and only if $\theta$ is the identity. 

\begin{thm}[{\cite{Lee:00,AB:09}}]\label{thm:nonempty} Every nontrivial automorphism of a thick spherical buildings maps some simplex onto an opposite simplex. 
\end{thm}

The \textit{opposition diagram} $\Diag(\theta)$ of $\theta$ is the triple $(\Ga,\Type(\theta),\pi_{\theta})$, where $\Ga$ is the Coxeter diagram of $\Delta$. Less formally, the opposition diagram of $\theta$ is depicted by drawing $\Ga$ and encircling the nodes of $\Type(\theta)$, where we encircle nodes in minimal subsets invariant under $\pi_0\circ \pi_{\theta}$. We draw the diagram ``bent'' (in the standard way) if $\pi_0\circ\pi_{\theta}\neq1$ (see the examples in the introduction). An opposition diagram is \textit{empty} if no nodes are encircled (that is, $\Type(\theta)=\emptyset$), and \textit{full} if all nodes are encircled (that is, $\Type(\theta)=S$). Note that $\Diag(\theta)$ is empty if and only if $\theta$ is the identity (by Theorem~\ref{thm:nonempty}).

\begin{defn}
Let $\Delta$ be a spherical building of type $(W,S)$. Let $\theta$ be a nontrivial automorphism of $\Delta$, and let $J\subseteq S$ be stable under the diagram automorphism $\pi_0\circ\pi_{\theta}$. Then $\theta$ is called:
\begin{compactenum}[$(1)$]
\item \textit{capped} if there exists a type $\Type(\theta)$ simplex in $\Opp(\theta)$, and \textit{uncapped} otherwise. 
\item \textit{domestic} if $\Opp(\theta)$ contains no chamber.
\item \text{$J$-domestic} if $\Opp(\theta)$ contains no type $J$-simplex. 
\end{compactenum}
A simplex $\alpha\in\Delta$ is said to be \textit{domestic for $\theta$} if it is not mapped onto an opposite simplex by $\theta$, and is \textit{non-domestic for $\theta$} otherwise. If the automorphism $\theta$ is clear from context, we will often simply say $\alpha$ is domestic (or non-domestic). 
\end{defn}

The requirement that $J$ be stable under $\pi_0\circ \pi_{\theta}$ in the above definition is to avoid trivialities. For if $J$ is not stable under $\pi_0\circ\pi_{\theta}$ then necessarily there is no type $J$-simplex mapped to an opposite by $\theta$ (see \cite[Lemma~1.3]{PVM:19a}).

The main result of \cite{PVM:19a} is the following useful fact.

\begin{thm}[{\cite[Theorem~1]{PVM:19a}}]\label{thm:capped} Every automorphism of a large spherical building is capped. 
\end{thm}

Let $\theta$ be an automorphism of $\Delta$, and suppose that $\alpha\in\Opp(\theta)$. It follows from \cite[Theorem~3.28]{Tit:74} that the map $\proj_{\alpha}:\Res(\alpha^{\theta})\to\Res(\alpha)$ is an isomorphism, and hence we have an isomorphism
$$
\theta_{\alpha}:\Res(\alpha)\xrightarrow{\sim} \Res(\alpha)\quad \text{given by}\quad \theta_{\alpha}=\proj_{\alpha}\circ\,\theta.
$$
If $\pi_{\theta}$ is the automorphism of $\Ga$ induced by $\theta$, and if $J=\tau(\alpha)$, then the automorphism of $\Ga_{S\backslash J}$ induced by $\theta_{\alpha}$ is 
$
\pi_{\theta_{\alpha}}=\pi_{S\backslash J}\circ \pi_0\circ \pi_{\theta}
$
(see \cite[Proposition~1.11]{PVM:19a}). There is a very useful relationship between domesticity of $\theta$ and domesticity of $\theta_{\alpha}$, as follows (in particular, this relationship facilitates inductive arguments). 

\begin{lemma}[{\cite[Proposition~1.13]{PVM:19a}}]
Let $\theta$ be an automorphism of a spherical building~$\Delta$ and let $\alpha\in\Opp(\theta)$. If $\beta\in\Res(\alpha)$ then $\beta$ is opposite $\beta^{\theta}$ in the building $\Delta$ if and only if $\beta$ is opposite $\beta^{\theta_{\alpha}}$ in the building $\Res(\alpha)$. 
\end{lemma}

An admissible diagram $\sX=(\Ga,J,\pi)$ is called \textit{type preserving} if $\pi=\mathrm{id}$ (these are the opposition diagrams of type preserving automorphisms). For example, the type preserving diagrams of type~$\sD$ are precisely the diagrams  $\sD_{n;i}^1$ with $n$ and $i$ both even, $\sD_{n;i}^1$ with $n$ odd and $i$ even, and all diagrams $\sD_{n,i}^2$. 

\subsection{Root systems and Chevalley groups}

The \textit{split} irreducible spherical buildings are those arising from Chevalley groups over a commutative field~$\mathbb{K}$. In this case there is an associated irreducible crystallographic root system~$\Phi$. Let $\alpha_1,\ldots,\alpha_n$ be a fixed system of simple roots, and let $\Phi^+$ denote the associated positive roots. Adopting standard notation (from \cite{Car:89,St:16}) we write $x_{\alpha}(a)$ ($\alpha\in\Phi$ and $a\in \mathbb{K}$) for the unipotent generators of the Chevalley group~$G=G_{\Phi}(\mathbb{K})$. Write $U_{\alpha}=\langle x_{\alpha}(a)\mid a\in\mathbb{K}\rangle$ for the root group of~$\alpha\in\Phi$, and 
$$U=\langle U_{\alpha}\mid \alpha\in\Phi^+\rangle.$$
By a \textit{unipotent element} we shall mean an element of $G$ conjugate to a member of~$U$.

Let $\Gamma=\Gamma(\Phi)$ be the Dynkin diagram (with the convention that the arrow points towards the short root in the case of double or triple bonds). The \textit{height} of a root $\alpha=k_1\alpha_1+\cdots+k_n\alpha_n$ is $\mathrm{ht}(\alpha)=k_1+\cdots+k_n$, and there is a unique root $\varphi$ of maximal height (the \textit{highest root of $\Phi$}). The \textit{polar type} of $\Phi$ is the subset $\wp\subseteq\{1,2,\ldots,n\}$ given by 
$$
\wp=\{1\leq i\leq n\mid \langle\alpha_i,\varphi\rangle\neq 0\}.
$$
Thus, in standard Bourbaki labelling, the polar types of $\sA_n,\sB_n,\sC_n,\sD_n$ are $\{1,n\},\{2\},\{1\},\{2\}$, respectively (a simple way to remember these is to note that the polar type is the set of nodes of $\Gamma$ to which the additional generator is joined when constructing the associated affine diagram). In the case that $\wp=\{p\}$ is a singleton set, we refer to $p$ as the \textit{polar node}.

The \textit{dual polar type} of a Dynkin diagram is the subset $\wp'$ corresponding to the polar node of the dual diagram. Thus $\wp'=\wp$ in the $\sA_n$ and $\sD_n$ cases, while $\wp'=\{1\}$ for $\sB_n$, and $\wp'=\{2\}$ for $\sC_n$. If $\wp'=\{p'\}$ we refer to $p'$ as the dual polar node.

\subsection{Polar closed admissible diagrams}\label{sec:polarclosed}

The definition of admissible diagrams naturally extends to admissible Dynkin diagrams (with the additional condition that $\pi$ preserves arrows on the Dynkin diagram). Let $\sX$ be an admissible Dynkin diagram, and let $\sX=\sX_0,\sX_1,\ldots$ be subdiagrams such that, for $j\geq 1$, the diagram $\sX_j$ is obtained from $\sX_{j-1}$ by removing an encircled polar type from one of the connected components of $\sX_{j-1}$. This process terminates at some step~$j=k$ (that is, $\sX_k$ has no polar types encircled). We say that $\sX$ is \textit{polar closed} if $\sX_k$ is an empty diagram (that is, has no nodes encircled). 

For example, the following sequence $\sX_0,\sX_1,\sX_2,\sX_3,\sX_4$ shows that $\sB_{5;4}^1$ is polar closed
$$
\begin{tikzpicture}[scale=0.5,baseline=-0.5ex]
%\node at (0,0.8) {};
%\node at (0,-0.8) {};
\node [inner sep=0.8pt,outer sep=0.8pt] at (-2.5,0) (0) {$\bullet$};
\node [inner sep=0.8pt,outer sep=0.8pt] at (-1.5,0) (1) {$\bullet$};
\node [inner sep=0.8pt,outer sep=0.8pt] at (-0.5,0) (2) {$\bullet$};
\node [inner sep=0.8pt,outer sep=0.8pt] at (0.5,0) (3) {$\bullet$};
\node [inner sep=0.8pt,outer sep=0.8pt] at (1.5,0) (4) {$\bullet$};
\phantom{\draw [line width=0.5pt,line cap=round,rounded corners] (1.north west)  rectangle (1.south east);}
\phantom{\draw [line width=0.5pt,line cap=round,rounded corners] (4.north west)  rectangle (4.south east);}
\draw (-2.5,0)--(-0.5,0);
\draw (0.5,0.1)--(1.5,0.1);
\draw (0.5,-0.1)--(1.5,-0.1);
\draw (-0.5,0)--(0.5,0);
\draw (1-0.15,0.3) -- (1+0.08,0) -- (1-0.15,-0.3);%arrow
\draw [line width=0.5pt,line cap=round,rounded corners] (1.north west)  rectangle (1.south east);
\draw [line width=0.5pt,line cap=round,rounded corners] (2.north west)  rectangle (2.south east);
\draw [line width=0.5pt,line cap=round,rounded corners] (0.north west)  rectangle (0.south east);
\draw [line width=0.5pt,line cap=round,rounded corners] (3.north west)  rectangle (3.south east);
\end{tikzpicture}
\quad\mapsto\quad
\begin{tikzpicture}[scale=0.5,baseline=-0.5ex]
%\node at (0,0.8) {};
%\node at (0,-0.8) {};
\node [inner sep=0.8pt,outer sep=0.8pt] at (-2.5,0) (0) {$\bullet$};
\node [inner sep=0.8pt,outer sep=0.8pt] at (-1.5,0) (1) {$\times$};
\node [inner sep=0.8pt,outer sep=0.8pt] at (-0.5,0) (2) {$\bullet$};
\node [inner sep=0.8pt,outer sep=0.8pt] at (0.5,0) (3) {$\bullet$};
\node [inner sep=0.8pt,outer sep=0.8pt] at (1.5,0) (4) {$\bullet$};
\draw (-0.5,0)--(0.5,0);
\draw (0.5,0.1)--(1.5,0.1);
\draw (0.5,-0.1)--(1.5,-0.1);
%\draw (-0.5,0)--(0.5,0);
\draw (1-0.15,0.3) -- (1+0.08,0) -- (1-0.15,-0.3);%arrow
\draw [line width=0.5pt,line cap=round,rounded corners] (0.north west)  rectangle (0.south east);
\draw [line width=0.5pt,line cap=round,rounded corners] (2.north west)  rectangle (2.south east);
\draw [line width=0.5pt,line cap=round,rounded corners] (3.north west)  rectangle (3.south east);
%\draw [line width=0.5pt,line cap=round,rounded corners] (4.north west)  rectangle (4.south east);
\end{tikzpicture}
\quad\mapsto\quad
\begin{tikzpicture}[scale=0.5,baseline=-0.5ex]
%\node at (0,0.8) {};
%\node at (0,-0.8) {};
%\node [inner sep=0.8pt,outer sep=0.8pt] at (-2.5,0) (0) {$\bullet$};
%\node [inner sep=0.8pt,outer sep=0.8pt] at (-1.5,0) (1) {$\times$};
\node [inner sep=0.8pt,outer sep=0.8pt] at (-0.5,0) (2) {$\bullet$};
\node [inner sep=0.8pt,outer sep=0.8pt] at (0.5,0) (3) {$\bullet$};
\node [inner sep=0.8pt,outer sep=0.8pt] at (1.5,0) (4) {$\bullet$};
\draw (-0.5,0)--(0.5,0);
\draw (0.5,0.1)--(1.5,0.1);
\draw (0.5,-0.1)--(1.5,-0.1);
%\draw (-0.5,0)--(0.5,0);
\draw (1-0.15,0.3) -- (1+0.08,0) -- (1-0.15,-0.3);%arrow
%\draw [line width=0.5pt,line cap=round,rounded corners] (0.north west)  rectangle (0.south east);
\draw [line width=0.5pt,line cap=round,rounded corners] (2.north west)  rectangle (2.south east);
\draw [line width=0.5pt,line cap=round,rounded corners] (3.north west)  rectangle (3.south east);
%\draw [line width=0.5pt,line cap=round,rounded corners] (4.north west)  rectangle (4.south east);
\end{tikzpicture}
\quad\mapsto\quad
\begin{tikzpicture}[scale=0.5,baseline=-0.5ex]
%\node at (0,0.8) {};
%\node at (0,-0.8) {};
%\node [inner sep=0.8pt,outer sep=0.8pt] at (-2.5,0) (0) {$\bullet$};
%\node [inner sep=0.8pt,outer sep=0.8pt] at (-1.5,0) (1) {$\times$};
\node [inner sep=0.8pt,outer sep=0.8pt] at (-0.5,0) (2) {$\bullet$};
\node [inner sep=0.8pt,outer sep=0.8pt] at (0.5,0) (3) {$\times$};
\node [inner sep=0.8pt,outer sep=0.8pt] at (1.5,0) (4) {$\bullet$};
%\draw (-0.5,0)--(0.5,0);
%\draw (0.5,0.1)--(1.5,0.1);
%\draw (0.5,-0.1)--(1.5,-0.1);
%\draw (-0.5,0)--(0.5,0);
%\draw (1-0.15,0.3) -- (1+0.08,0) -- (1-0.15,-0.3);%arrow
%\draw [line width=0.5pt,line cap=round,rounded corners] (0.north west)  rectangle (0.south east);
\draw [line width=0.5pt,line cap=round,rounded corners] (2.north west)  rectangle (2.south east);
%\draw [line width=0.5pt,line cap=round,rounded corners] (3.north west)  rectangle (3.south east);
%\draw [line width=0.5pt,line cap=round,rounded corners] (4.north west)  rectangle (4.south east);
\end{tikzpicture}
\quad\mapsto\quad
\begin{tikzpicture}[scale=0.5,baseline=-0.5ex]
%\node at (0,0.8) {};
%\node at (0,-0.8) {};
%\node [inner sep=0.8pt,outer sep=0.8pt] at (-2.5,0) (0) {$\bullet$};
%\node [inner sep=0.8pt,outer sep=0.8pt] at (-1.5,0) (1) {$\times$};
%\node [inner sep=0.8pt,outer sep=0.8pt] at (-0.5,0) (2) {$\bullet$};
%\node [inner sep=0.8pt,outer sep=0.8pt] at (0.5,0) (3) {$\times$};
\node [inner sep=0.8pt,outer sep=0.8pt] at (1.5,0) (4) {$\bullet$};
%\draw (-0.5,0)--(0.5,0);
%\draw (0.5,0.1)--(1.5,0.1);
%\draw (0.5,-0.1)--(1.5,-0.1);
%\draw (-0.5,0)--(0.5,0);
%\draw (1-0.15,0.3) -- (1+0.08,0) -- (1-0.15,-0.3);%arrow
%\draw [line width=0.5pt,line cap=round,rounded corners] (0.north west)  rectangle (0.south east);
%\draw [line width=0.5pt,line cap=round,rounded corners] (2.north west)  rectangle (2.south east);
%\draw [line width=0.5pt,line cap=round,rounded corners] (3.north west)  rectangle (3.south east);
%\draw [line width=0.5pt,line cap=round,rounded corners] (4.north west)  rectangle (4.south east);
\end{tikzpicture}
$$
while $\sB_{5;3}^1$ is not polar closed because
\begin{align*}
\sX_0&=\begin{tikzpicture}[scale=0.5,baseline=-0.5ex]
%\node at (0,0.8) {};
%\node at (0,-0.8) {};
\node [inner sep=0.8pt,outer sep=0.8pt] at (-2.5,0) (0) {$\bullet$};
\node [inner sep=0.8pt,outer sep=0.8pt] at (-1.5,0) (1) {$\bullet$};
\node [inner sep=0.8pt,outer sep=0.8pt] at (-0.5,0) (2) {$\bullet$};
\node [inner sep=0.8pt,outer sep=0.8pt] at (0.5,0) (3) {$\bullet$};
\node [inner sep=0.8pt,outer sep=0.8pt] at (1.5,0) (4) {$\bullet$};
\phantom{\draw [line width=0.5pt,line cap=round,rounded corners] (1.north west)  rectangle (1.south east);}
\phantom{\draw [line width=0.5pt,line cap=round,rounded corners] (4.north west)  rectangle (4.south east);}
\draw (-2.5,0)--(-0.5,0);
\draw (0.5,0.1)--(1.5,0.1);
\draw (0.5,-0.1)--(1.5,-0.1);
\draw (-0.5,0)--(0.5,0);
\draw (1-0.15,0.3) -- (1+0.08,0) -- (1-0.15,-0.3);%arrow
\draw [line width=0.5pt,line cap=round,rounded corners] (1.north west)  rectangle (1.south east);
\draw [line width=0.5pt,line cap=round,rounded corners] (2.north west)  rectangle (2.south east);
\draw [line width=0.5pt,line cap=round,rounded corners] (0.north west)  rectangle (0.south east);
%\draw [line width=0.5pt,line cap=round,rounded corners] (4.north west)  rectangle (4.south east);
\end{tikzpicture}
\quad\mapsto\quad
\sX_1=\begin{tikzpicture}[scale=0.5,baseline=-0.5ex]
%\node at (0,0.8) {};
%\node at (0,-0.8) {};
\node [inner sep=0.8pt,outer sep=0.8pt] at (-2.5,0) (0) {$\bullet$};
\node [inner sep=0.8pt,outer sep=0.8pt] at (-1.5,0) (1) {$\times$};
\node [inner sep=0.8pt,outer sep=0.8pt] at (-0.5,0) (2) {$\bullet$};
\node [inner sep=0.8pt,outer sep=0.8pt] at (0.5,0) (3) {$\bullet$};
\node [inner sep=0.8pt,outer sep=0.8pt] at (1.5,0) (4) {$\bullet$};
\draw (-0.5,0)--(0.5,0);
\draw (0.5,0.1)--(1.5,0.1);
\draw (0.5,-0.1)--(1.5,-0.1);
%\draw (-0.5,0)--(0.5,0);
\draw (1-0.15,0.3) -- (1+0.08,0) -- (1-0.15,-0.3);%arrow
\draw [line width=0.5pt,line cap=round,rounded corners] (0.north west)  rectangle (0.south east);
\draw [line width=0.5pt,line cap=round,rounded corners] (2.north west)  rectangle (2.south east);
%\draw [line width=0.5pt,line cap=round,rounded corners] (3.north west)  rectangle (3.south east);
%\draw [line width=0.5pt,line cap=round,rounded corners] (4.north west)  rectangle (4.south east);
\end{tikzpicture}
\quad\mapsto\quad
\sX_2=\begin{tikzpicture}[scale=0.5,baseline=-0.5ex]
%\node at (0,0.8) {};
%\node at (0,-0.8) {};
%\node [inner sep=0.8pt,outer sep=0.8pt] at (-2.5,0) (0) {$\bullet$};
%\node [inner sep=0.8pt,outer sep=0.8pt] at (-1.5,0) (1) {$\times$};
\node [inner sep=0.8pt,outer sep=0.8pt] at (-0.5,0) (2) {$\bullet$};
\node [inner sep=0.8pt,outer sep=0.8pt] at (0.5,0) (3) {$\bullet$};
\node [inner sep=0.8pt,outer sep=0.8pt] at (1.5,0) (4) {$\bullet$};
\draw (-0.5,0)--(0.5,0);
\draw (0.5,0.1)--(1.5,0.1);
\draw (0.5,-0.1)--(1.5,-0.1);
%\draw (-0.5,0)--(0.5,0);
\draw (1-0.15,0.3) -- (1+0.08,0) -- (1-0.15,-0.3);%arrow
%\draw [line width=0.5pt,line cap=round,rounded corners] (0.north west)  rectangle (0.south east);
\draw [line width=0.5pt,line cap=round,rounded corners] (2.north west)  rectangle (2.south east);
%\draw [line width=0.5pt,line cap=round,rounded corners] (3.north west)  rectangle (3.south east);
%\draw [line width=0.5pt,line cap=round,rounded corners] (4.north west)  rectangle (4.south east);
\end{tikzpicture}
\end{align*}
where $\sX_2$ is a non-empty diagram with no polar node encircled. In fact, by inspection of the list of diagrams, the admissible Dynkin diagrams of classical type that are not polar closed are precisely the diagrams listed in the introduction.

Suppose that $\sX$ is polar closed, and let $\sX=\sX_0,\sX_1,\ldots,\sX_k$ be a sequence of subdiagrams obtained by successively removing encircled polar nodes with $\sX_k$ empty (as in the definition of polar closed). Let $\varphi_1,\ldots,\varphi_k\in\Phi^+$ be the highest roots associated to the polar nodes that are removed at each step. For example, the sequence of highest roots corresponding to the above sequence for $\sB_{5;4}^1$ is (in the basis of simple roots)
\begin{align*}
\varphi_{\sB_5}=(12222)\quad\mapsto\quad\varphi_{\sA_1}=(10000)\quad\mapsto\quad\varphi_{\sB_3}=(00122)\quad\mapsto\quad\varphi_{\sA_1}=(00100).
\end{align*}
Note that the sequence $\sX_0,\ldots,\sX_k$ is not unique (for example, in the $\sB_{5;4}^1$ sequence given above we may remove the polar node of the $\sB_{3}$ diagram at the second step instead of the $\sA_1$ polar node), however it is clear that the set $\{\varphi_1,\ldots,\varphi_k\}$ of highest roots is independent of the choices made. Moreover, note that the roots $\varphi_1,\ldots,\varphi_k$ are mutually perpendicular (by the definition of the polar type), and hence the group 
$$
U(\sX)=\langle U_{\varphi_1},\ldots,U_{\varphi_k}\rangle
$$ 
generated by the associated positive root groups is abelian. We call an element $g\in U(\sX)$ \textit{generic} if 
$g=x_{\varphi_1}(a_1)\cdots x_{\varphi_k}(a_k)$ with $a_1,\ldots,a_k\neq 0$.

%
%\subsection{Opposition in projective and polar spaces}
%
%In projective spaces it is just co-dimension and disjoint.
%
%In the classical polar spaces, totally isotropic substaces $U$ and $V$ are opposite if and only if they have the same dimension, and for each point $p$ of $U$ there is a point $p'$ of $V$ such that $p$ and $p'$ are opposite. 
%

\subsection{Projective spaces}

Buildings of type $\sA_n$ are equivalent to projective spaces in a well known way, with type $1$ vertices of the building corresponding to points of the projective space, type $2$ vertices corresponding to lines, and so on, with type $n$ vertices corresponding to hyperplanes. An automorphism of a projective space is a \textit{collineation} (respectively, \textit{duality}) if the corresponding automorphism of the building of type $\sA_n$ is type preserving (respectively, interchanges types $i$ and $n-i+1$ for $i=1,\ldots,n$). 

We denote by $\PG(n,\KK)$ the projective space over $\K$ of dimension~$n$ over a (possibly noncommutative) field~$\KK$. The points, lines, etc, are the $1$-spaces, $2$-spaces, etc of the vector space~$\KK^{n+1}$.

If $\FF$ is a subfield of $\KK$, with $\KK$ 2-dimensional over $\FF$, then the natural inclusion of $\PG(n,\FF)$ in $\PG(n,\KK)$ is called a \textit{Baer subspace} (however note that Baer subspaces are not ``subspaces'' in the technical sense of point-line geometry). A \textit{Baer collineation} of $\PG(n,\KK)$ is a collineation that pointwise fixes a Baer subspace. 

A \textit{symplectic polarity} of $\PG(n,\KK)$ is a duality $\theta$ of the form $U^{\theta}=\{v\in\KK^{n+1}\mid (u,v)=0\text{ for all $u\in U$}\}$, where $(\cdot,\cdot)$ is a nondegenerate symplectic form on $\KK^{n+1}$ (necessarily $n$ is odd and $\KK$ is commutative; and here $U$ is a subspace). 

An $(I_{n-i},-I_{i})$-homology of $\PG(n-1,\K)$, $\K$ a commutative field, is a collineation arising from a linear map in the underlying vector space with diagonal matrix containing $n-i$ times 1 and $i$ times $-1$.

Domesticity in buildings of type $\sA_n$ is well understood (see Theorem~\ref{thm1:An} below).

\begin{thm}[{\cite[Theorem~3.1, 4.3]{TTM:11} and \cite[Theorem~3.5]{PVM:19a}}]\label{thm1:An}
Let $\Delta$ be the building $\sA_n(\KK)$, regarded as the projective space $\PG(n,\KK)$, with $\KK$ a possibly noncommutative field and $n\geq 2$. 
\begin{compactenum}[$(1)$]
\item A collineation $\theta$ of $\Delta$ has opposition diagram $\Diag(\theta)=\mathsf{A}_{n;i}^1$ (with $0\leq i\leq n/2$) if and only if $\theta$ pointwise fixes a unique subspace of projective dimension~$n-i$ but does not fix any subspace of larger projective dimension pointwise.
\item If $|\KK|>2$ then $\Delta$ admits domestic dualities if and only if $n$ is odd and $\KK$ is commutative. In this case, every domestic duality is a symplectic polarity (that is, $\theta$ is induced from a nondegenerate symplectic form on $\KK^{n+1}$), with opposition diagram $\Diag(\theta)=\mathsf{A}_{n;(n-1)/2}^2$. 
%\item If $\KK=\FF_2$ then there exist domestic dualities of $\Delta$ that are not symplectic polarities, and each such duality $\theta$ maps at least one simplex of each type $\{1,\ldots,n\}\backslash\{j\}$ onto an opposite simplex, for each $1\leq j\leq n$.
\end{compactenum}
\end{thm}

Thus the focus of this paper is on buildings of types $\sB_n$, $\sC_n$, and $\sD_n$.

\subsection{Polar spaces}\label{sec:polarspace}

Building of type $\sB_n$, $\sC_n$, and $\sD_n$ may all be considered as polar spaces~$\Pi$, where the points of~$\Pi$ are the type $1$ vertices of the building. We recall the basic definitions below. 

\begin{defn}[{See \cite[Chapter~7]{DeB:16}}] Let $n>0$. A \textit{polar space} of rank~$n$ is a pair $\Pi=(\cP,\Omega)$ consisting of a set $\cP$  (the \textit{points}) and a set $\Omega$ of subsets of $\cP$ (the \textit{singular subspaces}) satisfying the following axioms:
\begin{compactenum}[$(1)$]
\item A singular subspace $X\in\Omega$ together with the singular subspaces contained in it define a (possibly reducible) $d$-dimensional projective space for some $d\in\{-1,0,\ldots,n-1\}$ (then $d=\dim(X)$ is the \textit{dimension} of $X$).
\item The intersection of any two singular subspaces is again a singular subspace. 
\item If $X$ is an $(n-1)$-dimensional singular subspace and if $p\in\cP\backslash X$ then there exists a unique $(n-1)$-dimensional singular subspace $Y$ containing $p$ such that $\dim(X\cap Y)=n-2$. The singular subspace $X\cap Y$ consists of those points $q$ of $X$ for which $\{p,q\}$ is contained in some singular subspace of dimension~$1$.
\item There exist two disjoint singular subspaces of dimension $n-1$.
\end{compactenum}
\end{defn}

Let $\Pi=(\cP,\Omega)$ be a polar space of rank~$n$. The singular subspaces of dimension $1$ (respectively~$2$) are called \textit{lines} (respectively~\textit{planes}). The singular subspaces of dimension $n-1$ (respectively $n-2$) are called \textit{maximal} (respectively \textit{submaximal}) singular subspaces. A singular subspace of dimension $i$ will be called a \emph{singular $i$-space}.  Points $p,q\in\cP$ are \textit{collinear} if they are contained in a line, and note that if two distinct points are contained in a singular subspace then they are collinear by the first axiom. The (unique) line determined by two distinct points $x,y$ will be denoted $\<x,y\>$. The symbol $\<\cdot\>$ will be used throughout and will have the obvious meaning of generation.  

Let $\cL$ denote the set of lines of $\Pi=(\cP,\Omega)$. Taking incidence to be containment, one may regard $\Pi$ as a point line geometry $\Pi'=(\cP,\cL)$. Indeed, there is an equivalent set of axioms, due to Buekenhout \& Shult, that characterise polar spaces in terms of this incidence structure (see \cite[\S 7.2]{DeB:16}). In particular, $\Pi'$ satisfies the following condition (the Buekenhout-Shult one-or-all axiom): \textit{If $p\in\cP$ and $L\in\cL$ are not incident, then either one or all points of $L$ are collinear with~$p$.}

The set $\Omega$ can be recovered from the point line geometry $\Pi'=(\cP,\cL)$ as the set of all $X\subseteq \cP$ that are \textit{subspaces} of $\Pi'$ (that is, every line that has at least two distinct points in $X$ has all of its points in $X$) and \textit{singular} (that is, $X$ consists of mutually collinear points). Thus we will identify $\Pi$ and $\Pi'$. 

We write $p\perp q$ to indicate that points $p$ and $q$ are collinear and we write $p^{\perp}$ for the set of all points collinear to $p$ (including $p$ itself). If $X\subseteq\cP$ then we write $X^{\perp}=\bigcap_{p\in X}p^{\perp}$. 

Points $x,y\in \cP$ are opposite (in the building theoretic sense) one another if and only if they are not collinear. More generally, opposition of singular subspaces is given by the following obvious lemma.

\begin{lemma} Singular spaces $X,Y\in\Omega$ are opposite if and only if they have the same dimension, are disjoint, and for each point $x\in X$ there is a point $y\in Y$ such that $x$ and $y$ are opposite. In particular, maximal singular subspaces $X,Y$ are opposite if and only if they are disjoint.
\end{lemma}

We recall the following definitions.

\begin{defn}\label{defn:polarspacethings} Let $\Pi$ be a polar space of rank~$n$. 
\begin{compactenum}[$(1)$]
\item A \textit{subspace of corank $t$}, $0\leq t\leq n-1$, of a polar space of rank $n$ is a subspace having non-empty intersection with every singular subspace of dimension $t$, and such that it is disjoint from at least one singular subspace of dimension $t-1$. 
%\item A \textit{geometric subspace} of a polar spaces is a set $S$ of points such that if distinct points $x,y\in S$ are collinear, then all points on the line $xy$ belong to~$S$. The \textit{corank} of a geometric subspace $S$ equals $i$ if every $i$-dimensional subspace meets $S$ in at least one point, and there exists an $i$-dimensional subspace meeting it in exactly one point. It is easy to see that geometric subspaces are (possibly degenerate) polar spaces. 
\item A \textit{geometric hyperplane} is a subspace of corank~1, that is, a subspace with the property that every line contains at least one point of the subspace, and at least one line contains exactly one point of it. It is called \emph{singular} if it coincides with $p^\perp$ for some point $p$; otherwise it is called \emph{nondegenerate}, and then the point line geometry induced in it is a polar space.
\item A subspace of corank~2 will also be called a \emph{geometric subhyperplane.}
\item A \textit{collineation} of $\Pi$ is a permutation of the point set $\cP$ that preserves the line set~$\cL$.
\item A \textit{central collineation} (with \textit{centre}~$p\in\cP$) is a collineation that fixes every point collinear to~$p$.
\item An \textit{axial collineation} (with \textit{axis} $L\in\cL$) is a collineation~$\theta$ that stabilises each line that intersects $L$. Necessarily $\theta$ fixes each point collinear with all points of $L$. 
\item Central collineations $\theta_i$ ($i=1,2$) with centres $p_i$ ($i=1,2$) are \textit{perpendicular} if $p_1$ and $p_2$ are collinear. 
\item Axial collineations $\theta_i$ ($i=1,2$) with centres $L_i$ ($i=1,2$) are \textit{perpendicular} if either $L_1$ and $L_2$ intersect in a point and are not coplanar (we call these of \textit{the first kind}), or $L_1$ and $L_2$ are contained in the same $3$-space and do not intersect (we call these \textit{of the second kind}). 
\item Two central collineations, or two axial collineations, are \textit{opposite} if their centres are opposite. 
\item An \emph{ideal Baer collneation} of  polar space of rank 3 is a collineation that has at least one fixed line, induces a Baer collineation in every fixed plane, and fixes every plane through each fixed line. 
\end{compactenum}
\end{defn}

We will always assume, unless explicitly stated otherwise, that polar spaces have thick lines. That is, every line contains at least $3$ points. We call a polar space of rank $n$ \textit{non-thick} if every $(n-2)$-dimensional singular subspace is contained in exactly two maximal singular subspaces (of dimension $n-1$), and thick otherwise. Thick polar spaces correspond to buildings of type $\sB_n$ and $\sC_n$ (the chambers of the building are maximal flags of singular subspaces in the polar space, and in this way the type $i+1$ vertices of the building correspond to singular subspaces of dimension~$i$ for $0\leq i\leq n-1$). In the non-thick case every $(n-2)$-dimensional singular subspace is contained in exactly two maximal singular subspaces, and there are two types of maximal singular subspaces (members of the same type intersect each other in subspaces of even codimension). The so called ``oriflamme complex'' of such a polar space is then a thick building of type $\sD_n$.

We may regard an automorphism $\theta$ of a building $\Delta$ of type $\sB_n$, $\sC_n$, or $\sD_n$ as an automorphism, also denoted $\theta$, of the associated polar space $\Pi=(\cP,\Omega)$. Since neither dualities of $\sB_2/\sC_2$ buildings nor trialities of $\sD_4$ buildings play a role in this work, all such automorphisms will be collineations of the polar space $\Pi$ (however, note that in the $\sD_n$ case, the automorphism $\theta$ may interchange the type $n-1$ and $n$ vertices of the building).

All thick polar spaces, with the exception of one class in rank~$3$, are embeddable in projective space. More precisely, they arise as the isotropic geometry of nondegenerate quadratic forms, alternating forms, and Hermitian forms, or are a subspace thereof (see \cite{Tit:74}); in the latter case the underlying field is non-commutative and has characteristic $2$ and they are described by a pseudo-quadratic form, but we will not need this in this work, since, for example, we will prove results that exclude non-commutative fields, such as Proposition~\ref{startingpoint}. The polar spaces related to quadratic forms are called \emph{orthogonal polar spaces} and arise as nondegenerate quadrics in projective space.

The split spherical buildings of type $\sD_n$ correspond to \textit{hyperbolic quadrics} (with standard equation $X_{-2n}X_{2n}+X_{-2n+1}X_{2n-1}+\cdots+X_{-1}X_1=0$), and those of type $\sB_n$ correspond to the \textit{parabolic quadrics} (with standard equation $X_{-2n}X_{2n}+X_{-2n+1}X_{2n-1}+\cdots+X_{-1}X_1=X_{0}^2$). The split buildings of type $\sC_n$ are related to a nondegenerate alternating form in standard form $X_{-2n}Y_{2n}-X_{2n}Y_{-2n}+X_{-2n+1}Y_{2n-1}-X_{2n-1}Y_{-2n+1}+\cdots+X_{-1}Y_1-X_{1}Y_{-1}$. Note that these arise as the fixed point geometry of symplectic polarities in projective space. It is well-known that, if the field $\K$ is perfect of characteristic $2$, then the polar spaces of type $\mathsf{B}_n$ and $\mathsf{C}_n$ are isomorphic; if the field is not perfect (but still has characteristic $2$), then the polar space of type $\mathsf{B}_n$ is a polar subspace of the one of type $\mathsf{C}_n$. For ease of formulation, we will call a symplectic polar space in characteristic distinct from $2$ a \emph{proper symplectic polar space}.

Polar spaces associated to a nondegenerate $\sigma$-Hermtitian form, with $\sigma$ an involution of the underlying field $\KK$, are called Hermitian polar spaces. They are called \textit{minimal Hermitian} if the dimension of the projective space is $2n-1$, where $n$ is the rank of the polar space. The standard equation of a minimal Hermitian polar space is $X_{-2n}^{\sigma}X_{2n}+X_{2n}^{\sigma}X_{-2n}+X_{-2n+1}^{\sigma}X_{2n-1}+X_{2n-1}^{\sigma}X_{-2n+1}+\cdots+X_{-1}^{\sigma}X_1+X_{1}^{\sigma}X_{-1}=0$.

%
%\james{perhaps expand the previous paragraph: what are the names: symplectic, orthogonal, hyperbolic, and so on (see Bart's book, around page 223, at least for finite fields:}
%\begin{compactenum}[$(1)$]
%\item hyperbolic quadric $Q^+(2n-1,\KK)$ has form $x_1y_2+x_2y_1+\cdots+x_{2n-1}y_{2n}+x_{2n}y_{2n-1}$.
%\item parabolic quadric $Q(2n,\KK)$ has form $2x_0y_0+x_1y_2+x_2y_1+\cdots+x_{2n-1}y_{2n}+x_{2n}y_{2n-1}$.
%%\item elliptic quadrics $Q^-(2n+1,\LL/\KK)$: here $\LL:\KK$ is separable quadratic extension. See Bart.
%%\item pseudo-elliptic is as above, with inseparable quadratic extension.
%\item 
%\item Hermitian polar spaces
%\item minimal Hermitian polar space
%\item symplectic polar spaces $W(2n-1,\KK)$.
%\end{compactenum}
%
%

In the language of polar spaces it is convenient to use ``(projective) dimension as a singular subspace'' rather than ``type of a vertex of the building'', and therefore there is a shift of indexing. If $\theta$ is an automorphism of $\Pi$ we say that $\theta$ is \textit{$i$-domestic} if $\theta$ maps no singular subspace of dimension~$i$ to an opposite. Thus $i$-domesticity means domestic on type $i+1$ vertices of the building. There is an exception in non-thick polar spaces, where there are two types of maximal singular subspaces. These both have projective dimension $n-1$, and we will call them $(n-1)'$-spaces and $(n-1)$-spaces, and these correspond to the vertices of types $n-1$ and $n$ in the $\sD_n$ building. The expressions \textit{point-domestic} and \textit{line-domestic}, etc, have the obvious meanings.

%
%
%
%\begin{convention}\label{conv:conventions}
%We make the following conventions and observations:
%\begin{compactenum}[$(1)$]
%\item When dealing with polar spaces, collineation shall mean collineation of the polar space. Thus for buildings of type $\sD_n$ a collineation may interchange the types $n-1$ and $n$. When it it important to distinguish, we talk of type preserving collineations.
%\item In a polar space, $i$-domestic shall mean domestic on singular subspaces of (projective) dimension~$i$. Thus there is an index shift: $\{i\}$-domestic in the building sense, is $(i-1)$-domestic in the polar space sense, with some discussion of $\sD_n$. 
%\item Field shall allow non-commutative. 
%\item All buildings are thick (even if not explicitly stated)
%\end{compactenum}
%\end{convention}
%
%
%%
%
%Include the rank~$3$ polar space case. This is exceptional (as there are non-embeddable example). Therefore we may restrict attention to rank $\geq 4$ for polar spaces. So, things to include:
%
%\begin{compactenum}[$(1)$]
%\item Rank $2$ (generalised polygons). Owing to lack of classification of the actual geometries, we only have characgerisayion (in terms of fixed structure). The split $\sG_2$ details are in our paper. The classical quadrangles could similarly be worked out. Similarly the Ree-Tits octagons, while not currently known, could be worked out and is not considered here.
%\item Rank $3$ polar spaces;
%\item Small buildings;
%\item Trialities of $\sD_4$;
%\item Split exceptional buildings in our paper and also Hendrik's paper (and non-split dealt with elsewhere)
%\end{compactenum}

\section{Class I collineations}\label{sec2}

Recall that the class I collineations of a polar space are those domestic collineations that map at least one point to an opposite point. Thus the opposition diagram of a class I collineation is $\sX_{n;i}^1$ for some $0\leq i<n$ (where $\sX\in\{\sB,\sC,\sD\}$). Theorem~\ref{thm1:nonpointdom_1} (proved below) shows that these collineations are characterised by large structured fixed point sets. 

\begin{lemma}\label{lem:typeI}
Let $\Pi$ be a polar space of rank $n$ and let $0\leq i\leq n-1$. Suppose that $\theta$ is $i$-domestic. Moreover,
\begin{compactenum}[$(1)$]
\item if $i<n-1$ suppose that $\theta$ is $(i+1)$-domestic, and
\item if $i>0$ suppose that $\theta$ is not $(i-1)$-domestic. 
\end{compactenum}
Then $\theta$ fixes pointwise a geometric subspace of corank~$i$.
\end{lemma}

\begin{proof}
The statement in the case $i<n-1$ is precisely \cite[Theorem~6.1]{TTM:12}. If $i=n-1$ then we follow the ideas of the proof of \cite[Theorem 6.1]{TTM:12}. Only the first step of the proof in \cite[Theorem~6.1]{TTM:12} uses the fact that there exist singular subspaces of dimension $i+1$. Since here $i=n-1$ (and so no such singular subspaces exist), we provide a different argument for this step of the proof for this case.

Let $U$ be a non-domestic submaximal singular subspace and let $M$ be any maximal singular subspace containing $U$. By $(n-1)$-domesticity, $M^\theta$ intersects $M$ in a point $p$. We now claim that $p^\theta=p$. Suppose for a contradiction that $p^\theta\neq p$ and choose a submaximal singular subspace $U_0$ contained in $M$, with $p\in U_0$ and $p^{\theta^{-1}}\notin U_0$. Then $U_0^\theta$ contains $p^\theta$ but not $p$. Let $M_0$ be any maximal singular subspace containing $U_0$, but distinct from $M$. Then $M_0$ and $M_0^\theta$ have a point $p_0$ in common (by $(n-1)$-domesticity). Clearly, $p_0\notin M^\theta$. But $p_0$ is collinear with all points of $U_0^\theta$ (since it belongs to $M_0^\theta$) and also to $p$ (as it belongs to $M_0$), which implies that $p_0$ is collinear to all points of $M^\theta$, a contradiction. This proves the claim. 

It follows that $U^\perp\cap(U^\theta)^\perp$ is fixed pointwise, and the remainder of the proof of \cite[Theorem~6.1]{TTM:12} now applies verbatim.
\end{proof}

We can now prove Theorem~\ref{thm1:nonpointdom_1}.

\begin{proof}[Proof of Theorem~\emph{\ref{thm1:nonpointdom_1}}] 

Let $\theta$ be a nontrivial domestic collineation of a polar space $\Pi$ with opposition diagram $\sX_{n;i}^1$ with $\sX\in\{\sB,\sC,\sD\}$, $n\geq 2$, and $1\leq i\leq n$. Since we assume that $\Pi$ is large the collineation $\theta$ is capped, and so since $\theta$ is domestic we have $i<n$. Thus $\theta$ is domestic on type $i+1$ vertices of the associated building, that is, $\theta$ is $i$-domestic (recall the index shift for projective dimension). Then the hypothesis of Lemma~\ref{lem:typeI} holds, and hence $\theta$ fixes pointwise a geometric subspace of corank~$i$. 

For the converse, we assume that $\theta$ maps at least one point to an opposite point, fixes a geometric subspace of corank~$i$, and fixes no geometric subspace of any smaller corank. If $\theta$ maps no singular $(i-1)$-space to opposite then $\theta$ has diagram $\sX_{n;j}^1$ for some $j<i$ (by the classification of admissible diagrams, and the fact that a point is mapped to opposite). Thus by the previous paragraph $\theta$ fixes a geometric subspace of corank~$j<i$, a contradiction. It remains to show that $\theta$ does not map any singular $i$-space to opposite. This is clear from the definition of subspaces of corank~$i$:  the fixed structure intersects every singular $i$-space, hence no such singular subspace is mapped to an opposite. 
\end{proof}

The simplest situation occurs with opposition diagram $\mathsf{X}_{n;1}^1$, since this is equivalent to a collineation pointwise fixing exactly a geometric hyperplane.  It then depends on the structure of the possible geometric hyperplanes whether or not some general statements can be made about such collineations. In particular, if the polar space is split or non-embeddable, we can nail them all down. If the polar space is orthogonal (arises from a quadric) in characteristic distinct from $2$, then the geometric hyperplane must be nondegenerate. In the next paragraphs, we 
%
%
%\begin{cor}
%Let $\Delta$ be a thick building of type $\sX_n$ with $\sX\in\{\sB,\sC,\sD\}$ and let $\Pi=(\cP,\Omega)$ be the associated polar space. If a collineation $\theta$ has opposition diagram $\sX_{n;i}^1$ ($0\leq i<n-1$) then $\theta$ fixes pointwise a geometric subspace of corank~$i$. In particular, the fixed point set of $\theta$ is a possibly degenerate polar space (possibly with lines of size~$2$). 
%\end{cor}
%
%\begin{proof}
%If $\theta$ has opposition diagram $\sX_{n;i}^1$ then $\theta$ is domestic on vertices of type $i+1$ of the building, and hence is $i$-domestic (in the polar space indexing), and indeed is $j$-domestic for all $j\geq i$. If $i=0$ then $\theta$ is the identity (by Theorem~\ref{thm:nonempty}) and if $i>0$ then $\theta$ is not $(i-1)$-domestic. The result follows from Theorem~\ref{thm:typeI}. 
%\end{proof}
%
%
give more precise details for these cases. We start with symplectic polar spaces, where we can completely classify the collineations with diagrams $\mathsf{C}_{n;1}^1$ and $\mathsf{C}_{n;2}^1$. Consider first the case~$\sC_{n;1}^1$. 

\begin{thm}\label{polarsymplectic}
Every collineation $\theta$ of the symplectic polar space $\Delta=\mathsf{C}_{n,1}(\K)$ with polar opposition diagram (that is, $\Diag(\theta)=\sC_{n;1}^1$) is a central elation.
 \end{thm}

\begin{proof}
By \cite[Theorem~5.1]{TTM:12} the collineation $\theta$ fixes pointwise a geometric hyperplane~$H$. If $H$ spans the whole ambient $(2n-1)$-dimensional projective space, then $\theta$ is the identity (since every colineation of $\mathsf{C}_{n,1}(\K)$ is inherited from a collineation of the ambient projective space). Hence $H$ is contained in a projective hyperplane of the ambient projective space, and hence coincides with the perp of some point $p$. Since $\theta$ fixes $p^\perp$ pointwise, dually, $\theta$ fixes $p$ subspace-wise, that is, $\theta$ fixes all subspaces through $p$. Hence $\theta$ is the central collineation with center $p$ (and axis $p^\perp$ in the ambient projective space).  
\end{proof}

The opposition diagram $\sC_{n;2}^1$ is the \textit{polar-copolar} diagram of type $\sC_n$ (extending the terminology from~\cite{PVM:21}). In this case we have the following classification.

\begin{thm}\label{pocopo}
Let $n\geq 3$. A collineation $\theta$ of the symplectic polar space $\Delta=\mathsf{C}_{n,1}(\K)$ with $|\K|\geq 3$ has opposition diagram $\Diag(\theta)=\sC_{n;2}^1$ if and only it $\theta$ is either
\begin{compactenum}[$(1)$]
\item a product of two nontrivial perpendicular long root elations, or
\item a nontrivial member of the group generated by two opposite long root groups and not conjugate to a long root elation. 
\end{compactenum}
In particular, if $\mathrm{char}(\KK)\neq 2$ then all nontrivial short root elations have opposition diagram $\sC_{n;2}^1$.
\end{thm}

\begin{proof}
%For a point $p$, denote by $\theta(p)$ the central root elation with center $p$; for a line $L$ denote by $\theta(L)$ the short root elation with axis $L$.
%
Suppose that $\theta$ has opposition diagram $\sC_{n;2}^1$. Since $|\K|\geq 3$ the polar space is large, and so by Theorem~\ref{thm1:nonpointdom_1} the collineation $\theta$ pointwise fixes a subspace~$S$ of corank 2. We claim that in the ambient projective space $\PG(2n-1,\K)$, $S$ is a subspace, necessarily of dimension $2n-3$. Indeed, since $S$ is a subspace and is fixed pointwise, the projective subspace it generates in $\PG(2n-1,\K)$ is fixed pointwise by $\theta$ and we may assume it coincides with $S$ (otherwise we replace $S$ by it). If $\dim S=2n-2$, then $S$ is a hyperplane and $\theta$ is line-domestic (as every line contains a fixed point), a contradiction.  Now assume for a contradiction that $\dim S\leq 2n-4$. Since $\theta$ is not line-domestic, there exists a singular plane $\pi$ of $\Delta$ intersecting $S$ in just a point $p$. Let  $L$ be any line in $\pi$ not through $p$. Then every singular plane through $L$ intersects $S$, so $L^\perp\subseteq \<L,S\>$. This yields $\dim S\in\{2n-4,2n-5\}$. If $\dim S=2n-5$, then $L^\perp=\<L,S\>$, and considering another line $M\subseteq\pi$ not through $p$, we obtain $M^\perp\subseteq \<M,S\>=\<L,S\>$ and hence comparing dimensions, we see $L^\perp= M^\perp$, a contradiction. This shows $\dim S=2n-4$, hence $\dim\<L,S\>=2n-2$ and there is a unique point $x$ on $L$ such that $x^\perp=\<L,S\>=\<\pi,S\>$. Varying $L$ in $\pi$ appropriately, this implies $x^\perp=y^\perp$ for at least two points of $\pi$, a contradiction. Hence $\dim S=2n-3$. 

So $S=p^\perp\cap q^\perp$, for two points $p,q$ of $\Delta$. Suppose first that $p$ and $q$ are not collinear. Then the matrix $M_A$ of $\theta$ can be written as a block diagonal matrix with blocks the  $(2n-2)\times(2n-2)$-identity matrix and one $2\times 2$ block $A$ preserving a nondegenerate 2-dimensional symplectic form. Hence $A$ has determinant $1$. But then, $A$, and hence $\theta$, is generated by the unipotent upper triangular and lower triangular matrices, hence by opposite long root elations. 

Suppose now that $p$ and $q$ are collinear. We consider the symplectic form $x_1y_2-x_2y_1+x_3y_4-x_4y_3+\cdots+x_{2n-1}y_{2n}-x_{2n}y_{2n-1}$ and we may assume without loss of generality that $p=(1,0,0,\ldots,0)$, $q=(0,0,1,0,\ldots,0)$. Then one easily calculates that the matrix of $\theta$ is block diagonal, with all trivial blocks except the first $4\times4$ block, which equals
$$\left(\begin{array}{cccc} 1 & a & 0 & c \\ 0 & 1 & 0 & 0 \\ 0 & c & 1 & b \\ 0 & 0 & 0 & 1\end{array}\right),$$ for some $a,b,c\in\K$. If we denote the corresponding collineation of $\Delta$ by $\theta(a,b,c)$, then we see that $\theta(a,b,c)=\theta(a,0,0)\theta(0,b,0)\theta(0,0,c)$, where the first two are long root elations and the third one is a short root elation. Also, one calculates that, if $a\neq 0$, then $\theta(a,c^2/a,c)$ is a long root elation with center $(a,0,c,0,0,\ldots,0)$, and $\theta(a,b,c)=\theta(a,c^2/a,c)\theta(0,b-c^2/a,0)$ is the product of two perpendicular long root elations.   Similarly if $b\neq 0$. Hence we may assume that $(a,b)=(0,0)$. In this case $\theta=\theta(0,0,c)$ is a short root elation. If $\kar(\K)=2$, then $\theta$ is a long root elation in the corresponding building of type $\mathsf{B}_{n}$ and hence has opposition diagram $\mathsf{B}_{n;1}^2=\mathsf{C}_{n;1}^2$ (by \cite[Theorem~2.1]{PVM:21}). Suppose now $\kar(\K)\neq 2$. 
Then $\theta(0,0,c)=\theta(c/2,c/2,c/2)\theta(-c/2,-c/2,c/2)$, with $\theta(\pm d,\pm d,d)$ a central long root elation with center $(1,0,\pm 1,0,\ldots,0)$ (using respective signs).

Thus we have shown that every collineation with diagram $\sC_{n;2}^1$ is either as in (1) or~(2). The converse is a direct calculation (see the proof of Proposition~\ref{prop:PolarClosed_b}, for example).
\end{proof}

Quite similar to symplectic polar space is the result for thick non-embeddable polar spaces. We use $\mathsf{C_{3;1}^1}$ for the opposition diagrams here since these polar spaces behave like symplectic ones (and their Dynkin type is morally $\mathsf{C_3}$ in view of the underlying root system in the algebraic group sense, made apparent by the commutation relations of their root groups). 

\begin{thm}\label{polarnonemb}
Every collineation $\theta$ of a thick non-embeddable polar space $\Delta$ with polar opposition diagram (that is, $\Diag(\theta)=\sC_{3;1}^1$) is a central elation.
\end{thm}

\begin{proof}
According to \cite{Coh-Shu:90}, the only geometric hyperplanes are the singular ones, that is, all points collinear with a given fixed point $p$. Let $x$ be an arbitrary point not collinear to $p$. Then $x^\perp\cap p^\perp$ is pointwise fixed, and so $x^\theta$ is contained in $\{p,x\}^{\perp\perp}$. But then $\theta$ coincides with the unique central elation with centre $p$ mapping $x$ to $x^\theta$. For an explicit expression of such central elation (showing it really exists), see \cite[Proposition~4.13]{Bru-Mal:14}.
\end{proof}

A similar argument shows:

\begin{prop}\label{polarquadric}
Every collineation $\theta$ of an orthogonal polar space $\Delta$ over a field of characteristic not $2$, with opposition diagram $\Diag(\theta)=\mathsf{B}_{n;1}^1$ pointwise fixes a nondegenerate geometric hyperplane. Also, $\theta$ is an involution. 
\end{prop}

\begin{proof}
Since the characteristic of the field is not~$2$ the polar space is large, and so by  Theorem~\ref{thm1:nonpointdom_1} the collineation $\theta$ fixes a geometric hyperplane $H$ pointwise. Suppose that $H$ s singular, that is, $H=p^\perp$ for some point $p$ of $\Delta$. Since $\theta$ commutes with the nondegenerate polarity $\rho$ arising from the symmetric bilinear form that defines $\Delta$, and since $\theta$ fixes all points in $p^\rho=\<p^\perp\>$ (generation in the ambient projective space), it has to stabilize all hyperplanes through $p$, hence also all lines through $p$. But each such line outside $H$ intersects $\Delta$ in one further point $x$, hence $x^\theta=x$, for each point outside $H$. This shows the first assertion.  

A similar argument shows that $\theta$ is a homology in the ambient projective space with axis $\<H\>$ and center $c:=\<H\>^\rho$. Since lines through $c$ intersect $\Delta$ in at most two points, $\theta$ is an involution. 
\end{proof}

Next we consider split polar spaces of type $\mathsf{B}_n$, $n\geq 3$, and $\mathsf{D}_n$, $n\geq 4$.  For type $\mathsf{B}_n$, there is not anything we can say on top of Proposition~\ref{polarquadric}. However, for type $\mathsf{D}_n$, we can be slightly more specific. Indeed, the opposition diagram $\mathsf{D}_{n,1}^1$ characterizes the unique involution pointwise fixing a polar subspace of split type $\mathsf{B}_{n-1}$, and it fixes no chamber. 

\begin{thm}\label{polarDn}
Every collineation $\theta$ of a polar space $\Delta$ of type $\mathsf{D}_n$, $n\geq 4$, with opposition diagram $\mathsf{D}_{n;1}^1$ is an involution pointwise fixing a polar subspace of type $\mathsf{B}_{n-1}$. It follows that $\theta$ is automatically type-rotating. 
\end{thm}
\begin{proof}
 This follows from the fact that there are no nondegenerate quadrics in $\PG(2n-2,\K)$ of Witt index $n$ (since every pair of maximal subspaces would have to intersect nontrivally), and there is a unique quadric in $\PG(2n+1,\K)$ with Witt index $n$, and it is a parabolic one. 
\end{proof}

\begin{remark}\label{remnonperfect}Note that, if $\kar(\K)=2$, then $\mathsf{B}_{n,1}(\K)$ is a subspace of $\mathsf{C}_{n,1}(\K)$ (with equality if $\K$ is perfect). In the non-perfect case there exist line-domestic collineations distinct from central collineations, whereas in the perfect case, we can rely on Theorem~\ref{polarsymplectic} to deduce that we always have a central collineation. A counterexample for the non-perfect case is given by the collineations of $\PG(2n-1,\K)$ with action \[(x_{-n},x_{-n+1},\ldots,x_{-1},x_1,\ldots,x_n)\mapsto(x_{-n},x_{-n+1},\ldots,x_{-2},ax_1,a^{-1}x_{-1},x_2,\ldots,x_n),\] with $a$ a non-square, with respect to the (standard) form $x_{-n}x_n+\cdots+x_{-1}x_1\in\K^2$. (If this was a central collineation, then the point with coordinates all $0$ except $(x_{-1},x_1)=(a,1)$  would be the centre, but this point does not belong to the polar space.) In fact one can choose the coordinates so that the fixed subspace is the above one, that is, has equation $X_{-1}=aX_1$, and then a generic collineation pontwise fixing that geometric hyperplane is given by  \[(x_{-n},\ldots,x_{-1},x_1,\ldots,x_n)\mapsto(x_{-n},\ldots,(ab+1)x_{-1}+a^2bx_1,bx_{-1}+(ab+1)x_1,\ldots,x_n),\] with $ab^2+b\in\K^2$. 
\end{remark}
%\james{Hendrik to expand} For split polar spaces, the \emph{phantom polar} opposition diagram $\mathsf{B}_{n,1}^1$  also characterizes ``unique'' collineations. In characteristic 2, they are central (long root) elations (of a minimal mixed polar space). If the characteristic is different from 2, then they are involutions pointwise fixing a non-degenerate geometric hyperplane (in the ambient projective space they are involutory homologies). Easy to prove (a hyperplane is fixed, and so is its polar; these are the axis and the center, respectively, of a homology which must be involutory since secants through that point intersect the parabolic quadric in exactly two points, and these must be mapped onto each other).

\section{Point-domestic collineations of a polar space}

In this section we prove Theorems~\ref{thm1:base}--\ref{thm1:nonsplit}, dealing with point-domestic collineations. We begin with preliminary observations, before dividing the analysis into collineations of class II and class~III. Note that by \cite[Theorem~1]{PVM:19b} all point domestic collineations of a polar space are capped.

\subsection{Preliminary observations}

Here is the first basic observation.
\begin{prop}\label{pdlinefixed}
Let $\Delta$ be a polar space of rank at least $2$. Let $\theta$ be an automorphism of $\Delta$ and suppose that $\theta$ is point-domestic. Then for every point $p$ with $p^\theta\neq p$, the line $\<p,p^\theta\>$ is fixed under $\theta$. 
\end{prop}

\begin{proof}
Suppose for a contradiction that there exists a point $p$ with $p^\theta\neq p$ and $\<p,p^\theta\>$ not fixed under $\theta$. Pick $x\in\<p,p^\theta\>\setminus\{p,p^\theta\}$. Then $x^\theta\in\<p^\theta,p^{\theta^2}\>\setminus\{p^\theta,p^{\theta^2}\}$. By point-domesticity, $x$ is collinear to all points of $\<p^\theta,p^{\theta^2}\>$, and so $p,p^\theta,p^{\theta^2}$ are contained in a plane $\pi$ of $\Delta$, which they span. Let $\alpha$ be a plane containing $\<p,p^\theta\>$ and so that $\alpha$ and $\pi$ are not contained in a common singular subspace of $\Delta$. Pick $q\in\alpha\setminus\pi$ and $r\in\<q,p^\theta\>\setminus\{q,p^\theta\}$ and note that $q$ is not collinear to $p^{\theta^2}$. We then have $r\perp r^\theta$ and $p^\theta\perp r^\theta$ (as $p\perp r$, both are in $\alpha$). So $r^\theta$ is collinear to all points of $\<r,p^\theta\>$. Likewise, $q^\theta$ is collinear to all points of $\<q,p^\theta\>=\<r,p^\theta\>$. Hence all points of $\<q^\theta,r^\theta\>$ (in particular, $p^{\theta^2}$) are collinear to all points of $\<q,p^\theta\>$ (in particular, $q$). This contradiction concludes the proof.
\end{proof}

We note an immediate consequence of Proposition~\ref{pdlinefixed}.

\begin{cor}\label{square}
If $\theta$ is a point-domestic collineation of some polar space $\Delta$, then every member of $\<\theta\>$ is point-domestic.
\end{cor}

\begin{proof}
Let $x$ be any point of $\Delta$. If $x=x^\theta$, then $x$ is fixed by every member of $\<\theta\>$; if $x\neq x^\theta$ then $\theta$ stabilizes $\<x,x^\theta\>$ and so does every member of $\<\theta\>$.  
\end{proof}

The above proposition prompts the following question: Which collineations of a projective space have the property that each line joining a non-fixed point and its image is fixed? The answer is in the following proposition.

\begin{prop}\label{projsp}
Let $\Sigma$ be a finite-dimensional projective space with dimension at least $2$ and $\theta$ a collineation with the property that each point of $\Sigma$ is contained in at least one fixed line~(or, equivalently, the line $\<x,x^\theta\>$ is fixed for every point $x\neq x^\theta$). Then exactly one of the following holds.
\begin{compactenum}[$(i)$]
\item There are no fixed points and the set of fixed lines forms a line spread of $\Sigma$. The projective dimension of $\Sigma$ is odd, say $2n+1$, and the fixed subspaces form a projective space of dimension $n$. 
\item The set of fixed points forms a Baer subspace $\mathfrak{B}$ of $\Sigma$.
\item The set of fixed points is the union of two complementary subspaces $U,V$ of $\Sigma$ and $\theta$ is a $(U,V)$-homology, i.e., $\theta$ fixes all subspaces in $U\cup V$ and all subspaces containing one of $U,V$. 
\item The set of fixed points is a subspace $U$ of dimension $\ell$, with $2\ell+1\geq\dim\Sigma$, and the intersection of all fixed hyperplanes is a subspace $V\subseteq U$ of dimension $\dim\Sigma-\ell-1$. Then $\theta$ is a $(U,V)$-elation, i.e., $\theta$ fixes every subspace of $U$ and every subspace containing $V$.  
\end{compactenum}
\end{prop}

\begin{proof}
If $\theta$ has no fixed points, then clearly $(i)$ holds. So from now on we may assume that there is at least one fixed point. The proof is divided into the following steps.
\smallskip

\noindent\textit{Step 1.} Suppose first $\dim\Sigma=2$. If all lines are fixed, then $\theta$ is the identity and $(iv)$ holds with $\ell=2$. Hence we may assume that some line $L$ is not fixed. We claim that $x=L\cap L^\theta$ is fixed. Indeed, if not then $x\neq x^\theta\in L^\theta$. Since $x^\theta$ also belongs to every line through $x$ that is fixed under $\theta$, we obtain the contradiction that $L^\theta$ is fixed. The claim is proved. We obtain a second fixed point $y$ by intersecting the fixed lines through two distinct points of $L\setminus\{x\}$. 

Suppose first that no point off $xy$ is fixed. Let $z\in xy$ be arbitrary and choose a line $M\ni z$, with $M\neq L$. If $M=M^\theta$, then $z=M\cap xy$ is fixed; if $M\neq M^\theta$, then $M\cap M^\theta$ is fixed and it is necessarily $z$ as otherwise it lies off $xy$, contrary to our assumption. We have proved that each point of $xy$ is fixed and therefore $\theta$ is axial and hence central, so $(iv)$ holds with $\ell=1$. 

Suppose now a point $z$ off $xy$ is fixed, but no quadrangle is fixed. Let $K$ be a line disjoint from $\{x,y,z\}$. If $K$ is fixed, then the  intersection of the lines joining $x$ and $yz\cap K$, and $y$ and $xz\cap K$, respectively, is fixed and forms a quadrangle with $x,y,z$, a contradiction. Hence $K\neq K^\theta$ and since $K\cap K^\theta$ is fixed, it lies on one of the lines $xy,yx,xz$, say $xy$. Using a similar argument as above, one shows that for every line $M$ disjoint from $\{x,y,z\}$, the image $M^\theta$ intersects $M$ on $xy$. Hence we conclude that $xy$ is fixed pointwise and $(iii)$ holds. 

Finally suppose $\theta$ fixes some quadrangle. Then the fixed point structure is a subplane $\pi$. Every point is contained in at least one line of $\pi$ (by the assumption on $\theta$), and every line $N$ is either fixed, and hence contains many fixed points, or is not fixed and contains the fixed point $N\cap N^\theta$. So $(ii)$ holds.
\smallskip

\noindent\textit{Step 2.} We note that, if some subspace $U$ is fixed, then the restriction $\theta_{|_U}$ of $\theta$ to $U$ satisfies the assumptions of the proposition for $\Sigma$ replaced by $U$. %\james{clarify below where this is used}
\smallskip

\noindent\textit{Step 3.} Suppose $\dim\Sigma\geq 3$ and the fixed point set $P$ does not generate $\Sigma$. We first claim that $P$ is a subspace. Indeed, let $x,y\in P$ and let $z\in\Sigma\setminus\<P\>$. Then the line $L=\<z,z^\theta\>$ is fixed, and hence the plane $\alpha=\<L,x\>$ is fixed. Hence steps 1 and 2, and the choice of $z$ outside $\<P\>$ implies that $\alpha\cap P$ is a full line $M$ through $x$. We may assume $y\notin M$, as otherwise the claim follows. Then $u=L\cap M$ is fixed and so is the plane $\<y,z,u\>$. Again steps 1 and 2 and the choice of $z$ implies that all points of the line $\<y,u\>$ are fixed. Then the fixed plane $\<x,y,u\>$ contains two lines all of whose points are fixed, and the claim follows from steps 1 and~2. %\james{all step 1 go to step 1 and 2}

Note that the previous argument also shows that, whenever $z\in\Sigma\setminus P$, then $\<z,z^\theta\>$ contains a unique fixed point. Let $Q$ be the set of all points $q\in P$ such that some line through $q$ not in $P$ is fixed. We claim that $Q$ is a subspace. Indeed, let $q,r\in Q$. Then the definition of $Q$ implies that there exist points $x,y\in\Sigma\setminus P$ with $q\in\<x,x^\theta\>$ and $r\in\<y,y^\theta\>$. Clearly $X=\<x,y,q,r\>$ is $3$-dimensional, and $X\cap P=\<q,r\>$. Hence the line $L=\<x,y\>$ does not intersect $P$. This implies that $L\cap L^\theta=\emptyset$. Let $s\in\<q,r\>$ arbitrary. Let $z\in L$ and $z'\in L^\theta$ be such that $s\in \<z,z'\>$. Then $z'$ is the unique point on $L^\theta$ such that $\<z,z'\>$ intersects $P$. We deduce that $z'=z^\theta$ and so $s\in Q$. The claim is proved. 

Next we claim that $Q$ is the intersection of all fixed hyperplanes. Indeed, first let $H$ be a hyperplane containing $Q$. Pick $z\in H\setminus P$. Then $\<z,z^\theta\>$ intersects $Q\subseteq H$, hence $z^\theta\in H$. Consequently $H$ is fixed. Now let $H$ be a hyperplane not containing $Q$. Select $x\in Q\setminus H$. Let $z\in \Sigma\setminus P$ be such that $x\in\<z,z^\theta\>$. Since $x\notin H$, we may assume $z\in H$ (a hyperplane intersects every line nontrivially). But then $z^\theta\notin H$ and so $H\ne H^\theta$. The claim is proved. 

Now notice that in the dual of $\Sigma$, the roles of $P$ and $Q$ are interchanged. Hence, up to duality, we may assume $k=\dim Q\leq \frac{1}{2}(n-2)$ (if $k=\frac{1}{2}(n-1)=\dim P$, then $(iv)$ holds). Let $W$ be a subspace complementary to $Q$. Then $W\cap W^\theta$ has dimension at least $n-2k-2\geq 0$. Hence $W\cap W^\theta$ is nonempty. Let $w\in W\cap W^\theta$ be arbitrary. Then $\<Q,w\>$ is fixed and so is $w=\<Q,w\>\cap W=\<Q,w\>\cap W^\theta$. Hence $W\cap W^\theta\subseteq P$ and so $\ell=\dim P\geq n-k-1$. If $\ell\geq n-k$, then the dual argument could be applied and yielded $n-k-1\geq  n-k$, a contradiction. Hence $k+\ell=n-1$ and $(iv)$ holds.

\noindent\textit{Step 4.} Suppose $\dim\Sigma\geq 3$, the fixed point set $\mathfrak{B}$ generates $\Sigma$ and has the structure of a thick projective space. Then by steps 1 and 2,  every fixed plane $\alpha$ intersects $\mathfrak{B}$ in a Baer subplane of $\alpha$ and $(ii)$ holds. 

\noindent\textit{Step 5.} Suppose $\dim\Sigma\geq 3$, the fixed point set $P$ generates $\Sigma$ and has the structure of a non-thick generalized projective space. We first claim that, if $L$ is a thick line of $P$, then all points of $\Sigma$ incident with $L$ belong to $P$. Indeed, suppose not and let $\{x,y\}$ be a thin line of $P$. The plane $\pi=\<L,x\>$ is fixed by $\theta$ and steps~1 and~2 implies that $\pi\cap P$ is a Baer subplane of $\pi$. Let $B$ be a line of $\pi$ through $x$ intersecting $P$ in at least three points. Then $B$ is a proper subline of its carrier in $\Sigma$. The plane $\<B,y\>$ is fixed and steps~1 and~2 imply that $\<x,y\>$ contains at least three points of $P$, a contradiction. The claim is proved. 

Now let $\{x,y\}$ again be a thin line (which exists by assumption). Let $P_x$ and $P_y$ be the sets of points of $P$ lying on a thick line together with $x$ and $y$, respectively. We claim that $\{P_x,P_y\}$ is a partition of $P$. Indeed, if $u\in P_x\cap P_y$, then the plane $\<x,y,u\>$ contains two lines entirely consisting of fixed points, hence steps~1 and~2 imply that all points of that plane are fixed; in particular all points of $\<x,y\>$, contradicting our assumption on $\<x,y\>$. Also, suppose there exists $v\in P\setminus(P_x\cup P_y)$. Then the plane $\<x,y,v\>$ is fixed and $\<x,y,v\>\cap P=\{x,y,v\}$. This contradicts steps~1 and~2 once again. The claim follows. Consider now two arbitrary points $z,w$ of $P_x$, $z\neq w\neq x\neq z$. Then the plane $\<x,z,w\>$ contains two lines all of whose points are in $P$. Steps~1 and~2 imply that also all points of $\<z,w\>$ are fixed. Hence $P_x$, and similarly $P_y$, is a subspace of $\Sigma$. Now $(iii)$ holds.  
\end{proof}

Examples of each kind are easy to find. One remark concerning Case $(ii)$: if the underlying structure is a field (the projective space is Pappian then), then such an automorphism $\theta$ is induced by the identity matrix and a companion field automorphism pointwise fixing a subfield $\FF$ of a field $\K$ such that $\K$ is 2-dimensional over $\FF$. Hence $\theta$ is an involution. But if the underlying division ring is not commutative, then examples exist where $\theta$ is not an involution (note that this corrects \cite[Theorem~7.2]{TTM:11} where the statement ``ideal Baer involution'' should be replaced with ``ideal Baer collineation''). Consider for instance a field $\FF$ with a quadratic Galois extension $\K$, and denote the Galois involution $x\mapsto\overline{x}$. Then define the division ring $\K_{\FF}((t))$ of Laurent series $f(t)$ over $\K$ with multiplication induced by $xt=t\overline{x}$, for all $x\in\K$. Let $a\in\K$ be such that $a\overline{a}=1$, but $a\notin\{1,-1\}$ (always exists). Then the mapping $\theta:f(t)\mapsto f(at)$ defines an automorphism   pointwise fixing the field $\K((t^2))$ of Laurent series $f(t^2)$ over $\K$ (and $\K_{\FF}((t))$ is 2-dimensional over $\K((t^2))$ and the induced automorphism on any projective space of dimension at least 2 over $\K_{\FF}((t))$ satisfies $(ii)$). However, the order of $\theta$ is the multiplicative order of $a$ in $\K$, which is distinct from $2$ as $a\neq-1$.

This has now the following interesting consequence.

\begin{cor}\label{oddrankchamber}
A point-domestic collineation $\theta$ of a polar space $\Delta=(X,\Omega)$ of odd rank $r$ fixes a chamber.
\end{cor}

\begin{proof}
We first claim that $\theta$ fixes a maximal singular subspace, i.e., a singular subspace of dimension $r-1$. Note that \cite[Theorem~4.2]{TTM:12} proves the claim for polar spaces that are not symplectic, however here we provide an independent proof that includes the symplectic case. 

Suppose $\theta$ does not fix any maximal singular subspace. Let $d\leq r-2$ be the maximal dimension of a fixed singular subspace and let $U$ be fixed under $\theta$ with $\dim U=d$. Then $d\geq 1$ by Proposition~\ref{pdlinefixed}. Since $d<r-1$, there exists a point $x$ collinear to all points of $U$. If $x$ is fixed, then the singular subspace generated by $U$ and $x$ is fixed under $\theta$ and has dimension $d+1$, a contradiction to the maximality of $d$. Hence $x\neq x^\theta$ and $\theta$ fixes the line $xx^\theta$. Since $U$ is fixed, $x^\theta$ is collinear to all points of $U$ and $U$ and $xx^\theta$ generate a subspace $U'$ of dimension $d+1$ or $d+2$, fixed under $\theta$, again a contradiction to the maximality of $d$. We conclude that the claim holds. 

Now let $U$ be a maximal singular subspace fixed by $\theta$. By Proposition~\ref{projsp}, we always fix a chamber of $U$, and hence a chamber of $\Delta$, except in Case~$(i)$.  But that case does not occur here since $\dim U=r-1$ is even. 
\end{proof}

\begin{lemma}
A point-domestic collineation $\theta$ of a polar space $\Delta$ of rank at least $3$ with non-Pappian planes (hence defined over a non-commutative division ring) fixes a chamber.
\end{lemma}

\begin{proof}
If $\Delta$ has rank 3, then this follows from Corollary~\ref{oddrankchamber}. Suppose now that the rank of $\Delta$ is at least $4$. Then $\Delta$ is thick. If the opposition diagram of $\theta$ is $\mathsf{B}_{n;1}^2$, then Proposition~\ref{axialbase}, proved independently, implies that $\theta$ is an axial collineation, hence it fixes any chamber containing the axis. Consequently, by the classification of opposition diagrams, $\theta$ has opposition diagram $\mathsf{B}_{n;i}^2$, for $i>1$. But that means that there is some non-domestic $3$-space $U$. The duality $\theta_U$, however, must be domestic, implying it is a symplectic polarity, contradicting the fact that $U$ is non-Pappian.  
\end{proof}

Finally, we observe that we automatically have a class II collineation if it is $(n-1)$-domestic, as follows immediately from Lemma~\ref{lem:typeI}.
\begin{cor}\label{oppdiafull}
Let $\Delta$ be a thick building of type $\sX_n$ with $\sX\in\{\sB,\sC,\sD\}$ and let $\Pi=(\cP,\Omega)$ be the associated polar space. If a collineation $\theta$ has opposition diagram $\sX_{n,i}^2$, with $2i<n$, then $\theta$ fixes pointwise a geometric subspace or corank at most~$n-1$. 
\end{cor}

\subsection{Class II automorphisms}

If $\theta$ is a class II automorphism of a polar space of rank~$n$ then $\Diag(\theta)=\sX_{n;i}^2$ for some $0\leq i\leq n/2$ (where $\sX\in\{\sB,\sC,\sD\}$).

\subsubsection{The opposition diagrams $\sX_{n;1}^2$}

The point-domestic collineations with opposition diagram $\sX_{n;1}^2$ (with $\sX\in\{\sB,\sC,\sD\}$) can be classified. This is the content of Theorem~\ref{thm1:base}, which we prove in this subsection. We begin with the symplectic case, that is, Theorem~\ref{thm1:base}(1).

\begin{thm}\label{copolar}
Let $n\geq 2$. A collineation $\theta$ of the symplectic polar space $\Delta:=\mathsf{C}_{n,1}(\K)$, assumed to have at least one fixed point if $n=2$, having copolar opposition diagram (that is, $\mathsf{C}_{n;1}^2$) is unique. It is a involutive homology if $\kar(\K)\neq 2$, and if $\kar(\K)=2$, then $\theta$ is an axial elation.  
\end{thm}

\begin{proof}
First assume $n\geq 3$. Then $\theta$ is $j$-domestic, for all $j\geq 3$, $j\leq n$. So it follows from \cite[Theorem~6.1]{TTM:12} that the fixed point structure of $\theta$ is a subspace of corank 2. Hence the same cases as in the proof of Theorem~\ref{pocopo} apply. Clearly, with the notation of that proof, the case where $p$ and $q$ are non-collinear only occurs if all points of the hyperbolic line $\<p,q\>$ are fixed. This gives rise to a unique involutive homology if $\kar(\K)\neq 2$ and to the identity if $\kar(\K)=2$. If we have a collineation $\theta(a,b,c)$, $a,b,c\in\K$ then one easily calculates that $\theta(a,b,c)$ maps the point $(x_i)_{1\leq i\leq 2n}$ to a collinear one if and only if $$ax_2^2+2cx_2x_4+bx_4^2=0,$$ which implies the theorem. 

Next assume that $n=2$. Then we have a point-domestic symplectic generalized quadrangle. If follows from \cite{TTM:12b} that $\theta$ either linewise fixes a spread, or pointwise fixes an ideal subquadrangle, or is an axial elation. Since it is assumed that $\theta$ fixes at least one point, it is easily seen that the first case cannot occur (it leads to the identity).  

Suppose now $\theta$ pointwise fixes an ideal subquadrangle. Taking the standard symplectic form $x_0y_1-x_1y_0+x_2y_3-x_3y_2$, we may assume that the ideal subquadrangle contains all lines meeting both hyperbolic lines $(*,*,0,0)$ and $(0,0,*,*)$, with self-explaining notation. Then $\theta$ has the form $(x_0,x_1,x_2,x_3)\mapsto(x_0,x_1,ax_2,ax_3)$, for some $a\in \K$. This preserves the given symplectic form if and only if $a^2=1$, showing that we have an involutive homology if $\kar(\K)\neq 2$, and the identity if $\kar(\K)=2$. Since in characteristic $\neq 2$, a symplectic quadrangle does not admit axial elations, the theorem is proved. 
\end{proof}

Next we consider all other polar spaces. Rank 3 will be an exception, but the domestic collineations in this case are already classified in \cite[Section 6]{TTM:12}. This includes the more intricate non-embeddable polar spaces (which do not admit axial elations). 
 
%An \emph{axial elation} of a polar space is an automorphism that fixes some line $L$ and all points collinear to $L$ pointwise, and maps any other point $p$ to a collinear point $q$ such that the line $pq$ intersects $L$ nontrivially. 
In order to prove Theorem~\ref{thm1:base}(2), we need some preparations. 
We start with gathering some properties of geometric hyperplanes and subhyperplanes in the following proposition. The \emph{radical} of a subspace $S$ is the set of points of $S$ collinear to all points of $S$. It is always a singular subspace. 
%A \emph{geometric hyperplane} is a subspace of corank 1; a \emph{geometric subhyperplane} is a subspace of corank 2. 

\begin{prop}\label{corank}
Let $\Delta=(X,\Omega)$ be a nondegenerate polar space of rank at least $3$. Then the following properties hold.
\begin{itemize}
\item[$(i)$] Each geometric hyperplane and subhyperplane is a possibly degenerate polar space (and every point of the radical is called a deep point).
\item[$(ii)$] The radical of a geometric hyperplane or subhyperplane is at most $0$-dimensional, or $1$-dimensional, respectively.
\item[$(iii)$] Let $S$ be a geometric subhyperplane and let $p\in X$ be a point off $S$. Set $S'=\{x\in X\mid x\in L, p\in L, S\cap L\neq\emptyset\}$. If $S\subseteq p^\perp$, then $S'=p^\perp$ is a geometric hyperplane. If some point of $S$ is not collinear to $p$, then $S'$ is a geometric subhyperplane.   
\item[$(iv)$] Every plane containing a deep point of a geometric subhyperplane contains a line of that subhyperplane.
\end{itemize}  
\end{prop}

\begin{proof}
\begin{itemize}
\item[$(i)$] This follows immediately from the definition of subspace. 
\item[$(ii)$] Suppose the radical of a geometric hyperplane $H$ contains a line $L$. Then a line $L'$ opposite $L$ has no point in common with $H$, contradicting the definition of a geometric hyperplane. The other statement is proved in a similar fashion. 
\item[$(iii)$] Suppose first that $S\subseteq p^\perp$. Assume for a contradiction that some line $L$ through $p$ does not intersect $S$. Select any plane $\pi$ containing $L$. Since $S$ is a geometric subhyperplane, it intersects $\pi$ in a point $x$ (it cannot intersect in a line since this would contradict the fact that L does not meet $S$). Let $K$ be a line in $\pi$ not containing $x$, and let $\alpha$ be a plane through $K$ not collinear to $p$. Then $\alpha$ and $S$ should have a point $y$ in common, and by construction $y\notin K$. Then $y\perp p$, contradicting the choice of $\alpha$. This shows that $p^\perp=S'$. Clearly, this is a geometric hyperplane. 

Now suppose some point of $S$ is not collinear to $p$. Assume for a contradiction that some plane $\beta$ is disjoint from $S'$. Let $\beta_p$ be a plane through $p$ intersecting $\beta$ in a line $M$. Then $\beta_p$ contains some point $s$ of $S$; hence the line $ps$ is contained in $S'$. Consequently the point $ps\cap M$ is contained in $S'$. So $\beta$ contains a point of $S'$ after all.
\item[$(iv)$] Suppose for a contradiction that some plane $\alpha$ intersects a geometric subhyperplane $S$ (only) in a deep point $x$. Let $\alpha'$ be a plane through a line $K$ of $\alpha$ not through $x$ such that $\alpha'$ is not collinear to $x$. Then $\alpha'$ has some point $y$ in common with $S$ and $xy$ belongs to $S$, hence $xy\cap K$ belongs to $S$, a contradiction.  
\end{itemize} 
\end{proof}

Next, we need some more properties of subspaces (of certain corank).
\begin{prop}\label{subspaces} Let $S$ be a subspace of a polar space and assume that $S$ is not a singular subspace (but it can be a degenerate polar space).
\begin{compactenum}[$(i)$]
\item
If some singular subspace $D$ of dimension $d$ is a geometric hyperplane of $S$, then it contains the radical and maximal singular subspaces of $S$ have dimension either $d$ or $d+1$.
\item
If a set $D$ of points of $S$ is collinear to its complement$S\setminus D$, then either it is contained in de radical or the complement of $D$ is contained in the radical.
\item
The complement in $S$ of a proper geometric hyperplane of $S$ can never be contained in the radical.
\end{compactenum}  
\end{prop}

\begin{proof} For $x\in S$ we briefly write $x^\perp$ for $x^\perp\cap S$. 
\begin{itemize}
\item[$(i)$] Suppose first that there exists a point $x$ in the radical $R$ of $S$, which is not contained in $D$. Let $y,z$ be two non-collinear points of $S$. The lines $xy$ and $xz$ both contain some point of $D$, which, by assumption, are collinear, contradicting the choice of $y$ and $z$. So $D$ contains the radical of $S$. Then $D/R$ is a singular subspace which is a hyperplane in the residue of $R$, which is either a (nondegenerate) polar space (a contradiction since for every subspace there exists a disjoint line) or just a set of points. Then $D/R$ is either empty or a single point, which proves $(i)$.
\item[$(ii)$] Suppose $D$ is not contained in $R$ and let $x\in D\setminus R$. Then $x^\perp$ is a (proper) geometric hyperplane of $S$ and contains $S\setminus D$ by assumption. Suppose $x^\perp\setminus (D\cup R)\neq\emptyset$ and let $y\in x^\perp\setminus (D\cup R)$. Select $z\in S\setminus(x^\perp\cup y^\perp)$. Then $z\in D$ as $z\notin x^\perp$. But then $y\in D$ as $y\notin z^\perp$, a contradiction. 
\item[$(iii)$] Suppose $H$ is a proper geometric hyperplane of $S$ containing the complement of the radical $R$ of $S$. Let $x\in R\setminus H$ and $y\in H\setminus R$. Then every point $z\in\<x,y\>\setminus\{x,y\}$ belongs to $S\setminus (R\cup H)=\emptyset$, a contradiction.
\end{itemize}   The proposition is proved. 
\end{proof}

Finally we need to recognise symplectic polar spaces, which we already treated in Theorem~\ref{copolar}.

\begin{prop}\label{symplecticchar}
A polar space $\Delta$ of rank at least $3$ is symplectic if and only if for some pair of noncollinear points $\{x,y\}$ (and then for all such pairs) every point $z$ is contained in a line of the polar space containing a point of $\{x,y\}^\perp$ and a point of $\{x,y\}^{\perp\perp}$.   
\end{prop}

\begin{proof} The ``only if'' part is clear. We now show the``if''  part.

Replacing $\Delta$ with the intersection of the perps of two appropriate opposite singular subspaces of codimension 2, we may assume that $\Delta$ has rank 2 (but is classical, implying its automorphism group acts transitively on opposite pairs of points and hence we may assume the given condition holds for each pair of noncollinear points). We show that each point is projective i.e., for each point $x$ and every pair $\{y,z\}$ of points opposite $x$, the sets $\{x,y\}^\perp$ and $\{x,z\}^\perp$ either have exactly one point in common, or coincide. 

Indeed, we may assume $z\notin\{x,y\}^{\perp\perp}$ (otherwise $\{x,y\}^\perp=\{x,z\}^\perp$). Now, the given condition implies that $z$ is on a line $L$ containing $y'\in\{x,y\}^{\perp\perp}$ and $z'\in\{x,y\}^\perp$. Then $z'\in\{x,y,z\}^\perp$. If also $z'\neq z''\in\{x,y,z\}^\perp$, then $z\in\{z',z''\}^\perp\supseteq\{x,y\}^{\perp\perp}$. Hence $y'\in\{z',z''\}^\perp$, implying the whole line $L$ is collinear with $z''$, hence $z''\in L$ and consequently $z'=z''$. 

The result now follows from \cite{Sch:92}.
\end{proof}

We now prove Theorem~\ref{thm1:base}(2).
%\james{Consider merging the following two}

\begin{prop}\label{axialbase}
Let $\Delta=(X,\Omega)$ be a polar space of rank at least $3$, not of symplectic type, and let $\theta$ be an automorphism of $\Delta$ with opposition diagram $\mathsf{B}_{n;1}^2$ or $\mathsf{D}_{n;1}^2$. Then $\theta$ is an axial elation (and so $\Delta$ is an orthogonal polar space), or the rank is $3$ and $\theta$ is an ideal Baer collineation. 
\end{prop}

%\begin{lemma}
%Let $\Delta=(X,\Omega)$ be a polar space of rank at least $4$, not of symplectic type, and let $\theta$ be an automorphism of $\Delta$ with opposition diagram $\mathsf{B}_{n;1}^2$ or ${\mathsf{D}_{n;1}^2}$. Then the fixed point structure of $\theta$ is a subspace of corank $2$. The same conclusion holds for rank $3$ if $\theta$ is not a generalized Baer collineation.
%\end{lemma}

\begin{proof}
We first show that the fixed point structure of $\theta$ is precisely a subspace of corank $2$, whenever $\theta$ is not an ideal Baer collineation in rank 3.

Theorems 6.1 and Theorem 7.2 of \cite{TTM:12} already imply that $\theta$ pointwise fixes a subspace $S$ of corank 2. Note that $S$ is not a singular subspace as otherwise we can find a plane disjoint from $S$.  We show that $\theta$ does not fix any point outside $S$. Indeed, suppose for a contradiction that $\theta$ fixes some point $x\in X\setminus S$. 

We first claim that $\theta$ fixes pointwise every line incident with $x$ and with a point of $S$. 
Indeed, we distinguish between two situations. 

Suppose first that not all points of $S$ are collinear to $x$.  Then $x^\perp\cap S$ is a geometric hyperplane $H$ of $S$. Suppose for a contradiction that all points of $S\setminus H$ are collinear to all points of $H$. Then by Proposition~\ref{subspaces}$(ii)$, $(iii)$, $H$ is contained in the radical of $S$. Hence the union of all lines through $x$ intersecting $S$ is a singular subspace, but also a geometric subhyperplane by Proposition~\ref{corank}$(iii)$, clearly a contradiction (an opposite subspace is disjoint). %Since $H$ is a geometric hyperplane of $S$, it is easy to see that $H$ is equal to the radical of $S$. By Fact 1, $\dim H\leq 1$ and by Fact 2, the rank of $S$ is at most $2$. If $\dim H=1$, then the ranks of $S$ and $\Delta$ are equal, a contradiction. If $\dim H=0$, then $S$ had rank $1$ and so $\Delta$ has rank at most $3$. 
We conclude that there exist points $a\in H$ and $b\in S\setminus H$ which are not collinear. Hence the projection of $b$ onto the line $ax$ is also fixed. Consider any line $L$ in $S$ through $a$; then the restriction of $\theta$ to the singular plane spanned by $L$ and $x$ is the identity as it fixes a line pointwise and two additional points. So $\theta$  pointwise fixes all lines through $x$ intersecting $S$ in a point collinear to $a$. By connectivity of $H$, $\theta$ fixes all lines through $x$ intersecting $S$ pointwise.    

Now suppose that all points of $S$ are collinear to $x$. Assume first that there is a point $y\neq x$ with $y\perp S$. Note that $\{x,y\}^\perp=S$ in this case (this follows from Proposition~\ref{corank}$(iii)$). By Proposition~\ref{symplecticchar} we can select a point $z$ not contained in a line joining a point  of $S$ and $\{x,y\}^{\perp\perp}$. Note that $z$ is not fixed under $\theta$ as, by the choice of $z$, we can find a point $w$ of $x^\perp\setminus S$ collinear to $z$, which must be fixed by the one-or-all axiom as otherwise $ww^\theta$ contains $x$. Then the line $zz^\theta$ intersects $x^\perp$ in some point $u$, which is fixed by $\theta$, since every line through $x$ is fixed and since $zz^\theta$ is fixed by Proposition~\ref{pdlinefixed}. If $u\notin S$, then a similar argument as in the previous paragraph implies that $\theta$ fixes $x^\perp$ pointwise, implying line-domesticity, a contradiction.  Hence we may assume $u\in S$. Since $S$ is a section of $x^\perp$, it is a nondegenerate polar space, so there is a point $v\in S\setminus u^\perp$. The projection $x'$ of $v$ onto $zz^\theta$ is a fixed point of $\theta$ and $S$ is not contained in $x'^\perp$. Hence, by the first part of the proof, the line $zz^\theta$ is fixed pointwise. This contradiction implies that $z$ is fixed after all, and hence also the point $w$ above, and hence, as before, this implies that $x^\perp$ is fixed pointwise. This contradiction shows the claim.

The previous proof clearly showed that $S$ is not contained in $x^\perp$, as otherwise $x^\perp$ is fixed pointwise and so $\theta$ is line-domestic.  Let $H_x$ be the union of all lines through $x$ intersecting $S$. Then by Proposition~\ref{corank}$(iii)$  $H_x$ is a geometric subhyperpane of $\Delta$. %Indeed, let $\alpha$ be any plane of $\Delta$. We may assume that $x\notin \alpha$. Then there is a line $L\in \alpha\cap x^\perp$. The plane determined by $x$ and $L$ has at least one point $a$ in common with $S$ (as $S$ is a subspace of corank 2). Hence $ax\cap L$ is a point of $\alpha\cap H_x$, which shows our claim. 

Finally we claim that every line $K$ contains a fixed point. Indeed, let $K$ be any line and let $\beta$ be any plane containing $L$. Since $S$ is a subspace of corank 2, $\beta$ contains a point $s\in S$, and since $H_x$ is a subspace of corank 2, $\beta$ contains a point $h\in H_x$. If $s\neq h$, then interchanging the roles of $h$ and $x$ yields a pointwise fixed line in $\beta$ (namely, $hs$), and hence a fixed point on $K$. So we may assume $s=h\in H$. Suppose first that there exists a point $r\in S$ not collinear to $s$. Then by the first part of the proof of the present claim, the plane containing $r$ and the projection of $r$ onto $\beta$ contains a pointwise fixed line, which intersects $K$, proving the claim. So we may assume that $s$ is collinear with all points of $S$. Then $s$ is a deep point of $S$ and by Proposition~\ref{corank}$(iv)$ every plane through $s$ has a line in common with $S$, again yielding a fixed point on $K$.

Hence $\theta$ is line-domestic, contradicting our hypothesis. This shows that indeed, the set of fixed points for $\theta$ is the geometric subhyperplane $S$. Let $x$ be a point off $S$. We claim that the line $xx^\theta$ intersects $S$. Indeed, suppose not. Then $x^\perp\cap S = (x^\theta)^\perp\cap S$. Since the line $xx^\theta$ does not intersect $S$, Proposition~\ref{corank}$(iii)$ implies that $S$ is not contained in $x^\perp$. But then the projection from a point of $S\setminus x^\perp$ onto the line $xx^\theta$ is a fixed point outside $S$, a contradiction. %the union of the lines through $x$ intersecting $S$ is a geometric subhyperplane $S'$. The same reference yields that the union of lines through $x^\theta$ intersecting  $S'$ is a geometric hyperplane with $xx^\theta$ in the radical, contradicting Proposition~\ref{corank}$(ii)$. 
The claim follows.

Set $xx^\theta\cap S= c$. If some point $z$ opposite $c$ existed in $S$, then the projection of $z$ onto $xx^\theta$ is a point off $S$ that would be fixed, a contradiction. Hence $c$ is a deep point for $S$. We can now choose $y\in X\setminus S$ opposite $c$ and obtain a second deep point $c'=yy^\theta\cap S$ of $S$. Hence $cc'$ is the radical of $S$ (since the radical cannot be larger than a line by Proposition~\ref{corank}$(ii)$) and every point $z$ which is not fixed satisfies $zz^\theta\cap cc'\neq\emptyset$. So $\theta$ is an axial elation with axis $cc'$. 
\end{proof}

\subsubsection{Split polar spaces}

Our attention now goes to the split case. Indeed, we can classify all point-domestic collineations of split polar spaces which are also domestic in the maximal singular subspaces. First a lemma and some notation: For three pairwise disjoint lines $L,M,N$ of $\PG(3,\K)$, denote by $\proj^L_M(N)$ the projection of $N$ with center $L$ onto $M$, that is, $\proj^L_M(N)$ takes a point  $x\in N$ onto the intersection $\<L,x\>\cap M$. 

\begin{lemma}\label{regulus}
Let $L_1,L_2$ and $L_3$ be three pairwise disjoint lines in $\PG(3,\K)$ and let $M$ be a transversal, that is, $M$ is a line which intersects $L_i$ in a point $x_i$, for all $i\in\{1,2,3\}$. Let $L$ be a line through $x_1$ in the plane $\<L_1,M\>$, with $L_1\neq L\neq M$, and, for $|\K|>2$, let $L'$ be a line intersecting both $L_1$ and $M$, but disjoint from the triple $\{x_1,x_2,x_3\}$.  Then $\proj^{L_1}_{L_2}(L_3)\proj^L_{L_3}(L_2)$ has a unique fixed point, and $\proj^{L_1}_{L_2}(L_3)\proj^{L'}_{L_3}(L_2)$ has exactly two fixed points. 
\end{lemma}

\begin{proof}
This follows easily as $\{L_1,L_2,L_3\}$ and $\{L,L_2,L_3\}$ have exactly one common transversal, namely $M$, and $\{L_1,L_2,L_3\}$ and $\{L',L_2,L_3\}$ have exactly two common transversals, namely $M$ and the unique transversal of $\{L_1,L_2,L_3\}$ through $L'\cap L_1$. 
\end{proof}

The next theorem is the bulk of Theorem~\ref{thm1:middle}. Recall the definition of $\theta_\alpha$ for a collineation $\theta$ that maps an object $\alpha$ to opposite, just after Theorem~\ref{thm:capped} in Section~\ref{sec:background}. Recall also that a proper symplectic polar space is one in characteristic different from $2$.

\begin{thm}\label{BCD}
Let $\theta$ be a point-domestic collineation of a large parabolic, proper symplectic, or hyperbolic polar space $\Delta$ with opposition diagram $\mathsf{B}_{n;i}^2$, $2\leq 2i\leq n-1$, $\mathsf{C}_{n,i}^2$, $2\leq 2i\leq n-1$, or $\mathsf{D}_{n,i}^2$, $2\leq 2i\leq n-2$, respectively. Then $\theta$ is the product of $i$ pairwise orthogonal long root elations, an $(I_{2n-2i},-I_{2i})$-homology, or the product of $i$ pairwise orthogonal long root elations, respectively.
\end{thm}

\begin{proof}
We show this theorem by induction on $i$, the case $i=1$ being contained in Proposition~\ref{axialbase} and Theorem~\ref{copolar}. Let now $i\geq 2$.

We first observe that, under the product of at most $n/2$ perpendicular root elations in a parabolic polar space of rank $n$, or a hyperbolic polar space of rank $n+2$, every fixed line is either pointwise fixed or a unique point on it is fixed. Moreover, every point is in a fixed line. Also, under an $(I_{2n-2i},-I_{2i})$-homology of a symplectic polar space of rank $n$, every fixed line is either pointwise fixed of exactly two points on it are fixed. Moreover, every point is on a fixed line. 

The opposition diagram tells us that there is a line $L$ mapped onto an opposite. By largeness, we can choose $L$ such that $\theta_L$ is by induction either the product of $i-1$ pairwise orthogonal long root elations, an $(I_{2n-2i},-I_{2i-2})$-homology, or the product of $i-1$ pairwise orthogonal long root elations, respectively. Now let $\Sigma$ be a singular $3$-space through $L$ fixed under $\theta_L$ and such that all planes through $L$ in $\Sigma$ are also fixed under $\theta_L$. Then each plane $\pi$ of $\Sigma$ containing $L$ has the property $\pi\cap\pi^\theta\neq\emptyset$. Set $K=\Sigma\cap\Sigma^\theta$ and let $K'\subseteq\Sigma$ be such that ${K'}^\theta=K$. First assume that $K$ and $K'$ are disjoint. Then the point $x\in K'$ is mapped onto the point $\<L,x\>\cap K$. Hence the restriction to $K'$ of the action of $\theta$ is given by $\proj^L_K(K')$. Let $M$ be a line in $\Sigma$ disjoint from $K\cup K'$. We claim that $M$ is opposite $M^\theta$. Suppose not and let $u\in M^\theta$ be a point collinear to all points of $M$. Note that $u\notin K$ and $\<M,K\>=\Sigma$, which now contradicts the fact that $u$ is collinear to all points of $M\cup K$. Hence the claim. Now choosing $M$ appropriately in $\Sigma$, Lemma~\ref{regulus} leads to a contradiction to our observation in the second paragraph of this proof. If $K'\cap K$ is some point $z$, then the restriction to $K'$ of $\theta$ is given by the perspectivity with center $L\cap\<K,K'\>$. Considering a line $M$ intersecting $K,K'$ in an appropriate different point (on the line $\<z,L\cap\<K,K'\>\>$ for the symplectic case, not on that line for the other cases), we again obtain a contradiction to our observation in the second paragraph of this proof.   

So we have shown that $K$ is fixed pointwise. Now let $\pi$ be any plane through $L$ not fixed under $\theta_L$. Then by our observation again, $\pi$ is contained in a singular $3$-space $\Sigma$ fixed by $\theta_L$. We claim that $K:=\Sigma\cap\Sigma^\theta$ is stabilized by $\theta$ (with exactly one or two fixed points, depending on the orthogonal or symplectic case).    There are again two cases to rule out: $K'$ (defined as before as ${K'}^\theta=K$) disjoint from $K$, and $K'\cap K$ a singleton. Let's explain the first one; the second one is much easier and left to the reader. 

First the orthogonal case. Let $\pi$ be the unique plane containing $L$ in $\Sigma$ fixed by $\theta_L$. Let $M$ be the line in $\pi$ intersecting all of $L,K$ and $K'$ in points.   We can now choose a unique line $T$ through $M\cap L$ in the plane $\<L,M\>$ such that $\proj^T_K(K')=\theta/K'$. It follows that $\theta_T$ fixes all planes through $T$ inside $\Sigma$, and the previous part of the proof now implies that $K=\Sigma\cap \Sigma^\theta$ is fixed pointwise under $\theta$, a contradiction (as $\theta_L$ does not fix all planes through $L$ in $\Sigma$). 

Now the symplectic case goes similarly, noting that there are two planes $\pi_1,\pi_2$ through $L$ fixed under $\theta_L$. Let $M_i$ be the unique transversal of $L,K,K'$ in the plane $\pi_i$, then we can again appropriately choose a point $z$ on $M_1$ such that, setting $T:=\<z,M_2\cap L\>$, the map $\proj^T_K(K')$ coincides with $\theta/K'$. We obtain the same contradiction. 

Hence we have shown that $\theta$ stabilizes $W:=L^\perp\cap {L^\theta}^\perp$, in which $\theta$ induces by induction, in the orthogonal case, the product $\sigma$ of $i-1$ pairwise orthogonal long root elations, say with axis the $(2i-3)$-dimensional singular subspace $A_W$, and in the proper symplectic case, an $(I_{2n-2-2i},-I_{2i-2})$-homology $\sigma$, say with axes the $(2n-2i-3)$-dimensional singular subspace $A_W$ and the $(2i-3)$-dimensional singular subspace $A'_W$. In any case, $W$ contains a line $L^*$ which is mapped onto an opposite. 
Playing the same game with $L^*$ as with $L$ above, and noting that, in the orthogonal case, the hyperbolic quadric $H$ defined by $L$ and $L^\theta$ is stabilized, and in the (proper) symplectic case the union $L\cup L^\theta$ is stabilized (as these are the unique lines intersecting all of the stabilized lines $\<x,x^\theta\>$, with $x\in L$) we see that, in the orthogonal case, $\theta$ induces  an axial elation $\sigma'$ on $H$, say with axis the line $A$, or a unique $(I_2,-I_2)$-homology $\sigma'$ in the nondegenerate symplectic space induced in the span of $L$ and $L^\theta$, say with axes the lines $A$ and~$A'$. 

Now we first treat the orthogonal case. Let $\sigma'$ be the natural extension of $\sigma$ to $\Delta$ (that is, $\sigma'$ is the product of the $(i-1)$ perpendicular long root elations generating $\sigma$, naturally extended to $\Delta$; note that every long root elation of $W$ extends uniquely to a long root elation of $\Delta$ by considering the same axis and the same pair of corresponding lines). Let $\sigma''$ be the long root elation with axis $A$ mapping $L$ to $L^\theta$. We claim that $\sigma'\sigma''$ restricted to $W^*:=L^*\cap {L*}^\perp$ (denote it by $\sigma^*$) coincides with $\theta$ restricted to $W^*$ (denote it by $\theta^*$). Indeed, we know that $\theta^*$ is the product of at most $(i-1)$ perpendicular long root elations. consequently the set of fixed points is a singular subspace, which must hence contain $A$ and  $A_W\cap W^*$. But these generate a singular subspace $A^*$ of dimension $2i-3$, hence $\sigma^*$ is the product of precisely $(i-1)$ perpendicular long root elations and has axis $B$. We can multiply $\sigma'\sigma''$ with the root elation $\rho$ with axis $\<L^*,{L^*}^\theta\>\cap A_W$, which is a line, mapping ${L^*}^\theta$ to $L^*$. Since $\rho$ induces the identity on $L^*\cap {L^*}^\perp$, we deduce that also $\sigma^*$ is the product of at most $(i-1)$ perpendicular long root elations. But its fixed point set also contains $B$ and hence $B$ is the axis of $\sigma^*$. Since $\sigma^*$ and $\theta^*$ agree over $W\cap W^*$ and both map $L$ to $L^\theta$, the claim is proved. Hence $\sigma'\sigma''$ coincides with $\theta$ over $W\cup W^*$. By \cite[Proposition~5.4]{PVM:21} we obtain $\sigma'\sigma''=\theta$ and $\theta$ is the product of $i$ perpendicular long root elations, as claimed. 

Now we treat the proper symplectic case. Here, it is more convenient to work algebraically. We can choose the standard alternating bilinear form so that it also has standard form in $W$ and in $H$. Then it follows from the foregoing that there exists $k\in \K$ such that the matrix of $\theta$ is diagonal with on the diagonal $2n-2i-2$ times 1, $2i-2$ times $-1$, two times $k$ and two times $-k$. But clearly, in order $\theta$ to be an automorphism of $\Delta$, we should have $k\in\{1,-1\}$ and the assertion follows. 
\end{proof}

We now extend Theorem~\ref{BCD} in the symplectic case to diagrams $\mathsf{C}_{2n;n}^2$, thereby completing the proof of Theorem~\ref{thm1:middle}. Note that domestic collineations with that opposition diagram exist which do not fix any chamber. But if we require at least one fixed point, then we have a well defined involutive homology.

\begin{thm}\label{symplectichomology}
Let $\theta$ be a point-domestic collineation of a large proper symplectic polar space $\Delta$ with opposition diagram $\mathsf{C}_{n,i}^2$, $2\leq 2i\leq n$, fixing at least one point if $n=2i$.  Then $\theta$ is  an $(I_{2n-2i},-I_{2i})$-homology.
\end{thm}

\begin{proof}
If $2i<n$, then this is proved in Theorem~\ref{BCD}, so assume from now on that $n=2i$. We proceed by induction on $n$. For $n=2$, the theorem follows from Theorem~\ref{copolar}. Now assume $n>1$. 

Let $x$ be a fixed point. It is easy to see that point-domesticity implies that the map $\theta_{(x)}$ induced by $\theta$ in $\Res_\delta(x)$ is also point-domestic.  We claim that some plane $\pi$ through $x$ is such that $\pi,\pi^\theta$ is an opposite pair in $\Res_\Delta(x)$. Indeed, suppose not. Then $\theta_{(x)}$ is point- and line-domestic and hence the identity. Let $X$ be a non-domestic $3$-space (which exists by the opposition diagram). Then it is easy to see that $x^\perp\cap X$ intersects $x^\perp\cap X^\theta$ as $x$ and $x^\perp\cap X$ span a singular subspace of dimension $1+\dim(x^\perp\cap X)$ which is moreover fixed under the action of $\theta$. This is a contradiction, which proves the claim. 

Now let $L$ be a line in $\pi$ not through $x$. Then $L^\theta$ is opposite $L$ in $\Delta$. The map $\theta_L$ fixes the plane spanned by $L$ and $x$ and therefore induction implies that it is either the ideintity, or an $(I_{n-2},-I_{n-2})$-homology. Either way, similarly as in the proof of Theorem~\ref{BCD} one now shows that $\Delta_L:=L^\perp\cap (L^\theta)^\perp$ is stabilized by $\theta$. Suppose $\Delta_L$ is fixed pointwise (in other words, $\theta_L$ is the identity). Then every singular subspace of dimension at least $4$ is domestic, a contradiction if $n\geq 6$. So if $n\geq 6$, then $\theta$ induces an $(I_{n-2},-I_{n-2})$-homology in $\Delta_L$. Pick, in this case, a non-domestic line $M$ in $\Delta_L$. Considering an embedding of $\Delta$ in $\PG(2n-1,\K)$, we see that $M^\perp\cap\Delta_L$ has dimension $2n-7$ and hence intersects the pointwise fixed subspaces in $\Delta_L$, which have dimension $n-3$, nontrivially if $(2n-7)+(n-3)-(2n-4)\geq 0$; hence if $n\geq 6$.  Consequently, if $n\geq 6$, then $\theta_M$ has a fixed point and we can again apply induction and show that $\Delta_M:=M^\perp\cap(M^\theta)^\perp$ is fixed; moreover $\theta$ induces an $(I_{n-2},-I_{n-2})$-homology in $\Delta_M$. The homologies induced in $\Delta_L$ and $\Delta_M$ agree on the non-empty intersection (which has dimension at least $3$), and since $\Delta_L$ and $\Delta_M$ span the whole space,  $\theta$ is itself an $(I_{n},-I_n)$-homology.  

We are left with the case $n=4$. Assume for a contradiction that $\Delta_L$ is fixed pointwise. Then $\Delta_L^\perp$ is a symplectic quadrangle stabilized by $\theta$ and Theorem~\ref{copolar} implies that  $\theta$ either fixes a spread $\mathcal{S}$ or is an $(I_2,-I_2)$-homology. In the former case, let $S\in\mathcal{S}$ be a line of the spread and $p$ any point of $\Delta_L$. Then the singular plane $\pi$ spanned by $S$ and $p$ is fixed by $\theta$ whereas no point on $L$ is fixed. But for an arbitrary point $q\in\pi\setminus(S\cup\{p\})$ we have $q\neq q^\theta$ and $qq^\theta$ is fixed, implying $S\cap qq^\theta$ is fixed, a contradiction. Hence $\theta$ induces an $(I_2,-I_2)$-homology in $\Delta_L^\perp$, say with axes $A$ and $A'$. Now Propositions~\ref{pdlinefixed} and~\ref{projsp}$(iii)$ imply that either $\<A,\Delta_L\>$ or $\<A',\Delta_L\>$ are fixed pointwise. This means that $\theta$ is an $(I_6,-I_2)$-homology and hence has opposition diagram $\mathsf{C_{4;1}^2}$, a contradiction. The claim is proved. 

Hence $\theta$ induces an $(I_2,-I_2)$-homology in $\Delta_L$, say with corresponding axes $A,A'$. Let again $M$ be a line of $\Delta_L$ mapped onto an opposite. Then $\Delta_M=\Delta_L^\perp$ and interchanging the roles of $L$ and $M$ we see that $\theta$ induces an  $(I_2,-I_2)$-homology in $\Delta_M$, say with axes $B,B'$. Then  Propositions~\ref{pdlinefixed} and~\ref{projsp}$(iii)$ imply that either $\<A,B\>$ and $\<A',B'\>$ are fixed pointwise, or $\<A,B'\>$ and $\<A',B\>$ are. Either way, we obtain an $(I_4,-I_4)$-homology. 
\end{proof}
%
%\hvm{Move the following to the last section, or delete it??} The previous result implies that an opposition diagram of type $\mathsf{C}_n$ is polar closed if and only if it is the opposition diagram of a conjugate of a member of $U^+$, where $B=U^+.H$, with $H$ the torus and $B$ a Borel.  Together with the foregoing results we now have the same result for all spherical buildings, given that James shows the converse statement, i.e., every classical polar closed opposition diagram can be realized by a unipotent element. 
%
%Our next goal is to take a closer look at all point-domestic collineations in arbitrary polar spaces. First when there is at least one fixed point. 
%

\subsubsection{General case of class II collineations}

%A collineation of a polar space of rank $n$ is called \emph{(sub)max-domestic} if no (sub)maximal singular subspace is mapped onto an opposite. 

Since the rank 3 situation is completely known, we may restrict to the rank $\geq 4$ case. We have the following starting point.

\begin{prop}\label{startingpoint}
Let $\Delta=(X,\Omega)$ be a polar space of rank $n\ge4$. Let $\theta$ be a point-domestic collineation of $\Delta$. Then all planes of $\Delta$ are Pappian, i.e., $\Delta$ is defined over a field, and $\theta$ has as companion field automorphism the one related to the corresponding Hermitian form (hence trivial for symplectic and orthogonal polar spaces, and an involution for Hermitian polar spaces). If $\theta$ is point-domestic, $(n-1)$-domestic and $(n-2)$-domestic, then $\Delta$ is always a symplectic or orthogonal polar space. 
\end{prop}

\begin{proof}
The opposition diagram tells us that there is some singular subspace $\Sigma$ of dimension~3 mapped to an opposite. Then $\theta_\Sigma$ is a domestic duality, which by largeness is a symplectic polarity (anyway, if the polar space would be small, then it is also defined over a field). These only exist in Pappian projective spaces. Now pick an arbitrary line $L$ in $\Sigma$ that is not fixed by the symplectic polarity. Then $L$ and $L^\theta$ are opposite, and the $3$-dimensional subspace $W$ in the ambient projective space of $\Delta$, generated by $L$ and $L^\theta$ inersects $\Delta$ in a generalized quadrangle. It follows directly from the proof of Theorem~8.5.11 in \cite{HVM:98} that the map $\kappa:L^\theta\to L:x\mapsto x^\perp\cap L$ is the restriction of a semi-liner mapping with companion field automorphism the involution $\sigma$ corresponding to the Hermitian form defining $\Delta$ (and hence trivial for symplectic and orthogonal polar spaces). Let $\sigma'$ be the companion field automorphism of $\theta$. Let $L'$ be the projection of $L^\theta$ onto $\Sigma$. Then, for each point $x\in L$, we have $x^{\theta\kappa}=L\cap (x^\theta)^\perp=L\cap x^{\theta_\Sigma}=x$, since $\theta_\Sigma$ is symplectic and hence $x\in x^{\theta_\Sigma}$. Since the companion field automorphism of $\theta\kappa$ is $\sigma'\sigma$, which must be the identity in view of $x^{\theta\kappa}=x$, for all $x\in L$, we deduce $\sigma=\sigma'$. % Note also that this also implies that $\theta$ preserves cross-ratio if $\Delta$ is symplectic or ortogonal, and is semi-linear if $\Delta$ is Hermitian. This follows directly from the proof of Dienst's result  \cite{Die:80}.\james{H to expand}

Left to show is the last assertion. Suppose, to that end, that $\theta$ is also $(n-1)$-domestic and $(n-2)$-domestic. Theorem~6.1 of \cite{TTM:12} implies that $\theta$ pointwise fixes a subspace of corank at most $n-2$. Since such a subspace always contans a line, the companion field automorphisn of $\theta$ is the identity, and so $\Delta$ is not Hermitian. % By Proposition~\ref{lem:typeI}, $\theta$ fixes at least one point. Let $\Omega$ be a singular subspace of second largest dimension in a flag of maximum size that is mapped onto an opposite (if $\Omega$ does not exist, then the opposition diagram of $\theta$ is polar and the result follows from Proposition~\ref{axialbase}). Then $\theta_\Omega$ has (dual) polar opposition diagram and the other statements follow from Proposition~\ref{axialbase}, and the fact that a subspace of corank at most $n-2$ always contains a line. 
\end{proof}

From the previous proposition emerge three cases: the symplectic case, the orthogonal case and the Hermitian case. The symplectic case is treated in Theorem~\ref{symplectichomology}. We now treat the \emph{proper} orthogonal case, that is, orthogonal polar spaces for which the minimal embedding has no secants of size at least 3. We will show that $\theta$ belongs to a conjugate of the group of collineations pointwise fixing the set $E_1(C)$ (in Tits' notation~\cite{Tit:74}). 

\begin{lemma}\label{orthpointfix}
A point-domestic collineation $\theta$ of a (strictly) orthogonal polar space $\Delta$ of rank $n\geq3$ with a fixed point fixes a chamber $C$ and all adjacent chambers. In particular, this happens if $\theta$ is both point-domestic and $(n-1)$-domestic, or if $\theta$ is point-domestic and $n$ is odd. 
\end{lemma}

\begin{proof} Let $x$ be a fixed point of $\theta$. We claim that $\theta$ fixes a maximal singular subspace containing~$x$. Indeed, if $n=2$, then this follows immediately from the fact that $\theta$ is point-domestic. If $n>2$, then by induction, it suffices to show that some line through $x$ is fixed by $\theta$. Let $L$ be an arbitrary line through $x$ and assume that $L$ is not fixed. Then by Proposition~\ref{pdlinefixed} the plane $\pi$ spanned by $L$ and $L^\theta$ is fixed. Combined with Proposition~\ref{projsp} this yields a fixed line through $x$ in $\pi$. The claim follows. 

So let $U$ be a maximal singular subspace fixed by $\theta$, with $x\in U$. We can again apply Proposition~\ref{projsp}. Since we know that $\theta$ is linear (see  Proposition~\ref{startingpoint}), and we have at least one fixed point (by assumption, namely $x$), there are only two possibilities for the fixed point structure: (1) Either $\theta$ pointwise fixes exactly two complementary subspaces $W_1$ and $W_2$, or (2) $\theta$ induces an elation in $U$ (possibly a trivial one). 

Suppose first (1) that $\theta$ pointwise fixes two complementary subspaces $W_1$ and $W_2$. Pick hyperplanes $V_1\subseteq W_1$ and $V_2\subseteq W_2$ and consider the action of $\theta$ in the residue $R$ of the subspace $V$ of dimension $n-3$ spanned by $V_1$ and $V_2$.  Note that $R$ is an orthogonal generalized quadrangle. Then $\theta$ fixes $\<V_1,W_1\>$ and $\<V_2,W_1\>$, which are points on the line $U$ of $R$, and no other points on that line are fixed. But $\theta$ induces a point-domestic collineation in $R$. Since such collineations pointwise fix a dual geometic hyperplane, the fixed point structure in $R$ is an ideal subquadrangle. But proper orthogonal quadrangles are line-minimal (for terminology and this result, see Section~5.9 of \cite{HVM:98}). This contradiction shows we always have case (2). 

Hence $\theta$ induces an elation in $U$. So we can select a chamber $C'$ in $U$ such that $\theta$ pointwise fixed $E_1(C')$. Set $C=C'\cup\{U\}$. Then $\theta$ fixes $C$ and all $i$-adjacent chambers of $C$ with $1\leq i\leq n-1$. There remains to show that $\theta$ also fixes all $n$-adjacent chambers. Consider the singular subspace $\Sigma$ of dimension $n-3$ in $C'$. Then the induced action of $\theta$ on the residue of $\Sigma$, which is a generalized quadrangle, fixes the point corresponding to the singular subspace $\Sigma'$ of dimension $n-2$ in $C'$. By point-domesticiy, it fixes all lines through that point, i.e., $\theta$ fixes all maximal singular subspaces containing $\Sigma'$. Hence $\theta$ fixes all chambers $n$-adjacent to $C$. 

If $\theta$ is $(n-1)$-domestic, than it fixes a point by Lemma~\ref{lem:typeI}. If $n$ is odd, then the opposition diagram of $\theta$ shows that $\theta$ is $(n-1)$-domestic. 
\end{proof}

In particular the previous lemma holds for all hyperbolic polar spaces. 

Now we investigate the Hermitian case.
A Baer-subspace of a projective space $\Sigma$ has the following two properties: (1) every hyperplane of $\Sigma$ is incident with at least one point of $\Sigma$ and (2) every point of $\Sigma$ is on some line of $\Sigma'$. Let $\Delta'$ be a rank $n$ sub polar space of the polar space $\Delta$ of rank $n$ with the following properties. \begin{compactenum}\item[(Pr1)] Every maximal singular subspace of $\Delta'$ is a Baer subspace of some maximal singular subspace of $\Delta$.\item[(Pr2)] Every maximal singular subspace of $\Delta$ containing a submaximal singular subspace of $\Delta'$ belongs to $\Delta'$. \end{compactenum}Then we call $\Delta'$ an \emph{ideal Baer sub polar space} of $\Delta$. This generalizes the definition of an ideal Baer sub polar space of rank 3 in \cite{TTM:12}. An involution that fixes an ideal Baer sub polar space of $\Delta$ and nothing more is called an \emph{ideal Baer involution}. %Wewillnowshowthat,withtheabovenotation,everyplaneof?meets??inatleastonepoint.Lemma7.1If??isanidealBaersubpolarspaceof?,theneveryplaneof?hasapointincommonwith??.ProofLet?beanarbitraryplaneof?,andletxbeanarbitrarypointof??.Wemayassumethatx/??.ConsideralineLof??throughx.Then?:=proj?L?isaplaneof?throughL.SinceallplanesthroughLbelongto??,weseethat?belongsto??,anditmeets?nontrivially(within?).Sincethepointsof??in?formaBaersubplane,thereisalineMin?belongingto??andintersecting?(intersectionpointin?).Now,sinceallplanesthroughMbelongto??,thereisatleastonesuchplanemeeting?inalineof?.Butsinceplanesof??areBaersubplanesofplanesof?,thatlin

\begin{prop}\label{hermpointfix}
A point-domestic collineation $\theta$ of a Hermitian polar space $\Delta$ of rank $n\geq 3$ with a fixed point is an ideal Baer involution and has opposition diagram $\mathsf{B}_{n;m}^2$, with $2m\in\{n,n+1\}$. 
\end{prop}

\begin{proof}
Let $x$ be any fixed point. Completely similar to the first paragraph in the proof of Lemma~\ref{orthpointfix} one can show that there exists a maximal singular subspace $U$ through $x$, fixed under the action of $\theta$. Now Propositions~\ref{projsp} and~\ref{startingpoint} imply that the fixed point structure of $\theta$ in $U$ is a Baer-subspace $\Sigma'$, and the restriction to $U$ of $\theta$ is an involution.   Let $W$ be a hyperplane of $U$ spanned by a hyperplane of $\Sigma'$. Then, by point-domesticity, all maximal singular subspaces of $\Delta$ containing $W$ are fixed by $\theta$, and so they all contain further points fixed under $\theta$. This shows that for each fixed point $x$ there exist many opposite fixed points. So the fixed point structure is a nondegenerate polar space, which is easily seen to be an ideal Baer sub polar space. 

Clearly $\theta^2$ is linear and point-domestic (by Corollary~\ref{square}) and so the identity by Proposition~\ref{startingpoint}. Hence $\theta$ is an involution and the lemma is proved (the opposition diagrams follow from the last assertion of Proposition~\ref{startingpoint}). 
\end{proof}

\begin{example} Let $\Delta$ be the rank $n$ polar space corresponding to the Hermitian variety with equation $\sum_{i=1}^nX_{-i}\overline{X}_i-X_i\overline{X}_{-i}=0$, in $\PG(2n-1,\K)$, where $x\mapsto\overline{x}$ is the corresponding field involution, $\K$ commutative. Let $\theta$ be the collineation defined by mapping a point with coordinates $(x_{-n},x_{-n+1},\ldots,x_{-1},x_1,\ldots,x_n)$ to  $(\overline{x}_{-n},\overline{x}_{-n+1},\ldots,\overline{x}_{-1},\overline{x}_1,\ldots,\overline{x}_n)$. If we denote by $\FF$ the fixed field of the field involution, then the fixed points in $\Delta$ are precisely all points of a Baer subspace $\PG(2n-1,\FF)$, hence $\theta$ pointwise fixes a symplectic polar space of rank $n$. The collineation is point-domestic as one easily calculates.   
\end{example}

We can now turn our attention to collineations of class III. 
%\newpage
\subsection{Class III automorphisms}

If $\theta$ is a class III automorphism of a polar space of rank~$n$ then, by Corollaries~\ref{oddrankchamber} and~\ref{oppdiafull}, the rank $n$ is even, and $\Diag(\theta)=\sX_{n;n/2}^2$ (where $\sX\in\{\sB,\sC,\sD\}$).

%We note that Theorem~\ref{thm:typeI} is also applicable to automorphisms with diagram $\mathsf{X}_{2n+1;n}^2$ (with $\sX\in\{\sB,\sC,\sD\}$). Thus we have the following.

%\begin{cor}
%Let $\Delta$ be a thick building of type $\sX_n$ with $\sX\in\{\sB,\sC,\sD\}$ and let $\Pi=(\cP,\Omega)$ be the associated polar space. If a collineation $\theta$ has opposition diagram $\sX_{n,n-1}^2$ (necessarily $n$ is odd) then $\theta$ fixes pointwise a geometric subspace or rank~$1$ (that is, a line). 
%\end{cor}

%\newpage

%\section{Polar spaces III: Point domestic diagrams}

%This is the main part of the paper. All opposition diagrams are with a 2 in the exponent. Here we know more. 

%\subsubsection{Domestic collineations without any fixed point} 

%As noted before in Corollary~\ref{oddrankchamber}, domestic collineations without fixed points can only occur in polar spaces of even rank $n=2m$. The opposition diagram is automatically either $\mathsf{B}_{2m;m}^2$, $\mathsf{C}_{2m;m}^2$ or $\mathsf{D}_{2m;m}^2$. 

First we need to be a little more precise in Case~$(i)$ of Proposition~\ref{projsp}.

\begin{prop}\label{spreadalgebra}
Let $\Sigma$ be a finite-dimensional projective space over a field $\K$ with dimension at least $2$ and let $\theta$ be a fixed point free collineation with the property that the line $\<x,x^\theta\>$ is fixed for every point $x$. Then the projective dimension of $\Sigma$ is odd, say $2n+1$, and the fixed subspaces form a projective space of dimension $n$ either over a quadratic (not necessarily separable) extension $\mathbb{L}$ of $\K$, or over a quaternion division algebra $\mathbb{H}$ which naturally contains $\K$ as a $2$-dimensional subalgebra over its center. 
\end{prop}

\begin{proof}
It is clear that the fixed point structure is a projective space $\frak{P}$ of dimension $n$. In order to determine the coordinatizing division ring, we consider a fixed 3-dimensional subspace, i.e., we consider the case $n=1$.  Applying Klein correspondence, $\theta$ acts on a hyperbolic quadric $Q$ in $\PG(5,\K)$ pointwise fixing an ovoid $O$. 

If $\theta$ is linear, then any triple of points of $O$ determine a conic on $Q$ which is fixed pointwise (by linearity). Considering four fixed points no three on a conic we see that $\theta$ pointwise fixes a 3-dimensional subspace $\Sigma$ of $\PG(5,\K)$ intersecting $Q$ in $O$. Hence $O$ is an elliptic quadric in $\Sigma$ (that is, a nondegenerate nonruled quadric). Since the stabilizer of $O$ induces a group containing $\PGL_2(\LL)$, for some quadratic extension $\LL$ of $\K$, and since the group of projectivities of a line of $\frak{P}$ is inherited from the collineation group of a 3-dimensional subspace of $\Sigma$ fixed by $\theta$, we deduce that $\frak{P}$ is defined over $\LL$. 

Now suppose that $\theta$ is not linear. Note that $\theta$ does not map any point to a collinear one. We show that $\theta$ is an involution. Let $x$ be an arbitrary point. If $x$ is fixed by $\theta$, then $x$ is also fixed by $\theta^2$. Now suppose that $x^\theta$ is opposite $x$. Let $\pi$ be any singular plane through $x$. Since $\theta$ preserves types, $\pi^\theta$ intersects $\pi$ in a unique point $p$. Clearly $p^\theta\in\pi^\theta\ni p$, hence by our remark above, $p=p^\theta$. Hence the set of fixed points in $p^\perp\cap(p^\theta)^\perp$ forms an ovoid $C$, which cannot be a plane conic as $\theta$ is not linear. Hence $C$ spans $p^\perp\cap(p^\theta)^\perp$ and so $\theta$ preserves $p^\perp\cap(p^\theta)^\perp$. Hence $\theta$ also preserves the perp of that set, which is $\{p,p^\theta\}$, implying that $p^{\theta^2}=p$. This shows that $\theta$ is an involution.

So back to $\frak{P}$, $\theta$ acts as a Galois group and hence $\frak{P}$ arises from Galois descent, implying that the coordinatising structure is a quaternion division algebra. This can also be seen in an elementary way: denote by $\K\rightarrow\K:k\mapsto \overline{k}$ the companion field involution, then coordinates in the case $n=1$ can be chosen such that $\theta$ maps $(x_1,x_2,x_3,x_4)$ to $(\overline{x}_2,a\overline{x}_1,\overline{x}_4,a\overline{x}_3)$, with $a=a^\theta\notin\{x\overline{x}:x\in\K\}$. This determines the quaternion division ring where multiplication is given by $(x,y)\cdot(u,v)=(xu+a\overline{y}v,yu+\overline{x}v)$.  
\end{proof}

An immediate corollary is Theorem~\ref{thm1:BCD1}:

\begin{cor}\label{corclassIII}
Let $\Delta=(X,\Omega)$ be a polar space of rank $n\ge4$ defined over the field $\K$. Let $\theta$ be a point-domestic collineation of $\Delta$ with any fixed points. Let $Y$ be the set of fixed lines, and let $\Upsilon$ be the set of fixed singular subspaces. Then $\Gamma=(Y,\Upsilon)$ is a polar space defined over either a quadratic extension of $\K$, or over a quaternion division algebra $\HH$ which is $2$-dimensional over $\K$ (and $\K$ is $2$-dimensional over he centre of $\HH$). 
\end{cor}

Now we first handle the split case (Theorem~\ref{thm1:pointdomnofixedpoint} ). Recall the definition of a \emph{minimal} Hermitian polar space of rank $n$: one with ambient projective space of dimension $2n-1$.  A \emph{mixed} polar space (of rank $n$) is a polar space which is a proper subspace of a symplectic polar space of rank $n$. Here, the characteristic is necessarily equal to 2. 

\begin{thm}\label{splitcasewithoutfixedpoints}
Let $\Delta=(X,\Omega)$ be a split polar space of rank $n\ge4$ defined over the field $\K$. Let $\theta$ be a point-domestic collineation of $\Delta$ without any fixed point. Then $n$ is even, say $n=2m$ and the opposition diagram is either $\mathsf{C}_{n,m}^2$ or $\mathsf{D}_{n,m}^2$. Moreover, $\Delta$ is either \begin{compactenum}[$-$]\item a symplectic polar space, and the fixed point structure of $\theta$ is a symplectic polar space over a quadratic extension of $\K$; or \item a hyperbolic orthogonal polar space, and the fixed point structure of $\theta$ is either a minimal Hermitian polar space, or a mixed polar space. \end{compactenum} 
\end{thm}

\begin{proof}
The symplectic case immediately follows from Propositon~\ref{spreadalgebra} and the fact that symplectic polar spaces are the only embedded polar spaces with the property that all points of the ambient projective space belong to the polar space.

In the hyperbolic case it suffices to consider rank $n=4$. Applying triality, we may assume that $\theta$ has opposition diagram $\mathsf{D_{4;2}^1}$, in which case the fixed point structure is a subspace of corank 2, hence an orthogonal generalized quadrangle naturally embedded in projective $5$-space and, by Remark~3.4.10 of \cite{HVM:98}, either a mixed quadrangle, or the dual of a minimal Hermitian quadrangle.

In the  parabolic case, we consider two opposite maximal singular subspaces fixed by $\theta$. Then $\theta$ stabilizes the hyperbolic polar subspace $H$ generated by those. A fixed line not contained in $H$ intersects $H$ in a point (as $H$ is a geometric hyperplane), which must subsequently be fixed, yielding a contradiction. 
\end{proof}

Note that the previous theorem says that no parabolic polar space admits a point-domestic collineation without fixed points. Hence in the symplectic case, the defining field cannot be perfect when it has characteristic 2. In fact we can even be more specific concerning the collineations and the said quadratic  extensions. The next proposition is Theorem~\ref{thm1:moreprecise1}.

\begin{prop}\label{morepprecise1a}
Let $\Delta=(X,\Omega)$ be a symplectic polar space of rank $n\ge4$ defined over the field $\K$. Let $\theta$ be a point-domestic collineation of $\Delta$ without any fixed point. Let $\LL$ be the quadratic extension of $\K$ over which the fixed points structure of $\theta$ is defined according to \emph{Theorem~\ref{splitcasewithoutfixedpoints}}. Then $\theta$ is an involution and either \begin{compactenum}[$-$]\item the characteristic of $\K$ is not $2$ and $\LL$ is a separable extension of $\K$, or \item the characteristic of $\K$ is $2$ and $\LL$ is an inseparable extension of $\K$.  
\end{compactenum}
Conversely, if $\K$ admits a quadratic extension $\LL/\K$, inseparable if the characteristic of $\K$ is $2$, then $\Delta$ admits a point-domestic collineation without fixed points as above. 
\end{prop}

\begin{proof}
Clearly it suffices to show that, if in a sympleclic quadrangle over $\K$ a spread is fixed by a fixed point free collineation $\theta$, then $\theta$ is an involution and the spread corresponds to an inseparable extension if the characteristic of $\K$ is $2$.

Let the symplectic quadrangle $\Delta$ be defined by the standard symplectic form $x_0y_1-x_1y_0+x_2y_3-x_3y_2$. Without loss of generality we may assume that $\theta$ (which has trivial companion field automorphism by Proposition~\ref{startingpoint}) maps the point $(1,0,0,0)$ to $(0,0,0,1)$ and $(0,0,1,0)$ to $(0,1,0,0)$. Hence $(0,1,0,0)$ is mapped to $(0,a,b,0)$, and $(0,0,0,1)$ to $(c,0,0,d)$, yielding the matrix $$\left(\begin{array}{cccc}0&0&0&c\\      0&a&e&0\\        0&b&0&0\\       1&0&0&d\end{array}\right),$$ for some $a,b,c,d,e\in\K$.  Expressing that $\theta$ maps each point $(x_0,x_1,x_2,x_3)$ to a collinear one, we obtain 
%$x_0y_1-x_1y_0+x_2y_3-x_3y_2=0$ if and only if
%$$cx_3(ay_1+y_2)-(ax_1+x_2)cy_3+bx_1(y_0+dy_3)-(x_0+dx_3)by_1=0.$$
$$cx_3x_1-(ax_1+ex_2)x_0+bx_1x_3-(x_0+dx_3)x_2\equiv 0.$$ This yields $e=-1$, $c=-b$ and $a=d=0$. 
 
 Since $\theta$ has no fixed points, $b$ is a nonsquare and we see that $\LL=\K(\sqrt{b})$, which proves the assertions (the converse is obvious in view of our explicit form of $\theta$). .  
\end{proof}

The hyperbolic case is even more canonical. This is Theorem~\ref{thm1:moreprecise2} from the introduction.

\begin{prop}\label{hyperbolicpointdomestic}
Let $\Delta=(X,\Omega)$ be a hyperbolic polar space of rank $n\ge4$ defined over the field $\K$. Let $\theta$ be a point-domestic collineation of $\Delta$ without any fixed point. Let $\PG(2n-1,\K)$ be the ambient projective space. Then $\theta$ naturally extends to $\PG(2n-1,\K)$ poinwise fixing a spread $\cS$ which defines a projective space $\PG(n-1,\LL)$ over a quadratic extension $\LL$ of $\K$. Every collineation of $\PG(2n-1,\K)$ with pointwise fixes $\cS$ is a point-domestic collineation of $
\Delta$ without fixed points. If $\LL/\K$ is separable, then the fixed polar space is minimal Hermitian, if $\LL/\K$ is inseparable, then the fixed polar space is the subspace of the symplectic polar space over $\LL$ obtained by restricting the long root elations to $\K$.

Conversely, if $\K$ admits a quadratic extension $\LL/\K$, then there exists a point-domestic collineation of $\Delta$ without fixed points as above.  
\end{prop} 

\begin{proof}
It again suffices to consider the case $n=4$. We order coordinates as in $$(x_{-4},x_{-3},x_{-2},x_{-1},x_1,x_2,x_3,x_4),$$ take $X$ to be the set of points with equation $X_{-4}X_4-X_{-3}X_3+X_{-2}X_2-X_{-1}X_1=0$, and assume that $\theta$ stabilizes the lines $\<p_{-4},p_{-3}\>, \<p_4,p_3\>, \<p_{-2},p_{-1}\>, \<p_2p_1\>$, where $p_i$ is the point with a $1$ in position $i$ and $0$ elsewhere. Set   $$M_i=\left(\begin{array}{cc}a_i&b_i\\      c_i&d_i\end{array}\right),$$ $a_i,b_i,c_i,d_i\in\K$, $i=-2,-1,1,2$. Then we can assume that $\theta$ has block matrix $$\left(\begin{array}{cccc}M_{-2}&0&0&0\\      0&M_{-1}&0&0\\        0&0&M_1&0\\       0&0&0&M_2\end{array}\right).$$ (Recall from Proposition~\ref{startingpoint} that $\theta$ is linear; action of $\theta$ on the left)

Since $\theta$ has no fixed points, we have $b_i\neq 0\neq c_i$, for all $i\in\{-2,-1,1,2\}$. Expressing that $\theta$ preserves the quadric (hence the quadratic form up to a scalar multiple), we deduce
that $\det A_{-2}=\det A_{-1}=\det A_1=\det A_2:=d$. Expressing that $\theta$ maps each point to a collinear one we deduce $\tr A_{-2}=\tr A_{-1}=\tr A_1=\tr A_2:=t$. We also have $A_{-i}=A_i$ for $i=1,2$. Moreover, since there are no fixed points in $\Delta$, the equation $x^2-tx+d=0$ has no solution in $\K$ and determines a quadratic field extension $\LL/\K$. It follows that $\theta$ pointwise fixes a spread $\cS$ of $\PG(7,\K)$. Clearly, every collineation of $\PG(7,\K)$ fixing that spread pointwise restricts to a point-domestic collineation of $\Delta$. 

Let $\delta$ be a solution in $\LL$ of the equation $x^2-tx+d=0$ above.  Then we can write every element of $\LL$ as $a+b\delta$, with $a,b\in\K$ and so with $\delta^2 =t\delta-d$. Now we identify each point $p=(x_{-4},x_{-3},\ldots, x_3,x_4)$ of $\PG(7,\K)$ with the point $p^\beta=(c_2x_{-4}-a_2x_{-3}+x_{-3}\delta,c_1x_{-2}-a_1x_{-1}+x_{-1}\delta, c_1x_1-a_1x_2+x_2\delta, c_2x_3-a_2x_4+x_4\delta)$ of $\PG(3,\LL)$. An elementary computation shows that the image of a point under the action of $\theta$ is identified with the same point of $\PG(3,\LL)$ (the coordinates are equal up to the scalar $\delta$). Hence we can identify the fixed lines of $\PG(7,\K)$ with the points of $\PG(3,\LL)$. Denote the conjugate $u+tv-v\delta$ of the element $u+v\delta\in\LL$ by $\overline{u+v\delta}$, where $u,v\in\K$. Then one calculates that a point $p$ of $\PG(7,\K)$ belongs to $\Delta$ if and only if the coordinates $(y_{-2},y_{-1},y_1,y_2)$ of $p^\beta$ satisfy the relation $c_1Y_{-2}\overline{Y}_2+c_2Y_{-1}\overline{Y}_1\in\K$. 

This proves the proposition (noting that the converse is again obvious by our explicit form of $\theta$).
\end{proof}

We now consider non-split polar spaces. The neatest situation occurs with the minimal Hermitian polar spaces. The following is Theorem~\ref{thm1:nonsplit}. 

\begin{prop}\label{hermitianpointdomestic}
Let $\Delta=(X,\Omega)$ be a minimal Hermitian polar space of rank $n\ge4$ defined over the field $\K$, with corresponding field involution $\sigma$ and field extension $\K/\FF$. Let $\theta$ be a point-domestic collineation of $\Delta$ without any fixed point. Let $\PG(2n-1,\K)$ be the ambient projective space. Then $\theta$ is an involution and naturally extends to $\PG(2n-1,\K)$ pointwise fixing a spread $\cS$ which defines a projective space $\PG(n-1,\HH)$ over a quaternion division algebra $\HH$ over its centre $\FF$, with $\K$ a $2$dimensional subalgebra of $\HH$. The only nontrivial collineation of $\Delta$ fixing every member of $\cS$ in $\Delta$ is $\theta$. The fixed polar space is a minimal quaternion polar space related to a nonstandard involution of $\HH$ pointwise fixing a $3$-dimensional subalgebra over $\FF$.  

Conversely, if $\K$ admits a quaternion extension $\HH$ such that the centre $\FF$ of $\HH$ is a subfield of $\K$ and $\K/\FF$ is a separable quadratic field extension, then there exists a point-domestic collineation of $\Delta$ without fixed points as above.  
\end{prop} 

\begin{proof}
Note that by Proposition~\ref{startingpoint}, $\theta$ is a involution. Also, if another involution $\theta'$ of $\Delta$ would fix exactly the same lines as $\theta$, then $\theta\theta'$ is a linear point-domestic collineation of $\Delta$, implying the stated uniqueness of $\theta$. We may again only consider the case $n=4$. The arguments run parallel to those of the proof of Proposition~\ref{hyperbolicpointdomestic}.  We order coordinates as in $$(x_{-4},x_{-3},x_{-2},x_{-1},x_1,x_2,x_3,x_4),$$ take $X$ to be the set of points with equation $$X_{-4}X_4^\sigma+X_4X_{-4}^\sigma-X_{-3}X_3^\sigma-X_3X_{-3}^\sigma+X_{-2}X_2^\sigma+X_2X_{-2}^\sigma-X_{-1}X_1^\sigma-X_1X_{-1}^\sigma=0,$$ and assume that $\theta$ stabilizes the lines $\<p_{-4},p_{-3}\>, \<p_4,p_3\>, \<p_{-2},p_{-1}\>, \<p_2p_1\>$, where $p_i$ is the point with a $1$ in position $i$ and $0$ elsewhere. 

Since $\theta$ is an involution we may assume that it interchanges $p_{-4}$ with $p_{-3}$, $p_4$ with $p_3$, $p_{-2}$ with $p_{-1}$, and $p_2$ with $p_1$. Hence it has block matrix
$$\left(\begin{array}{cccc}A_{-2}&0&0&0\\      0&A_{-1}&0&0\\        0&0&A_1&0\\       0&0&0&A_2\end{array}\right),$$ with $$M_i=\left(\begin{array}{cc}0&a_i\\      b_i&0\end{array}\right),$$ $a_i,b_i\in\K$, $i=-2,-1,1,2$. Expressing that $\theta$ preserves $X$ and that $\theta$ maps points of $\Delta$ to collinear points readily implies that $a_{-i}=a_i$, $b_{-i}=b_i$, for $i=1,2$, and $a_1b_1^\sigma=a_2b_2^\sigma=b_1a_1^\sigma=b_2a_2^\sigma=:\ell\in\FF$.  Then $\theta$ has no fixed points in $\Delta$ if and only if $\ell\notin \{kk^\sigma:k\in\K\}$. It is easily checked that this now defines a fixed point free involution in $\PG(7,\K)$, hence pointwise fixing a spread $\cS$. 

We can define the quaternion division algebra $\HH:=\{u+v\delta:(u,v)\in\K\times\K\}$ with ordinary addition and multiplication determined by $\delta v=v^\sigma\delta$ and $\delta^2=\ell$. Clearly $\FF$ is the centre of $\HH$ and $\K$ is a 2-dimensional division subalgebra over $\FF$.

Now we identify each point $p=(x_{-4},x_{-3},\ldots, x_3,x_4)$ of $\PG(7,\K)$ with the point $p^\beta=(a_2^\sigma x_{-4}+x_{-3}\delta,a_1^\sigma x_{-2}+x_{-1}\delta, a_1^\sigma x_1+x_2\delta, a_2^\sigma x_3+x_4\delta)$ of the left projective space $\PG(3,\HH)$. An elementary computation shows that the image of a point under the action of $\theta$ is identified with the same point of $\PG(3,\HH)$ (the coordinates are equal up to the left scalar $\delta$). Hence we can identify the fixed lines of $\PG(7,\K)$ with the points of $\PG(3,\HH)$. Denote the non-standard involution $u-v^\sigma\delta$ of the element $u+v\delta\in\HH$ by $\overline{u+v\delta}$, where $u,v\in\K$. Then one calculates that a point $p$ of $\PG(7,\K)$ belongs to $\Delta$ if and only if the coordinates $(y_{-2},y_{-1},y_1,y_2)$ of $p^\beta$ satisfy the relation $Y_{-2}a_2^{-\sigma}\overline{Y}_2+Y_{-1}a_1^{-\sigma}\overline{Y}_1=Y_{2}a_2^{-\sigma}\overline{Y}_{-2}+Y_{1}a_1^{-\sigma}\overline{Y}_{-1}$. 

This proves the proposition (noting that the converse is again obvious by our explicit form of $\theta$).
\end{proof}

Left to consider are the orthogonal and Hermitian polar spaces whose standard form has nontrival anisotropic forms. These contain hyperbolic and minimal Hermitian polar subspaces fixed under the given point-domestic collineation and so Propositions~\ref{hyperbolicpointdomestic} and~\ref{hermitianpointdomestic} provide necessary conditions for the existence. In general, if a polar subspace of hyperbolic or minimal Hermitian type of an arbitrary polar space admits a point-domestic collineation without fixed points, then it depends on the shape of the anisotropic form whether this collineation can be extended as a point-domestic collineation without fixed points, and the general rule is that the anisotropic form must, up to recoordinatization, be the sum of multiples of the norm form of the corresponding field extension or quaternion division algebra, respectively.  Let us give a few examples for particular fields to illustrate this. 

\subsubsection*{Some special fields}\label{sec:specialfields}
\textbf{Quadrics over the reals.} Consider a real quadric of Witt index $2n$ in a projective space of dimension $4n+2k-1$. Its standard equation is
$$X_{-2n}X_{2n}+X_{-2n+1}X_{2n-1}+\cdots+X_{-1}X_1=X_{0,-k}^2+X_{0,-k+1}^2+\cdots+X_{0,-1}^2+X_{0,1}^2+\cdots+X_{0,k}^2.$$
Let $M$ be a real $2\times 2$ matrix with non-real eigenvalues. By multplying with an appropriate scalar we may assume that its determinant is $1$. Then its trace is of the form $2\cos\varphi$, for 
%[
some $\varphi\in ]0,\pi[$.
%]
Then the collineation defined by $$\left\{\begin{array}{lcl}\left(\begin{array}{c}x_{-2\ell}\\ x_{-2\ell+1}\end{array}\right) & \mapsto &M\cdot \left(\begin{array}{c}x_{-2\ell}\\ x_{-2\ell+1}\end{array}\right), \ell=1,2,\ldots,n,\\
\left(\begin{array}{c}x_{2\ell-1}\\ x_{2\ell}\end{array}\right) & \mapsto & M\cdot \left(\begin{array}{c}x_{2\ell-1}\\ x_{2\ell}\end{array}\right),  \ell=1,2,\ldots,n, \\
\left(\begin{array}{c}x_{0,-\ell}\\ x_{0,\ell}\end{array}\right) & \mapsto & \left(\begin{array}{ll} \cos\varphi & \pm\sin\varphi\\\mp\sin\varphi & \cos\varphi \end{array}\right)\cdot \left(\begin{array}{c}x_{0,-\ell}\\ x_{0,\ell}\end{array}\right), \ell=1,2,\ldots,k\end{array}\right.$$
is point-domestic and fixed point free. The fix point structure is a Hermitian polar space in $\PG(2n+k-1,\CC)$ with equation (and previous conventions) $$Y_{-n}\overline{Y}_{n}-Y_{n}\overline{Y}_{-n}+\cdots +Y_{-1}\overline{Y}_{1}-Y_{1}\overline{Y}_{-1}=Y_{0,k}\overline{Y}_{0,k}+\cdots+Y_{0,1}\overline{Y}_{0,1}.$$

\noindent\textbf{Quadrics over the rationals.} Consider a rational quadric of Witt index $4$ in a projective space $\PG(11,\QQ)$  with equation $$X_{-4}X_{4}+X_{-3}X_{3}+X_{-2}X_2+X_{-1}X_1=X_{0,-2}^2+X_{0,-1}^2+X_{0,1}^2+\ell X_{0,2}^2,$$ with $\ell\in\QQ$. If $\ell$ is a perfect square, then we can set $X_{0,2}'=\sqrt{\ell}X_{0,2}$ and similarly as before we can construct fixed point free point-domestic collineations. However, if $\ell$ is not a perfect square, then $X_{0,-2}^2+\ell X_{0,2}^2$ and $X_{0,-1}^2+X_{0,1}^2$ do not define the norm of the same quadratic extension of $\QQ$, and indeed one can easily show that no fixed point free point-domestic collineation exists. 
\bigskip

\noindent\textbf{Complex Hermitian polar spaces.} Here similar considerations hold as for the real quadrics. Let a Hermitian variety of Witt index $2n$ in the projective space $\PG(4n+2k-1,\CC)$ be given   by the (standard) equation
$$\sum_{\ell=1}^{-2n} X_{-2\ell}\overline{X}_{2\ell}+X_{2\ell}\overline{X}_{-2\ell}-X_{-2\ell+1}\overline{X}_{2\ell-1}-X_{2\ell-1}\overline{X}_{-2\ell+1}=\sum_{\ell=1}^{k} X_{0,-\ell}\overline{X}_{0,-\ell}+X_{0,\ell}\overline{X}_{0,\ell}.$$
Then, for each $r\in\RR^+$, the semi-linear involution defined by $$\left\{\begin{array}{lcl}\left(\begin{array}{c}x_{-2\ell}\\ x_{-2\ell+1}\end{array}\right) & \mapsto & \left(\begin{array}{c}\overline{x}_{-2\ell+1}\\ -r.\overline{x}_{-2\ell}\end{array}\right), \ell=1,2,\ldots,n,\\
\left(\begin{array}{c}x_{2\ell-1}\\ x_{2\ell}\end{array}\right) & \mapsto & \left(\begin{array}{c}\overline{x}_{2\ell}\\ -r.\overline{x}_{2\ell-1}\end{array}\right),  \ell=1,2,\ldots,n, \\
\left(\begin{array}{c}x_{0,-\ell}\\ x_{0,\ell}\end{array}\right) & \mapsto & \left(\begin{array}{c}\pm\sqrt{r}.\overline{x}_{0,\ell}\\ \mp\sqrt{r}.\overline{x}_{0,-\ell}\end{array}\right), \ell=1,2,\ldots,k\end{array}\right.$$ is point-domestic and fixed point free.

Now a property similar to that of rational quadrics explained above holds for complex rational Hermitian varieties, and we will not repeat it. 
\bigskip

\noindent\textbf{Finite fields.} The previous considerations allow us now to conclude that a finite polar space $\Delta$ of rank at least $3$ admits a fixed point free point-domestic collineation if and only if it is a symplectic polar space in odd characteristic, or it is a hyperbolic quadric, an elliptic quadric or a minimal Hermitian variety. Moreover, all such collineations can be explicitly written down on appropriate standard forms. 

\begin{remark}
The inverse procedure, namely, to start with a Hermitian form over a field extension $\LL/\K$ or over a quaternion division algebra $\HH$ with $2$-dimensional subfield $\K$, and consider it over the subfield $\K$ to obtain a quadratic form or a Hermitian form, respectively, is well known (sometimes called \emph{field reduction}). However, other situation turn up here as this is not always related to a collineation and its fixed point structure (e.g., when the $2$-dimensional subfield $
\K$ of the quaternion division algebra $\HH$ is an inseparable extension of the ground field of the algebra). But when it is related to a collineation and its fixed point structure, then it is always a point-domestic fixed point free collineation.   
\end{remark}

\section{Polar closed diagrams}\label{sec:polarclosed1}

In this section we prove Theorem~\ref{thm1:polarclosed}, classifying those admissible Dynkin diagrams which arise as the opposition diagram of a unipotent element of a split spherical building. The result for the exceptional types $\sE_n$ ($n=6,7,8$), $\sF_4$ and $\sG_2$ is \cite[Theorem~5]{PVM:21}, and so here we consider the classical types. We will also prove Theorem~\ref{thm1:attained}.
%
%Recall that the admissible Dynkin diagrams that are not polar closed are precisely the diagrams $\sB_{n;i}^1$ and $\sD_{n;i}^1$ (with $i$ odd and $1\leq i<n$) and $\sC_{n;i}^2$ (for $0\leq i\leq n/2$). 

We shall need the following lemma. A \emph{nondegenerate generalized polarity $\rho$} in a projective space $\PG(V)$ is a symmetric relation in the point set such that for every point $p$ the set $p^\rho$ of points in relation with $p$ is a hyperplane of $\PG(V)$.  We say that a polar space $\Delta$ is \emph{embedded in a nondegenerate generalized polarity $\rho$ of some projective space $\PG(V)$} if each point $p$ of $\Delta$ is a point of $\PG(V)$ and if the hyperplane spanned by $p^\perp$ coincides precisely with $p^\rho$. 

\begin{lemma}\label{geomsub}
Let $\Gamma$ be any polar space of rank $n$ embedded in a nondegenerate generalized polarity $\rho$ of a projective space $\PG(V)$ over $\LL$ such that $\Gamma$ spans $\PG(V)$. Let $0\leq i\leq n-1$. Then no geometric subspace of corank $i$ is contained in a subspace of codimension $i+1$. 
\end{lemma}

\begin{proof}
Suppose for a contradiction that $S$ is a subspace of $\PG(V)$ of codimension $i+1$ containing a geometric subspace of $\Gamma$ of corank $i$. Since $\Gamma$ is embedded in $\rho$, every point $x$ of $\Gamma$ outside $S^\rho$ has the property that $x^\perp\cap S$ is a subspace of codimension $i+2$. Let $x_1$ be such a point (which certainly exists as $\Gamma$ is not contained in $S^\rho$ by assumption).  In $\Gamma_1:=\Res(x_1)$ this means we obtain a subspace $S_1$ with $\codim S_1=i+2$ intersecting every singular subspace of dimension $(i-1)$ nontrivially. Since $\Res(x_1)$ is embedded in a nondegenerate polarity of $\PG(V_2)$, with $\codim V_2=2$, and spans $\PG(V_2)$, we can apply the same argument and continue like this until we reach a polar space $\Gamma_i$ of rank $n-i$ spanning $\PG(V_{2i})$, with $\codim V_{2i}=2i$, and a subspace $S_i$ of codimension $2i+1$ in $\PG(V)$, hence codimension $1$ in $\PG(V_{2i})$,  intersecting every subspace of dimension $0$ of $\Gamma_i$ nontrivially. In other words, $\Gamma_i$ is contained in $S_i$, a contradiction. 
\end{proof}

%NOTE: try to formulate it with codimensions instead of dimensions to allow infinte dimensions.

We can now show that the opposition diagram $\mathsf{B}_{n;i}^1$ for odd $i$ does not occur as the opposition diagram of any unipotent element in a split building of type $\sB_n$. In fact, we can prove a slightly more general result by considering all proper orthogonal polar spaces. Recall that a \emph{proper orthogonal} polar space is a polar space for which the minimal embedding has no secants of size at least 3. These include all the quadrics of Witt index at least 2 in characteristic different from 2, and all hyperbolic ones.

\begin{prop}\label{prop:PolarClosed_a}
Let $\Delta$ be a proper orthogonal polar space standard embedded in some projective space $\PG(V)$ and let $\theta$ be a collineation of $\Delta$ with opposition diagram $\mathsf{B}_{n;i}^1$  for some odd $i$, $1\leq i\leq n-1$. Then $\theta$ is not a unipotent element.
\end{prop}

\begin{proof}
%If $\theta$ has opposition diagram $\sD_{n;i}^1$ with $i$ odd then necessarily $\theta$ interchanges types $n-1$ and $n$ of the building, and hence is not a unipotent element. The argument below for the $\sB_{n;i}^1$ case also applies to $\sD_{n;i}^1$, giving another proof in this case. 
%
Suppose for a contradiction that $\theta$ is a unipotent element of $\Delta$ with opposition diagram $\mathsf{B}_{n;i}^1$ for some odd $i$, $1\leq i\leq n-1$. By \cite[Theorem~6.1]{TTM:12} and Lemma~\ref{lem:typeI}, $\theta$ pointwise fixes a geometric subspace of corank $i$, which, by Lemma~\ref{geomsub} generates a subspace $S$ of $\PG(V)$ of codimension at most $i$. Since $\theta$ is unipotent, the set of fixed points of $\theta$ is a subspace $T$, and we claim it is $S$. Indeed, since $S$ is spanned by the pointwise fixed subspace of corank $i$, it is pointwise fixed (this also follows from the proof of Theorem 6.1 in \cite{TTM:12}). Suppose now $\codim T< i$. Then a random singular subspace of dimension $i-1$ intersects $T$ in at least a point and thus cannot be mapped onto an opposite, a contradiction. So $S=T$ and $\codim S=i$. 

Since $S$ is an axis for $\theta$, the subspace $S^\perp$ is a centre for $\theta$, that is, all subspaces of $\PG(V)$ containing $S^\theta$ are fixed. Since unipotent elements have incident pairs (axis,center), we have $S^\perp\subseteq S$. This now implies that $S^\perp$ is a singular subspace of $\Delta$ with dimension $i-1$. Since it acts as the centre of $\theta$, the latter maps an $(i-1)$-dimensional singular subspace $W$ of $\Delta$ disjoint from $S$ onto a disjoint singular subspace of dimension $i-1$ contained in $\<S,W\>$. Hence $W$ and $W^\theta$ are two disjoint maximal singular subspaces of the same type of the hyperbolic quadric $\<S,W\>\cap\Delta$ of rank $i$, implying that $i$ is even, a contradiction.
\end{proof}

We can now prove one direction of Theorem~\ref{thm1:polarclosed}. 

\begin{cor}\label{cor:unip}
Let $\theta$ be an automorphism of a split spherical building. If the opposition diagram of $\theta$ is not polar closed, then $\theta$ is not a unipotent element. 
\end{cor}

\begin{proof}
Recall that all admissible diagrams of classical type are polar closed except for those of type $\sA_{n;(n-1)/2}^2$ ($n$ odd), $\sA_{n;n}^1$, $\sB_{n;i}^1$ or $\sD_{n;i}^1$ (with $i$ odd and $1\leq i<n$), $\sC_{n;i}^2$ (with $1\leq i\leq n/2$), and those in Table~\ref{table:2}. The diagrams $\sA_{n;(n-1)/2}^2$ ($n$ odd), $\sA_{n;n}^1$, $\sB_{n;i}^1$ (with $i$ odd and $1\leq i<n$), and those in Table~\ref{table:2} all correspond to non-type preserving automorphisms of $\Delta$, and hence cannot arise as the opposition diagram of a conjugate of $U$. Proposition~\ref{prop:PolarClosed_a} and Theorem~\ref{symplectichomology} show that the remaining non-polar closed diagrams $\sB_{n;i}^1$ ($i$ odd) and $\sC_{n;i}^2$ cannot be obtained as the opposition diagrams of unipotent elements. 
\end{proof}

We must now prove the converse. To do so, we briefly recall explicit matrix realisations of the elements $x_{\alpha}(a)$ for split classical groups, following the conventions in~\cite{Car:89} (note that this leads to different quadratic forms than those listed in Section~\ref{sec:polarspace}). In each case we will give matrices for $x_{\alpha}(a)$ with $\alpha\in\Phi^+$, and the corresponding elements $x_{-\alpha}(a)$ are obtained by taking the transpose. 

For type $\sA_n$ the root system has simple roots $\alpha_i=e_i-e_{i+1}$ for $1\leq i\leq n$, and the highest root is $\varphi=e_1-e_{n+1}$. The Chevalley generators are (for $a\in\KK$)
$
x_{e_i-e_j}(a)=1+aE_{i,j}
$ for $1\leq i<j\leq n+1$.

For type $\sB_n$ the root system has simple roots $\alpha_i=e_i-e_{i+1}$ for $1\leq i\leq n-1$, and $\alpha_n=e_n$, and the highest root is $\varphi=e_1+e_2$. Let $(\cdot,\cdot)$ be the symmetric bilinear form defined on $\KK^{2n+1}$ by
$$
(x,y)=2x_0y_0+x_1y_{n+1}+\cdots+x_ny_{2n}+x_{n+1}y_1+\cdots+x_{2n}y_n
$$
(indexing of rows and columns of vectors and matrices begins at $0$ here). Then $\Delta=\sB_n(\KK)$ can be realised as the polar space associated to the quadric $f(x)=(x,x)$. The Chevalley generators are (see \cite[Section 11.2]{Car:89})
\begin{align*}
x_{e_i-e_j}(a)&=1+a(E_{i,j}-E_{n+j,n+i})\\
x_{e_i+e_j}(a)&=1+a(E_{i,n+j}-E_{j,n+i})\\
x_{e_k}(a)&=1+2aE_{k,0}-aE_{0,n+k}-a^2E_{k,n+k}.
\end{align*}

For type $\sC_n$ the root system has simple roots $\alpha_i=e_i-e_{i+1}$ for $1\leq i\leq n-1$, and $\alpha_n=2e_n$, and the highest root is $\varphi=2e_1$.  Let $(\cdot,\cdot)$ be the symplectic form on $\KK^{2n}$ given by
$$
(x,y)=x_1y_{n+1}+\cdots+x_ny_{2n}-x_{n+1}y_1-\cdots-x_{2n}y_n.
$$
We write the standard basis of $\KK^{2n}$ as $\varepsilon_1,\ldots,\varepsilon_n,f_1,\ldots,f_n$, so that $(\varepsilon_i,\varepsilon_j)=(f_i,f_j)=0$ for all $1\leq i,j\leq n$ and $(\varepsilon_i,f_j)=\delta_{i,j}$. The Chevalley generators are
\begin{align*}
x_{e_i-e_j}(a)&=1+a(E_{i,j}-E_{n+j,n+i})\\
x_{e_i+e_j}(a)&=1+a(E_{i,n+j}+E_{j,n+i})\\
x_{2e_k}(a)&=1+aE_{k,n+k}.
\end{align*}

For type $\sD_n$ ($n\geq 4$) the root system has simple roots $\alpha_i=e_i-e_{i+1}$ for $1\leq i\leq n-1$, and $\alpha_n=e_{n-1}+e_n$, and the highest root is $\varphi=e_1+e_2$. Let $f(x)$ be the quadratic form 
$$
f(x)=x_1x_{n+1}+\cdots+x_nx_{2n}
$$
on $\KK^{2n}$, and let $(x,y)$ be the associated symmetric bilinear form 
$
(x,y)=f(x+y)-f(x)-f(y)
$. The Chevalley generators are
\begin{align*}
x_{e_i-e_j}(a)&=I+a(E_{i,j}-E_{n+j,n+i})\\
x_{e_i+e_j}(a)&=I+a(E_{i,n+j}-E_{j,n+i}).
\end{align*}
Moreover, the diagram automorphism (interchanging nodes $n-1$ and $n$) has matrix
$$
\sigma=1-E_{n,n}-E_{2n,2n}+E_{n,2n}+E_{2n,n}.
$$

The following proposition proves the converse of Theorem~\ref{thm1:polarclosed}. Recall the definition of $U(\sX)$ from Section~\ref{sec:polarclosed}.

\begin{prop}\label{prop:PolarClosed_b}
Let $\sX=(\Gamma,J,\pi)$ be a polar closed admissible diagram of classical type. Every generic element of $U(\sX)$ has opposition diagram~$\sX$.
\end{prop}

\begin{proof}
Let $\theta$ be a generic element of $U(\sX)$.  By \cite[Lemmas~1.1 and~3.5]{PVM:21} the element $\theta$ maps the chamber $w_0B$ to Weyl distance $w_{S\backslash J}w_0$, and hence the type $J$-simplex of this chamber is mapped onto an opposite simplex by~$\theta$. Thus $J\subseteq \Type(\theta)$. It remains to prove the reverse inclusion. We consider each nonempty polar closed diagram~$\sX$, specifically, ${^2}\sA_{n;i}^1$ ($0\leq i\leq n/2$), $\sB_{n;i}^1$ (with $i$ even and $0\leq i<n$), $\sB_{n;n}^1$, $\sB_{n;i}^2$ ($0\leq i\leq n/2$), $\sC_{n;i}^1$ (for $0\leq i\leq n/2$), $\sD_{n;i}^1$ (with $i$ even), and $\sD_{n,i}^2$ (with $0\leq i\leq n/2$). 

In this proof (and only in this proof) ``dimension'' refers to vector space dimension (rather than projective dimension). Moreover, $\{i\}$-domestic means domestic on type $i$ vertices of the building.
\medskip

\noindent (1) $\sX=\sA_{n;i}^1$ with $1\leq i\leq n/2$. It is sufficient to show that $\theta$ is $\{j,n-j+1\}$-domestic for each $j>i$. The sequence of highest roots is $\varphi_k=e_k-e_{n-k+2}$ for $1\leq k\leq i$. Thus 
$$
\theta=(I+a_1E_{1,n+1})(I+a_2E_{2,n})\cdots(I+a_iE_{i,n-i+2})=I+a_1E_{1,n+1}+\cdots+a_iE_{i,n-i+2},
$$
with $a_1,\ldots,a_i\neq 0$. Since $\theta$ pointwise fixes the $(n-i+1)$-dimensional space spanned by the vectors $e_k$ with $1\leq k\leq n-i+1$ it follows that each space of dimension $j>i$ contains a fixed point. Thus if $(V,V')$ is a type $\{j,n-j+1\}$-simplex  (that is, $\dim(V)=j$, $\dim(V')=n-j+1$, and $V\subseteq V'$) then $V$ contains a fixed point. Thus $V^{\theta}\cap V'\neq\{0\}$, and so $\theta$ is $\{j,n-j+1\}$-domestic, hence the result. 
\medskip

\noindent (2) $\sX=\sB_{n;2i}^1$ with $1\leq i\leq n/2$. It is sufficient to show that $\theta$ is $\{j\}$-domestic for $j>2i$ and is not $\{1\}$-domestic. The highest roots are $\varphi_k$ with $k=1,\ldots,2i$, where $\varphi_k=e_{2k-1}+e_{2k}$ and $\varphi_{i+k}=e_{2k-1}-e_{2k}$ for $1\leq k\leq i$. Thus $\theta=\prod_{k=1}^i x_{e_{2k-1}+e_{2k}}(a_k)x_{e_{2k-1}-e_{2k}}(b_k)$ with $a_1,\ldots,a_i,b_1,\ldots,b_i\neq 0$, and so
\begin{align*}
\theta&=I+\sum_{k=1}^i\big(a_k(E_{2k-1,n+2k}-E_{2k,n+2k-1})+b_k(E_{2k-1,2k}-E_{n+2k,n+2k-1})+a_kb_kE_{2k-1,n+2k-1}\big).
\end{align*}
The $(2n-2i+1)$-dimensional space spanned by the vectors $e_k$ with $k\notin\{2,4,\ldots,2i\}\cup\{n+1,n+2,\ldots,n+2i\}$ and the vectors $e_{n+2k}-a_kb_k^{-1}e_{2k}$ with $k=1,\ldots,i$ is fixed pointwise  by $\theta$, and so $\theta$ is $j$-domestic for all $j>2i$.  Let $x=e_1+e_2+e_{n+1}-e_{n+2}$. Then $(x,x)=0$ (and so $\langle x\rangle$ is a point of the polar space), and we compute $(x,\theta x)=-a_1b_1\neq 0$, and so this point is mapped onto an opposite point by $\theta$. 
\medskip

\noindent (3) $\sX=\sB_{n;n}^1$. The first paragraph of the proof shows that $\theta$ maps a chamber to an opposite chamber, hence the result.
\medskip

\noindent (4) $\sX=\sB_{n;i}^2$ with $1\leq i\leq n/2$. It is sufficient to show that $\theta$ is $\{1\}$-domestic, and $\{j\}$-domestic for all $j>2i$. The highest roots are $\varphi_k=e_{2k-1}+e_{2k}$ for $1\leq k\leq i$, and so
$$
\theta=\prod_{k=1}^i (I+a_k(E_{2k-1,n+2k}-E_{2k,n+2k-1}))=I+\sum_{k=1}^ia_k(E_{2k-1,n+2k}-E_{2k,n+2k-1}),
$$
with $a_1,\ldots,a_i\neq 0$. The $(2n-2i+1)$-dimensional space spanned by the vectors $e_k$ with $k\notin\{n+1,n+2,\ldots,n+2i\}$ is fixed pointwise by $\theta$, and it follows that every subspace of dimension $j>2i$ contains a fixed point. Thus $\theta$ is $\{j\}$-domestic for all $j>2i$.  Moreover, for each point $p=\langle x\rangle$ we compute $(x,\theta x)=(x,x)=0$, and so $\theta$ is $\{1\}$-domestic. 
\medskip

\noindent (5) $\sX=\sC_{n;i}^1$ with $1\leq i\leq n$. It is sufficient to show that $\theta$ is $\{j\}$-domestic for $j>i$ and that $\theta$ is not point domestic. The highest roots are $\varphi_k=2e_k$ for $1\leq k\leq i$. Thus 
$$
\theta=\prod_{k=1}^i(I+a_kE_{k,n+k})=I+\sum_{k=1}^i a_kE_{k,n+k}.
$$
Thus $\theta$ fixes the $(2n-i)$-dimensional space spanned by $e_k$ with $k\notin\{n+1,\ldots,n+i\}$. Thus $\theta$ is $\{j\}$-domestic for all $j>i$. Since $\theta e_{n+1}=a_1e_1+e_{n+1}$ we have $(e_{n+1},\theta e_{n+1})=-a_1\neq 0$, and so $\langle e_{n+1}\rangle$ is mapped onto an opposite point, hence the result.
\medskip

\noindent (6) $\sX=\sD_{n;2i}^1$ with $1\leq i\leq n/2$. The highest roots are as in the $\sB_{n;2i}^1$ case, and the argument is identical to~(2).
\medskip

\noindent (7) $\sX=\sD_{n;i}^2$ with $1\leq i\leq n/2$. The highest roots are as in the $\sB_{n;i}^2$ case, and the argument is identical to~(3).
\end{proof}

Finally we prove Theorem~\ref{thm1:attained}. 

\begin{proof}[Proof of Theorem~\ref{thm1:attained}]
The non-polar closed type preserving admissible diagrams are $\sB_{n;i}^1$ with $i$ odd and $1\leq i<n$ and $\sC_{n;i}^2$ for $1\leq i\leq n/2$. If the underlying field is of characteristic~$2$ then, by assumption, the field is perfect and one can apply the previous Proposition and duality. Thus suppose that characteristic is not~$2$. We will show that these diagrams can be obtained by homologies of the associated buildings.
\medskip

\noindent (1) $\sX=\sB_{n;i}^1$ with $i$ odd and $1\leq i<n$. More generally, we show that for all $1\leq i<n$ the homology
$$
\theta=\mathrm{diag}(\epsilon,\underbrace{-1,\ldots,-1}_{i'},\underbrace{1,\ldots,1}_{n-i'},\underbrace{-1,\ldots,-1}_{i'},\underbrace{1,\ldots,1}_{n-i'}),
$$
where $\epsilon=(-1)^i$ and $i'=\lfloor i/2\rfloor$, has opposition diagram $\sB_{n;i}^1$. Note that $\theta$ preserves the form $(\cdot,\cdot)$ and hence $\theta\in \sB_n(\mathbb{F})$. Since $\theta$ fixes a subspace of dimension $2n-i+1$ pointwise, every subspace of dimension $j\geq i+1$ contains a fixed point, and so $\theta$ is $j$-domestic for each $j\geq i+1$. If $i$ is odd then the point $p=\langle e_0+e_n-e_{2n}\rangle$ is mapped onto an opposite point, and if $i$ is even then the point $p=\langle e_0+e_1-e_{n+1}\rangle$ is mapped onto an opposite point. Thus $\theta$ is not point domestic. We now show that $\theta$ is not $i$-domestic. Suppose that $i$ is even. Let $x_j=e_j+e_{2n-j+1}$ and $y_j=e_{n+j}-e_{n-j+1}$ for $1\leq j\leq i'=n/2$. Then $V=\langle x_1,\ldots,x_{i'},y_1,\ldots,y_{i'}\rangle$ is totally isotropic. If $V\ni p=\langle v\rangle$ with $v=a_1x_1+\cdots+a_{i'}x_{i'}+b_1y_1+\cdots+b_{i'}y_{i'}$ with $a_j\neq 0$ then $p$ is opposite $\theta\langle y_j\rangle\in \theta_iV$, because $(v,\theta y_j)=-2a_j$. Similarly, if $b_j\neq 0$ then $p$ is opposite $\theta\langle x_j\rangle\in \theta V$. Thus $V$ and $\theta V$ are opposite. Suppose now that $i$ is odd. There exists an integer $k$ with $i'<k<n-i'+1$, and thus both $e_k$ and $e_{n+k}$ are fixed by $\theta$. Let $x_1,\ldots,x_{i'}$ and $y_1,\ldots,y_{i'}$ be as above, and let $z=e_0+e_k-e_{n+k}$. Then $(z,z)=(z,x_j)=(z,y_j)=0$ for all $1\leq j\leq i'$, and thus $V=\langle z,x_1,\ldots,x_{i'},y_1,\ldots,y_{i'}\rangle$ is a totally isotropic space of dimension~$i$. If $p=\langle v\rangle$ with $v=cz+a_1x_1+\cdots+a_{i'}x_{i'}+b_1y_1+\cdots+b_{i'}y_{i'}$ with $c\neq 0$ then $p$ is opposite $\theta\langle z\rangle\in \theta V$. If $c=0$ then either $a_j\neq 0$ for some $j$ (in which case $p$ is opposite $\theta\langle y_j\rangle$) or $b_j\neq 0$ for some $j$ (in which case $p$ is opposite $\theta\langle x_j\rangle$). Thus $V$ is opposite $\theta V$, hence the result.

\medskip

\noindent (2) $\sX=\sC_{n;i}^2$ with $1\leq i\leq n/2$. Let $\theta$ be the $(2n-2i,2i)$-homology
$$
\theta=\mathrm{diag}(\underbrace{-1,\ldots,-1}_{i},\underbrace{1,\ldots,1}_{n-i},\underbrace{-1,\ldots,-1}_{i},\underbrace{1,\ldots,1}_{n-i}).
$$
Since $\theta$ fixes a space of dimension $2n-2i$ pointwise, $\theta$ is $j$-domestic for all $j\geq 2i+1$. Moreover, since $(x,\theta x)=0$ for all $x\in \mathbb{F}^{2n}$ we see that $\theta$ is point-domestic. Let $x_j=e_j+f_{n-j+1}$ and $y_j=e_{n-j+1}+f_j$ for $1\leq j\leq n/2$. Let $V=\langle x_1,\ldots,x_i,y_1,\ldots,y_i\rangle$ (note that $V$ is totally isotropic). We claim that $V$ and $\theta V$ are opposite. For if $p=\langle v\rangle$ with $v=a_1x_1+\cdots+a_ix_i+b_1y_1+\cdots+b_iy_i$ with $a_j\neq 0$ for some $j$ then $p$ and $\theta\langle y_j\rangle$ are opposite, and if $b_j\neq 0$ for some $j$ then $p$ and $\theta\langle x_j\rangle$ are opposite. Hence $\theta$ has opposition diagram $\sC_{n,i}^2$.
\smallskip

We now consider the non-type preserving diagrams. 
\medskip

\noindent (3) $\sX=\sA_{n;(n-1)/2}^2$ with $n\geq 3$ odd. Let $(\cdot,\cdot)$ be any nondegenerate symplectic form on $\mathbb{F}^{n+1}$ (note that $n+1$ is even), and for subspaces $U$ write $U^{\circ}=\{v\in\mathbb{F}^{n+1}\mid (u,v)=0\text{ for all $u\in U$}\}$. Then the map with $U^{\theta}=U^{\circ}$ is a symplectic polarity, and has the required diagram.
\medskip

\noindent (4) $\sX={^2}\sB_{2;1}^1$ or $\sX={^2}\sC_{2;1}^1$, that is, dualities of buildings of types $\sB_2$ and $\sC_2$ for perfect fields of characteristic~$2$. Every duality has this diagram.
\medskip

\noindent (5) $\sX=\sD_{n;2i+1}^1$ with $0\leq i<(n-1)/2$. Let $0\leq i\leq n/2$, and let
\begin{align*}
\theta_i&=I+\sum_{j=1}^{i}(E_{2j-1,2j}-E_{n+2j,n+2j-1})\\
\eta&=I+E_{n-1,n}+E_{n-1,2n}-E_{n-1,2n-1}-E_{n,2n-1}-E_{2n,2n-1}\\
\sigma&=I-E_{n,n}-E_{2n,2n}+E_{n,2n}+E_{2n,n},
\end{align*}
with $\theta_0=I$. Similar calculations to the above show that for $0\leq i\leq n/2$, the collineation $\theta_i$ has opposition diagram $\sD_{n,i}^{2}$, the collineation $\eta_i=\theta_i\eta$ has opposition diagram $\sD_{n,2i}^1$, and the duality $\sigma_i=\theta_i\sigma$ has opposition diagram $\sD_{n,2i+1}^1$.
\medskip

Finally, trialities of buildings of type  $\sD_4$ are treated in~\cite{HVM:13}, and the proof is complete. 
\end{proof}

\bibliographystyle{plain}

%
%\bibliographystyle{plain}
%\bibliography{/Users/jamesp/Dropbox/Research/bibtex/Parkinson}
%\bibliography{/Users/jwparkinson/Dropbox/Research/bibtex/Parkinson}

\end{document}